\documentclass[11pt]{book}
\usepackage{amssymb,amsmath,amsthm}
\usepackage{stmaryrd}
\usepackage[mathscr]{eucal}
\usepackage{geometry}
\usepackage[pdftex,
                        a4,
                        center
                        ]{crop}
\newcommand{\sfrac}[2]{\text{\footnotesize $\frac{#1}{#2}$}}
\newcommand{\msfrac}[2]{\text{\fontsize{10}{10}\selectfont $\frac{#1}{#2}$}}
\newcommand{\vsfrac}[2]{{\textstyle \frac{#1}{#2}}}
\newcommand{\bx}{$\blacksquare$}

\newcommand{\Fh}{\mathfrak{h}}
\newcommand{\TS}{\textstyle}
\newcommand{\MSh}{\mspace{0.25mu}}
\newcommand{\MSH}{\mspace{0.50mu}}
\newcommand{\MSA}{\mspace{1mu}}
\newcommand{\MSAH}{\mspace{1.5mu}}
\newcommand{\MSB}{\mspace{2mu}}
\newcommand{\MSC}{\mspace{3mu}}
\newcommand{\MSD}{\mspace{4mu}}
\newcommand{\MSE}{\mspace{5mu}}
\newcommand{\MShZ}{\mspace{-0.25mu}}
\newcommand{\MSHZ}{\mspace{-0.5mu}}
\newcommand{\MSZ}{\mspace{-1mu}}
\DeclareMathOperator{\Ree}{Re}
\DeclareMathOperator{\Imm}{Im}
\DeclareMathOperator{\Arg}{Arg}
\DeclareMathOperator{\Log}{Log}
\DeclareMathOperator{\sgn}{sgn}
\DeclareMathOperator{\Rank}{rank}

\DeclareMathOperator{\ctn}{ctn}
\DeclareMathOperator{\supp}{supp}
\DeclareMathOperator{\card}{card}
\usepackage{titlesec,titletoc}

\titleformat{\chapter}[display]
{\normalfont\fontsize{16.4}{19.8}\selectfont\bfseries}{\centering \tiny}{0.2em}{\centering}
\renewcommand{\thesection}{\arabic{section}}
\titleformat{\section}[hang]%
{\normalfont\large\bfseries}{\centering \S \thesection.}{.14em}{\vspace{1.26ex} \centering}
\renewcommand{\thesubsection}{\arabic{section}.\arabic{subsection}.}
\titleformat{\subsection}[runin]%
{\normalfont\bfseries}{\thesubsection}{0.2em}{}
\titleformat{\paragraph}[runin]%
{\normalfont\bfseries}{}{1em}{}
\titlecontents{section}[1.5em]{\normalfont}{}{}{\titlerule*[0.80pc]{.}\contentspage[\normalfont\thecontentspage]}[]
\titlecontents{subsection}[2.5em]{\normalfont\itshape}{}{}{\titlerule*[0.80pc]{.}\contentspage[\normalfont\thecontentspage]}[]
\titlecontents{subsubsection}[3.4em]{\normalfont\itshape}{}{}{\titlerule*[0.80pc]{.}\contentspage[\normalfont\thecontentspage]}[]
\usepackage{fancyhdr}
\newtheoremstyle{dah-thm}
{2.52ex}
{2.52ex}
{\itshape}
{}
{\bfseries}
{}
{0.5em}
{}
\newtheoremstyle{dah-rem}
{2.52ex}
{2.52ex}
{}
{}
{\bfseries}
{}
{0.5em}
{}

\theoremstyle{dah-thm}
\newtheorem{theorem}[subsection]{Theorem}
\newtheorem{lemma}[subsection]{Lemma}

\newtheorem{corollary}[subsection]{Corollary}
\theoremstyle{plain}
\newtheorem*{thm48bisA}{Theorem 4.8$\MSH$A}
\newtheorem*{thm48bisB}{Theorem 4.8$\MSH$B}
\newtheorem*{thm48bisC}{Theorem 4.8$\MSH$C}
\theoremstyle{dah-rem}
\newtheorem{remark}[subsection]{Remark}

\begin{document}

\begin{frontmatter}
\end{frontmatter}

\begin{mainmatter}
\chapter[$\hspace{7.59em}$Gaussians, Zeros, and Linear 
Combinations$\hspace{4.17em}$\textsl{[Part A]}] 
{ON $\MSB$GAUSSIANS, $\MSB$ZEROS, $\MSA$AND $\MSB$LINEAR \\
COMBINATIONS $\MSB$OF $\MSB$L$\MSB$-$\MSB$FUNCTIONS\\ \vspace*{1mm}
\large Part A: $\MSH$Entire Functions of Beurling-Selberg Type\\
and Probability Distributions\\  
\vspace*{1mm} \Large \textup{D.~A. HEJHAL\footnote{Part of the work 
in this paper was supported by a Bell Companies Fellowship at the 
Institute for Advanced Study in Princeton, N.J.}}    }

This is the first of (what is projected to be) a short series of 
technical reports aimed at articulating certain connections 
between the broad 
subject headings of Gaussians, zeros, and $\MSA$linear combinations 
of $L\MSA$-$\MSA$functions$\MSA$ within the setting of 
analytic number theory.  

In later 
installments, we shall find ourselves 
particularly interested 
in developing some $\MSA$techniques$\MSA$ by means of which it becomes 
feasible to say something fairly precise (at least 
near $\Ree(s) =\vsfrac{1}{2}$) about both the 
asymptotic value distribution and the 
distribution of zeros manifested by run-of-the-mill 
linear combinations $\MSAH F(s)\MSAH$ 
of the aforementioned sort. $\MSC$A large portion of our 
results will have underpinnings 
that go back to work of Atle Selberg from the 1940s and 
mid$\MSA$-$\MSHZ$'70s. $\MSB$See [Sel] for an updated 
perspective on part of this; $\MSC$also [BH]. 

To set the stage for these things, it $\MShZ$is 
$\MShZ$convenient $\MShZ$to $\MShZ$begin $\MShZ$with 
$\MShZ$several $\MSAH\MSh$more basic$\MSA\MSH$ 
reports devoted to discussing some background material  
that $\MShZ$will$\MSh$ for the most part be seen to either 
be standard $\MSH$-$\MSH$ or else a natural \emph{variant} of  
something that is. $\MSD$Apart from two or three items of 
the latter type, any novelty in these earlier installments 
will be confined solely to methodological matters.

\section{Some Interpolation Formulae and Related Heuristics}
\subsection{} The traditional reference for entire functions of 
Beurling-Selberg type is [Beu] coupled with [Sel2, 
pp.$\MSB$213--218,$\MSE\MSAH$226]. $\MSZ$In very 
loose terms, $\MSAH$the central issue here 
$\MSAH$function-theoretically$\MSA$ 
is as follows: $\MSE\MSA$Consider 
the real line ${\mathbb R}$.  \,Let $\chi_E(x)$ denote
the indicator function of set $E$ and
\renewcommand{\theequation}{1.\arabic{equation}}
\begin{equation}\label{1.1}
\sgn(x)=\left \{ \begin{array}{ll}
0, & x=0\\[1ex]
x/|x|, & x\neq 0
\end{array} \right \} . 
\end{equation}
Keeping $\ell \in {\mathbb Z}^+$, let $m(x)$ be either $\sgn(x)$ or 
$\chi_{[0,\ell]}(x)\MSB$.  \,How does one \emph{construct} entire functions
$\MSA f(z) \MSA$ of exponential type 
at most $2\pi$ such that, along ${\mathbb R}$, one has
\\[2ex]
(i)\phantom{ii}\quad  $f(x) \in {\mathbb R}$ and $ f(x) \geqq m(x)\MSB$;  
\\[2ex]
and
\\[2ex]
(ii)\phantom{i}\quad $ \lim_{\MSB|x|\rightarrow \infty}\, (f(x)-m(x))\,=\, 0 \MSD\MSD$?
\\[2ex]
In (ii),$\MSA$ the faster the vanishing, the better.  Can one rig things,
for instance, so that one also has
\\[2.5ex]
(iii)\quad $ \displaystyle \int^{\infty}_{-\infty}\,
|f(x)-m(x)|^q\,dx < \infty 
\mbox{\phantom{iii}for every $q \in [1,\infty)$}$$\MSD$?
\\[1.6ex]
An $L_1\MSA$-$\MSA$norm of (something close to) minimal size would 
be especially desirable.

In \S 1, $\MSA$our goal will simply be to \emph{informally} identify  
some natural candidates for $f(z)$. 
\,Substantiation of 
any proposed $f$ will be left for \S 2.$\MSA$\footnote{Note 
that the analogous constructions for $f \MSZ\leqq\MSZ m$ and non-integral 
$\MSA\ell\MSA$ are also treated there.}

\subsection{} To get started,   
we need a few basics. We'll 
assume the standard $L_2$ theory of Fourier transforms 
on ${\mathbb R}$ as being known. The prototypical formulae
\begin{align*}
&\hat{f}(p)=\int^\infty_{-\infty} f(x)e^{-2\pi ipx}dx\\
& f(x)=\int^\infty_{-\infty} \hat{f}(p)e^{2\pi ipx}dp\\
& \int^\infty_{-\infty}|f(x)|^2dx =
 \int^\infty_{-\infty}|\hat{f}(p)|^2dp
\vspace{2.52ex} \\
& \langle f,g \rangle\;\, =\, \;\langle \hat{f}, \hat{g} \rangle\\
& (f')^\wedge \;=\;2\pi ip\hat{f}(p)\\
& (f\ast g)^\wedge \;=\;\hat{f}\hat{g}
\end{align*}
are thus available as need be (in, for instance, Schwartz 
space ${\mathcal S}$). Over and beyond this, one readily 
checks that
\[
\, f(x)= g(\lambda x) \;\Rightarrow\; \hat{f}(p)=
\frac{1}{|\lambda|}\hat{g}(p/\lambda)
\]
for $\lambda \neq 0$ and that
\begin{align*}
f(x) = e^{-\pi ax^2} \;& \Rightarrow \; \hat{f}(p)
=\sqrt{\frac{1}{a}}\,\, e^{-\pi p^2/a}\\
f(x) = \max\{0, b-|x|\} \;& \Rightarrow \; \hat{f}(p)
 = \frac{\sin^2(\pi pb)}{\pi^2p^2}\\
f(x)= \chi_{[-c,c]}(x) \;& \Rightarrow \; \hat{f}(p)
 =  \frac{\sin(2\pi pc)}{\pi p}\\
f(x)=\sgn(x)\mspace{1mu}\chi_{[-c,c]}(x) \;& \Rightarrow \;
  \hat{f}(p) = -2\pi i p \frac{\sin^2(\pi pc)}{\pi^2p^2}\,\, .
\end{align*}

For test  
functions $\varphi$ in $C({\mathbb R}) \cap L_1({\mathbb R})$
whose ``periodization''\vspace*{-0.35ex}
\[
\, \Phi(x) \equiv \sum^{\infty}_{k=-\infty} \varphi (x+k)\vspace*{-0.20ex}
\]
converges \emph{uniformly} on $[0,1]$, one knows that the Poisson
summation formula 
\begin{equation}\label{1.2}
\sum^{\infty}_{k=-\infty} \varphi (k) = 
\sum^{\infty}_{m=-\infty} \hat{\varphi}(m) 
\end{equation}
holds anytime the right-hand 
series is convergent.  
Cf. $\MSZ$[Z1, pp.68, $\MSZ$89(3.4)]; $\MSA$also 
[Gau, $\MSZ$pp.$\MSA$88--89].   
$\MSAH$In particular, $\MSA$matters are readily seen to 
hold with absolute convergence throughout anytime         
$\varphi$ is $\MSA C^2$ and $\varphi^{(j)} \in L_1({\mathbb R})$ 
for $0 \leqq j \leqq 2$.$\MSA$\footnote{At 
the other extreme, see [Kat, pp.130$\MSB$(problem 15),~12(2.7),~125(1.10)]
for a ``mass {\linebreak} redistribution style'' construction of a continuous probability
density $g(t)$ on ${\mathbb R}$ whose Fourier transform $\varphi$ belongs to
$C({\mathbb R}) \cap L_1({\mathbb R})$ but, for which, 
$\varphi(n)=\delta_{n0}$ and $g(n)=0$ hold at every $n\in {\mathbb Z}$. 
Equation (1.2) clearly \emph{fails} for the pair $(\varphi , g)$. 
By either [Z1, p.12(6.2)] or [Kat, p.13], the periodizations of $\varphi$ 
and $g$ must sum to 0 and 1, respectively, almost everywhere.  Compare  
[KL, p.306] and (overly-optimistic) problem 12 in [Fel2, p.648].}

\subsection{} Another result that's very familiar, albeit 
on a much deeper level, is \linebreak the Paley-Wiener theorem, 
which characterizes $L_2({\mathbb R})$ functions obtained  
as restrictions of entire functions of exponential type.  
In this regard, cf. [Boa, p.103 (theorem 6.8.1)],  
[PaW, p.13 (note the typos in 6.10)], and [Z2, p.272 (theorem 7.2)].  
\,When the type is at most $2\pi \alpha$, the representation
\begin{equation}\label{1.3}
f(z)=\int^\alpha_{-\alpha} g(v)e^{2\pi ivz}dv
\end{equation}
holds with a unique $g\in L_2[-\alpha,\alpha]$ which, 
in turn, satisfies
\[
\int^\alpha_{-\alpha}|g(v)|^2dv = \int^\infty_{-\infty}|f(x)|^2dx\,.
\]
A simple application of Leibnitz's rule then gives
\begin{equation}\label{1.4}
f^{(m)}(z) = \int^\alpha_{-\alpha} (2\pi iv)^m g(v)e^{2\pi ivz}dv
\end{equation}
for every $m\geqq 1$ (the $L_2$ norm of $f^{(m)}$
 necessarily being finite in each case).

By applying Parseval's equation to the
{\it Fourier series}
of $(2\pi iv)^mg(v)$ on $[-\alpha,\alpha]$,
one immediately sees that:
\begin{align}
 \int^\alpha_{-\alpha}|g(v)|^2dv \MSA&=\MSA \frac{1}{2\alpha}
 \sum^\infty_{k=-\infty} \Big |f\Big (\frac{k}{2\alpha}\Big )
\Big|^2 = \int^\infty_{-\infty}|f(x)|^2dx\MSC;\label{1.5}\\[0.35ex]
 \frac{1}{2\alpha} \sum^{\infty}_{k=-\infty} \Big |f^{(m)}
 \Big (\frac{k}{2\alpha}\Big )\Big |^2 &= \int^\infty_{-\infty}
 |f^{(m)}(x)|^2dx\MSC\leqq\MSB(2\pi \alpha)^{2m} \int^\infty_{-\infty}
 |f(x)|^ 2dx\MSC.\vspace*{0.35ex}\label{1.6}
\end{align}
Relations (1.5) and (1.6) continue to hold  
when $k$ is replaced throughout by 
$k+\Theta \MSE(0\leqq \Theta < 1)$.  At the same time, it pays
to keep in mind that, by (1.4),
\begin{equation}\label{1.7}
|f^{(m)}(x)| \leqq \sqrt{2\alpha}
\MSA\frac{\MSD(2\pi \alpha)^m}{\sqrt{1+2m}}\MSA\|f\|_2 
\end{equation}
holds for every $\MSB x \in {\mathbb R} \MSB$ and $m \geqq 0 $.

\subsection{} To develop an \emph{interpolation formula} for $f$, 
one simply returns to (1.3) \linebreak and substitutes the  
standard $L_2$ Fourier expansion
\vspace*{-1.5ex}
\[
\sum^\infty_{n=-\infty} c_ne^{\pi inv/\alpha}
\vspace*{-1.5ex}
\]
of $g(v)$. \,Since 
$c_n = \frac{1}{2\alpha} f(\frac{-n}{2\alpha})$, 
a quick calculation yields
\begin{equation}\label{1.8}
f(z) = \frac{\sin(2\pi \alpha z)}{2\pi \alpha}
 \sum^\infty_{k=-\infty} \frac{(-1)^kf(k/2\alpha)}{z-k/2\alpha}
\end{equation}
just as in [Boa, p.220 (11.5.8)].  Cf. also here 
[Z2, p.275 (theorem 7.19)] and [Lev, p.150 (theorem 1)]. 
Expansion (1.8) is often referred to 
as the ``cardinal series'' or sampling theorem.

By passing to $(f(z)-f(0))/z$,\footnote{which again has 
exponential type at most $2\pi\alpha$} one 
can now accommodate milder growth\linebreak 
restrictions on $|f(x)|$. References [Boa]  
and [Z2] ({\it loc.$\MSC$cit.}) provide an immediate
elaboration on this point; \,one finds with very little 
effort that\\[-1.5ex]
\begin{align}
f(z) = \; & f'(0) \frac{\sin(2\pi\alpha z)}{2\pi\alpha} 
+ \frac{\sin (2\pi\alpha z)}{2\pi\alpha}
 \Big \{\frac{f(0)}{z}\Big \}\\[1.77ex]\label{1.9}
& + \, \frac{\sin(2\pi\alpha z)}{2\pi\alpha}
 \sum_{k\neq 0} (-1)^k f\Big (\frac{k}{2\alpha}\Big )
 \Big[ \frac{1}{z-k/2\alpha}+ \frac{1}{k/2\alpha}
\Big ] \nonumber
\end{align}
\\[-1ex]
anytime $|f(x)|/(1+|x|)\in L_2({\mathbb R})$. The classical 
relation
\begin{equation}\label{1.10}
\frac{\pi}{\sin\pi w} = \frac{1}{w} + \sum_{k\neq 0}
 (-1)^k \Big (\frac{1}{w-k}+\frac{1}{k}\Big )
\end{equation}
turns out to be exactly what is needed to handle  
the terms arising from $f(0)$ in $(f(z)-f(0))/z$.
(Notice that (1.10) is just (1.9) with $f \equiv 1$.)

Formula (1.9) has antecedents going back many years. 
In chronological order, see [Po1, \S5], [Val, especially 
p.204 (9)(11) with $\mu=1$], [PSz, problem III.165], 
[Har, theorem 10], and 
[Hig1,  
\S1.4 $\MSZ\MSZ$(minus the typo in (10))$\MSH$]  
for an interesting historical perspective on this. \,Note
that (1.8) follows from (1.9) simply by applying the 
latter to the product function $zf(z)\MSA$.

\subsection{}  
Looked at geometrically, the set of nodes in formula
(1.8) can be said to be $\vsfrac{1}{2\alpha}\MSA{\mathbb Z}$ in 
an obvious sense.  If \emph{additional} information in 
the form of values of $f'$ happens to be available at the 
nodes, an elementary reshuffling in (1.3) can be 
used to ``cut the nodal set down'' to just 
$\MSB\vsfrac{1}{\alpha}{\mathbb Z}\MSB$.\vspace*{0.375ex}

To see this, we follow Vaaler  
[Vaa, p.195 (3.5)(3.6)]. \vspace*{0.125ex}$\MSA$(Compare 
[Sel2, p.214 (lines 18--19)].)  Clearly:
\small
\begin{align*}
f(z) \,& =\,  \int^\alpha_0 g(v)e^{2\pi ivz}dv 
+ e^{-2\pi i\alpha z}\int^\alpha_0g(t-\alpha)e^{2\pi itz}dt
\\[1.26ex]
f\Big (\frac{k}{\alpha}\Big )\, & = \, \int^\alpha_0
 [g(v)+g(v-\alpha)]e^{2\pi ivk/\alpha}dv\MSD;
\\[1.26ex]
\frac{1}{2\pi i}f'(z) \, & = \, \int^\alpha_0 vg(v)e^{2\pi
 ivz}dv + e^{-2\pi i \alpha z} \int^\alpha_0 
(t-\alpha)g(t-\alpha)e^{2\pi itz}dt
\\[1.26ex]
\frac{1}{2\pi i}f' \Big (\frac{k}{\alpha}\Big )\,
 & = \, \int^\alpha_0[vg(v)+(v-\alpha)
g(v-\alpha)]e^{2\pi ivk/\alpha}dv \MSD.
\end{align*}
\normalsize
Accordingly, as $L_2$ Fourier series on $[0, \alpha]$,
\small
\begin{align*}
g(v) + g(v-\alpha) \,& \sim \, \sum_k \frac{1}{\alpha}f\Big
 (\frac{k}{\alpha}\Big ) e^{-2\pi ikv/\alpha}\\
vg(v)+(v-\alpha)g(v-\alpha)\, & \sim \, \sum_k \frac{1}{2\pi
 i\alpha}f' \Big (\frac{k}{\alpha}\Big )e^{-2\pi ikv/\alpha}\,.
\end{align*}
\normalsize
Denoting the left-hand expressions by $G_1$ and $G_2$, 
respectively, we can now write
\small
\begin{align*}
g(v) \,& = \, \Big (1-\frac{v}{\alpha}\Big ) G_1 
+ \frac{1}{\alpha} G_2\\
g(v-\alpha) \,& = \, \frac{v}{\alpha}G_1 -\frac{1}{\alpha} G_2
\end{align*}
\normalsize
and, in this way, obtain
\small
\begin{align*}
& g(v)\, \sim \,\Big (1-\frac{v}{\alpha}\Big ) \sum_k 
\frac{1}{\alpha} f \Big (\frac{k}{\alpha}\Big )
e^{-2\pi i kv/\alpha} + \frac{1}{\alpha} \sum_k 
\frac{1}{2\pi i\alpha}f'\Big (\frac{k}{\alpha}\Big )
 e^{-2\pi i kv/\alpha}
\\[1.26ex]
& g(v-\alpha)\, \sim \,\frac{v}{\alpha} \sum_k
 \frac{1}{\alpha} f \Big (\frac{k}{\alpha}\Big )
 e^{-2\pi ikv/\alpha} - \frac{1}{\alpha} \sum_k 
\frac{1}{2\pi i\alpha} f' \Big (\frac{k}{\alpha}
\Big ) e^{-2\pi ikv/\alpha}
\end{align*}
\normalsize
in a natural $L_2$ sense over $[0,\alpha]$.   
$\MSB$By substituting back, we immediately get
\small
\begin{align*}
f(z) \, = \,& \sum_k \frac{1}{\alpha} f\Big 
(\frac{k}{\alpha}\Big ) \int^\alpha_0 \Big
 (1-\frac{v}{\alpha}\Big ) e^{-2\pi ivk/\alpha}e^{2\pi ivz}dv\\
& + \,  e^{-2\pi i\alpha z} \sum_k \frac{1}
{\alpha} f\Big (\frac{k}{\alpha}\Big ) \int^\alpha_0
 \frac{t}{\alpha} e^{-2\pi itk/\alpha}e^{2\pi itz}dt\\
& + \, \sum_k \frac{1}{2\pi i\alpha} f' 
\Big (\frac{k}{\alpha}\Big ) \int^\alpha_0
 \frac{1}{\alpha} e^{-2\pi ivk/\alpha}e^{2\pi ivz}dv\\
&  + \, e^{-2\pi i \alpha z} \sum_k \frac{(-1)}{2\pi i
 \alpha} f' \Big (\frac{k}{\alpha}\Big ) 
\int^\alpha_0 \frac{1}{\alpha} e^{-2\pi 
itk/\alpha} e^{2\pi itz}dt\,.
\end{align*}
\normalsize
Writing $t=\alpha+u$  
in the second and fourth sums produces:
{\allowdisplaybreaks
\begin{align*}
f(z)\, & = \, \sum_k \frac{1}{\alpha} f\Big 
(\frac{k}{\alpha}\Big )\int^\alpha_0 \Big 
(1-\frac{v}{\alpha}\Big ) e^{-2\pi ivk/\alpha}e^{2\pi ivz}dv\\
& \quad +  \, \sum_k \frac{1}{\alpha} f\Big (\frac{k}{\alpha}\Big )
 \int^0_{-\alpha} \Big (1+\frac{u}{\alpha}\Big )
 e^{-2\pi i uk/\alpha}e^{2\pi iuz}du\\
& \quad + \, \sum_k \frac{1}{2\pi i\alpha^2} f' \Big 
(\frac{k}{\alpha}\Big ) \int^\alpha_0 e^{-2\pi 
ivk/\alpha} e^{2\pi ivz}dv\\
& \quad  + \, \sum_k \frac{(-1)}{2\pi i\alpha^2} f' \Big 
(\frac{k}{\alpha}\Big ) \int^0_{-\alpha} 
e^{-2\pi iuk/\alpha} e^{2\pi iuz}du\\
\, &= \, \sum_k f \Big (\frac{k}{\alpha}\Big )
 \int^\alpha_{-\alpha} \Big (\frac{\alpha-|v|}
{\alpha^2}\Big ) e^{2\pi iv(z-k/\alpha)} dv\\
& \quad + \, \sum_k f' \Big (\frac{k}{\alpha}\Big ) 
\int^\alpha_{-\alpha} \frac{\sgn(v)}{2\pi i\alpha^2}
 e^{2\pi iv(z-k/\alpha)}dv\\
\,&=  \, \sum_k f \Big (\frac{k}{\alpha}\Big ) 
\Big (\frac{\sin(\pi\alpha p)}{\pi\alpha p}\Big )^2\\
& \quad + \, \sum_k f' \Big (\frac{k}{\alpha}\Big ) (-p)
\Big (\frac{\sin (\pi\alpha p)}{\pi\alpha p}\Big )^2\\
& \qquad  \qquad \mbox{\{with $p\equiv \frac{k}{\alpha}-z\MSB$\}}
\\[1.26ex]
\, &=  \,  \sum_k f\Big (\frac{k}{\alpha}\Big ) \Big 
(\frac{\sin \pi\alpha z}{\pi\alpha}\Big )^2(z-k/\alpha)^{-2}\\
& \quad +\,  \sum_k f' \Big (\frac{k}{\alpha}\Big ) 
(z-k/\alpha)\Big (\frac{\sin \pi\alpha z}{\pi\alpha}
\Big )^2 (z-k/\alpha)^{-2}\,.
\end{align*}  }

In other words, the alternate representation
\begin{equation}\label{1.11}
f(z) = \Big (\frac{\sin \pi\alpha z}{\pi\alpha}
\Big )^2 \Big \{\sum^\infty_{k=-\infty} 
\frac{f(k/\alpha)}{(z-k/\alpha)^2} + 
\sum^\infty_{k=-\infty}\frac{f'(k/\alpha)}{z-k/\alpha}\Big \}
\end{equation}
is available to us anytime (1.3) holds. \,Compare: (1.8),~(1.5),~(1.6). One 
readily checks that the right-hand expression behaves like
\[
\qquad \qquad \qquad  f \Big (\frac{k}{\alpha}\Big )\, + \,
 f' \Big (\frac{k}{\alpha}\Big ) 
\Big (z-\frac{k}{\alpha}\Big )
\,  + \, O(1) \Big (z- \frac{k}{\alpha}
\Big )^2 \hspace{8.0em}\mbox{(1.11$'$)}
\]
near each nodal point $k/\alpha$. (See [Hig2, pp.$\MSC$97--100] for 
an interesting higher-order generalization; $\MSAH$also p.$\MSB$58 
in the second volume of this work.)

\subsection{} Just as with (1.8), once (1.11) 
is known, one can pass to $(f(z)-f(0))/z$ to 
widen the class of admissible $f$. (See [Vaa].)
For present purposes, however, this augmentation
is unnecessary and a simple limit trick suffices
to give us the heuristic formats that we seek 
(\`{a} la \S1.1) in \emph{both} cases of $m$.

One reasons as follows. In view of 
(1.11) and (1.11$'$) [with $\alpha = 1$] and 
the fact that $\ell \in {\mathbb Z}^+$, 
the interpolatory format
\begin{equation}\label{1.12}
{\mathcal F}_\ell(z)= \Big (\frac{\sin \pi z}{\pi} 
\Big )^2 \Big \{ \sum^\ell_{k=0} \frac{1}{(z-k)^2} 
+ \frac{A}{z}+ \frac{B}{\ell-z} \Big \} \quad (A>0,\, B>0)
\end{equation}
clearly occupies a distinguished position vis \`{a} 
vis $\chi_{[0,\ell]}(x)$ --- at least for suitably 
controlled $A$ and $B$.

For the $L_1$ norm of 
${\mathcal F}_\ell(x)-\chi_{[0,\ell]}(x)$ to be finite, it 
is necessary that $B-A$ reduce to zero. In considering 
(1.12) as a format, we therefore tacitly assume this to be so.

For bounded $A$, simple use of the identity
\begin{equation}\label{1.13}
1=\Big (\frac{\sin \pi x}{\pi} \Big )^2 
\sum^\infty_{n=-\infty} \frac{1}{(x-n)^2}
\end{equation}
shows 
that $\|{\mathcal F}_\ell-\chi_{[0,\ell]}\|_p = O(\ell^{1/p})$
for \emph{each}  p $>$ 1. (The contribution from  
$ |x| \MSZ\geqq\MSZ 3\ell $ is trivially seen to be $o(1)$.) 
Consideration of the elementary calculus estimate 
\[
\sum^\infty_{n=1} 
\frac{1}{(n+\omega)^2}\, < \, \frac{1}{\omega} \qquad \;(\omega>0)
\]
quickly demonstrates that the earlier exponent $1/p$ 
is replaceable by $0$.

In light of this, it now makes eminently good sense
to try to \emph{approximate} $\sgn(x)$ by letting  
$\ell \rightarrow \infty$ in the combination 
$ 2({\mathcal F}_\ell (x) - \frac{1}{2})$.  Insofar
as the choice of $\MSB A \MSB$ can be kept bounded, the term 
with $\MSB B \MSB$ automatically drops out \linebreak on compact 
subsets of ${\mathbb R}\MSA$ (in this connection, 
see also (1.7)). $\MSB$The expression
\[
{\mathcal F}^\ast (z)\, =\, 2 \bigg ( 
\frac{\sin \pi z}{\pi}\bigg )^2 \bigg \{ 
\sum^\infty_{k=0} \frac{1}{(z-k)^2} 
+ \frac{A}{z}\bigg \}  - 1
\]
thus takes on a special significance in regard 
to $\sgn(x)$.  Notice incidentally that 
${\mathcal F}^\ast(0)=1\MSC$ (not $0$)$\MSA$;
$\MSB$ also that ${\mathcal F}^\ast (x)=O(1)$.

By virtue of (1.13), ${\mathcal F}^\ast(z)$ can be 
written in three different ways:
\begin{align*}
{\mathcal F}^\ast \,&= \, \Big (\frac{\sin \pi z}{\pi}
 \Big )^2 \Big \{ \sum^\infty_{k=0} \frac{1}{(z-k)^2}
 - \sum_{k<0} \frac{1}{(z-k)^2 } + \frac{2A}{z} 
\Big \}; \tag{1.14$_a$} \\
 {\mathcal F}^\ast \,&=\,  -1 + \Big (\frac{\sin 
\pi z}{\pi} \Big )^2 \Big \{ 2 \sum^\infty_{k=0}
 \frac{1}{(z-k)^2} + \frac{2A}{z} 
\Big \}; \tag{1.14$_b$} \\[0.7ex]
 {\mathcal F}^\ast \,&=\;\:  1 + \Big (\frac{\sin 
\pi z}{\pi} \Big )^2 \Big \{\hspace{-.63ex}-2 \sum_{k<0} 
\frac{1}{(z-k)^2} + \frac{2A}{z} 
\Big \}\,. \tag{1.14$_c$}
\end{align*}

One wants the $L_1$ norm 
of ${\mathcal F}^\ast(x)-\sgn(x)$ 
to be finite.  Since the earlier calculus estimate can
be strengthened to read
\setcounter{equation}{14}
\begin{equation}\label{1.15}
\frac{1}{\omega} - \frac{1}{\omega^2} < \sum^\infty_{n=1}
 \frac{1}{(n+\omega)^2} < \frac{1}{\omega}\,\,,
\end{equation}
equations (1.14$_b$) and (1.14$_c$) immediately show 
that there is \emph{precisely one} admissible value of 
$A \MSA$; \,viz., $A=1$.  

We'll subsequently consider (1.14) only under 
this restriction.  
The obvious \emph{hope} here, of course, is that, when $A=1$, 
the norms $\|{\mathcal F}_{\ell} - \chi_{[0,\ell]}\|_p$ 
will remain uniformly bounded for all $\MSB p \geqq 1 \MSB$ as 
$\ell \rightarrow \infty$.

With this, our informal determination of plausible 
$f(z)$$\MSC$-$\MSC$formats is complete.

\begin{remark}
Though, in this section, the role played by equation (1.9) 
has mainly been motivational, in other approximation 
settings (less ``constrained'' as to $f'$), matters  
change and (1.9), rather than (1.11), becomes 
the ``ansatz-of-choice.''  $\MSA$See [Vaa, p.187 ff] 
-- or [Ess, p.16 ff] -- for a good example. 
\end{remark}

\section{Entire Functions of Beurling-Selberg Type}

\subsection{} After the introductory discussion given in \S1, it is 
relatively pedestrian to go back and 
place matters on a rigorous footing. For this 
purpose, we first recall (1.11$'$) and then write
\renewcommand{\theequation}{2.\arabic{equation}}
\setcounter{equation}{0}
\begin{align}
B(z) & =  \Big ( \frac{\sin\pi z}{\pi} \Big )^2
 \Big \{\sum^\infty_{k=0} \frac{1}{(z-k)^2} -
 \sum^\infty_{n=1} \frac{1}{(z+n)^2} +
 \frac{2}{z}\Big \}\,; \label{2.1}\\
b(z) & =  \Big ( \frac{\sin\pi z}{\pi} 
\Big )^2 \Big \{\sum^\infty_{k=1} \frac{1}{(z-k)^2} 
- \sum^\infty_{n=0} \frac{1}{(z+n)^2} +
 \frac{2}{z}\Big \}\,; \label{2.2}\\
W(z) & =  \Big ( \frac{\sin\pi z}{\pi} \Big )^2 
\Big \{\sum^\infty_{k=1} \frac{1}{(z-k)^2} -
 \sum^\infty_{n=1} \frac{1}{(z+n)^2} + 
\frac{2}{z}\Big \}\,; \label{2.3}\\
K(z) & =  \Big ( \frac{\sin\pi z}{\pi z} \Big )^2 
= \int^1_{-1}(1-|v|) e^{2\pi ivz}dv\,.\label{2.4}
\end{align}
One immediately checks that $B(z)$ corresponds to 
${\mathcal F}^{\ast}(z)$ (\`{a} la (1.14)), 
$b(z)=-B(-z),\, W=B-K \MSZ=\MSZ b+K$, and 
that the function  
$2H(x)$ appearing in [Beu, p.371 (bot)] is simply 
$b(x/2\pi)$. According to Selberg, the functions 
$B$ and $H$ were first considered by Beurling around
 1940 or so\,; cf. [Sel2, pp.226 (lines 10--17), 218 (20.17)].

Motivated by equations (20.19), (20.2), and (20.12)
 in [Sel2, pp.214--218], we also define
\begin{align}
S_\ell(z) & =  \sfrac{1}{2}\, [B(z)+B(\ell-z)] \label{2.5}
\\[.63ex]
\sigma_\ell(z) &=  \sfrac{1}{2}\, [b(z)\MSE+\MSE b(\ell-z)]\label{2.6}
\end{align}
for any positive real $\ell$. When $\ell$ is
 integral, one immediately checks that $S_\ell(z)$
 reduces to ${\mathcal F}_\ell(z)$ (cf. (1.12))
 with $B=A=1$. In a similar way, $\sigma_\ell(z)$
 is seen to reduce to
\\[.05ex]
\[
\Big (\frac{\sin \pi z}{\pi}\Big )^2 \Big \{
 \sum^{\ell-1}_{k=1} \frac{1}{(z-k)^2} + 
\frac{1}{z} + \frac{1}{\ell-z}\Big \}\,.
\]
\\[-1.20ex]
These reductions and the discussion in the 
first part of $\MSB$\S1.6$\MSA$ make it \emph{plausible} 
at least that one will have\vspace*{-1.25ex}
\[
\sigma_\ell (x)\, \leqq\, \chi_{[0,\ell]}(x)\, 
\leqq\, S_\ell(x) \vspace*{-.25ex}
\]
at every $x \in {\mathbb R}$. (Compare [Sel2, p.217 (lines 13, 
3, 10)].) \,For later use, notice 
too that $0\leqq K(x)\leqq 1$.

\begin{theorem}
The functions $B, b, W$ 
are entire and satisfy the following basic
 properties for real $x$:\\[.63ex]
(a) \quad $B(x)=W(x)+K(x), \, b(x)=W(x)-K(x)$\,;\\[.63ex]
(b) \quad $W(x)+W(-x)=0, \, B(x)+B(-x)=2K(x)$\,;\\[.63ex]
(c) \quad $b(x)\,\leqq \,\sgn(x)\,\leqq\, B(x)$\, (strictly if 
$x\notin {\mathbb Z})$\,;\\[.63ex]
(d) \quad $W(x) \in [1-K(x), 1]$ \,if\, $x>0$\,;\\[.63ex]
(e) \quad $W(x) \in [-1, -1+K(x)]$ \,if\, $x<0$\,;\\[.63ex]
(f) \quad $|b(x)-\sgn(x)|\leqq 2K(x)\;,\; 
|B(x)-\sgn(x)|\leqq 2K(x)$\,;\\[.63ex]
(g) \quad $\int^\infty_{-\infty}\, (\sgn(x)-b(x))\,dx =
 \int^\infty_{-\infty}\,(B(x)-\sgn(x))\,dx=1$\,.\\[.78ex]
For $z\in {\mathbb C}$, the magnitudes of the numbers
 $|B(z)|, |b(z)|$, and $|W(z)|$ are at most $O(1)\exp
(2\pi|\Imm(z)|)$.
\end{theorem}

\begin{proof} 
Assertions (a) and (b) are trivial; (c) is 
nearly so upon utilizing 
relations (1.14) and (1.15).
To verify (d), one needs to exploit a 
refinement of (1.15); viz.,
\[
\frac{1}{\omega}-\frac{1}{2\omega^2} < 
\sum^\infty_{k=1} \frac{1}{(k+\omega)^2} <
 \frac{1}{\omega}\,\,\,.
\]
The left-hand portion of this is immediately recognized 
as a standard consequence of the Euler-Maclaurin 
summation formula with $f(t)=(t+\omega)^{-2}$.  
Cf. [Ste, pp.124(bot), 132(9)(10), 133(lines 6--20)] with
${\mathfrak m}=2\tau =2$. (Alternatively: see 
[Jor, pp.8(bot), 254(2), 255(6)(7), 261(4)] 
and [Malm, pp.58(8), 69(47)].)

Assertions (e) and (f) follow from (d) by means 
of (b) and (a), respectively. 

To address (g), 
one simply notes that:
\begin{align*}
\int^0_{-\infty} (B(x)+1)dx \,& =\, 
 \int^\infty_0 (B(-y)+1)dy\\
\,&=\, \int^\infty_0 (2K(y)-B(y)+1)dy\\
\,&=\, \int^\infty_{-\infty}K(y)dy-\int^\infty_0(B(y)-1)dy\MSB;\\[0.5ex]
\int_{{\mathbb R}}(\sgn(x)-b(x))dx 
\,& =\, \int_{{\mathbb R}}(\sgn(-y)-b(-y))dy\\
\,& =\,  \int_{{\mathbb R}}(B(y)-\sgn(y))dy\MSE\MSB.
\end{align*}   

It remains to control things for $z\in {\mathbb C}$.
For this, it is enough to look at \linebreak $|W(z)|$$\MSB$ 
and, then, only over   
$\{\Ree(z)\geqq 0\}\cap \{|z|\geqq \frac{1}{2}\}$. $\MSC$By 
virtue of (1.13),
\begin{equation}\label{2.7}
1-W(z)\,=\, 2 \Big (\frac{\sin \pi z}{\pi}\Big)^2 
\Big [\frac{1}{2z^2}+\sum^\infty_{n=1}
\frac{1}{(z+n)^2}-\frac{1}{z}\Big ]\,.
\end{equation}
To estimate the bracketed term, one can either use
 Euler-Maclaurin with $f(t)=(t+z)^{-2}$ and
 a suitably big ${\mathfrak m}=2\tau$ or 
else {\it Stirling's formula} coupled with the fact ([Erd, p.15 (2)]) 
that \vspace*{-0.5ex}
\[
\frac{d^2}{dz^2} \log \Gamma(z)=
\sum^\infty_{n=0} \frac{1}{(z+n)^2}\,\,\,.
\]
The bracket is quickly seen to be $O(|z|^{-3})$
 and we are done. \bx \end{proof}

The foregoing ideas are readily combined with 
a little calculation to yield
\\[.63ex]
{\small
\begin{equation}\label{2.8}
1-W(x)= \left \{\begin{array}{l}
\hspace{2.9em}\displaystyle \Big ( \frac{\sin \pi x}{\pi}\Big )^2 
\Big [\sum^N_{k=1} \frac{2B_{2k}}{x^{2k+1}} 
+ R_N(x)\Big ]\,, \MSC\MSC x > 0 \\
\\
\displaystyle K(x)\Big [1-2x+2\sum^\infty_{n=2}(-1)^n 
\zeta (n)(n-1)x^n \Big], \MSC\MSC 0\leqq x<1
\end{array}\right \}\,,
\end{equation}}
\\[.63ex]
wherein $B_\nu$ are the standard Bernoulli numbers
([Jor, Ste]), $N\geqq 0$, and the remainder
term $R_N(x)$ satisfies
\[ R_N(x) = 2\theta B_{2N+2} x^{-2N-3} \mbox{\,
 with\, } 0< \theta < 1\,.     
\MSA\footnote{When combined with the elementary formula
for $W(u)-W(u+m)$ ($m \geqq 1$), this estimate 
leads to a very efficient way of \emph{computing} $\MSA W(x)$ at 
any $x \in {\mathbb R}$. Compare [Ess, p.18 (24)], where 
the same idea is readily seen to work after breaking 
things up into partial fractions.}
\]

An analogous formula holds for $1-W(z)$ on ${\mathbb C}
 \cap \{|\Arg(z)|\leqq \pi-\delta\}$ modulo a minor
 revision in the part concerning $\theta$. Cf. 
(2.7) and [Erd, p.47 (1)(7)]. Restricting things
 to $\{|z|<1\}$ leads to the familiar expansion
 of $\csc^ 2(\pi z)$ near $z=0$ (cf. [Erd, p.51 (3)])
 and to the curious Taylor series development
\begin{equation}\label{2.9}
\frac{W(z)}{K(z)} = 2z + \sum^\infty_{m=1} 4m\, 
\zeta (2m+1)z^{2m+1}\,.
\end{equation}

\begin{theorem}
Given any $\ell>0$. The 
functions $\sigma_\ell$ and $S_\ell$ are entire, 
have magnitude $O(1)\exp(2\pi|\Imm(z)|)$, and satisfy
 the following basic properties for real $x$:\\[.75ex]
(a) \qquad $\sigma_\ell (x)\, \leqq \, \chi_{[0,\ell]}(x)\, 
\leqq \, S_\ell (x)$\,;\\[.75ex]
(b) \qquad $\max \{\chi_{[0,\ell]}(x)-\sigma_\ell(x), \: 
S_\ell(x)-\chi_{[0,\ell]}(x)\} \leqq K(x) 
+ K(\ell-x)$\,; \\[.75ex]
(c) \qquad $ \int^{\infty}_{-\infty} \, (\chi_{[0,\ell]}(x)
-\sigma_\ell (x))\,dx\,
 = \, \int^{\infty}_{-\infty} \, (S_\ell (x) 
- \chi_{[0,\ell]}(x))\,dx\, = 1 $.
\vspace*{0.35ex}
\end{theorem}

\begin{proof}
It suffices to verify (a) for
$x \not= 0 \, , \,\ell$. 
Under this restriction,
\begin{equation}\label{2.10}
\chi_{[0,\ell]}(x) = \frac{1}{2} [\sgn(x)+\sgn(\ell-x)]
\end{equation}
holds and the desired estimate follows immediately
 from Theorem 2.2(c). Assertion (b) is proved similarly
 utilizing Theorem 2.2(f). Equation (c) follows directly
 from (2.10) and Theorem 2.2(g). \bx
\end{proof}

\begin{theorem} 
The entire functions 
$\{B, b, W, S_\ell, \sigma_\ell \}$ have type
 exactly equal to $2\pi$.
\end{theorem}

\begin{proof} Simply put $z=\beta+iy$ ($y$ large) 
in the asymptotic development for $W(z)$ cited 
after (2.8) and
keep $\beta$ bounded; likewise for $K(z)$. Any 
negative powers of $\ell-z$ are best re-expressed 
as power series in $z^{-1}$. (When $\ell \in \frac{1}{2} 
+ {\mathbb Z}$, a slight anomaly occurs in the 
$\{\sigma_\ell, S_\ell\}$ asymptotics.)

Application of the difference operator  
$T[f]\equiv f(z+1)-f(z)$ immediately furnishes
 an alternate proof in the case of 
$\{B, b, W\}$. With a bit more effort, the same
low-level approach also works for 
$\sigma_\ell$ and $S_\ell$. \bx
\end{proof}

\subsection{} To streamline the development of the
 next two theorems, it is helpful to preface 
matters with a general result about entire
 functions of exponential \linebreak type.

\begin{lemma} 
Given any positive $p$ 
and $\tau$. Let $f(z)$ be an entire function of
 exponential type $\leqq \tau$. Suppose 
that $f(x) \in L_p({\mathbb R})$. Then:\\[.63ex]
(a)\quad\, $f(x) \in L_q({\mathbb R})$ for 
$q \in [\,p, \infty)$;\\[.63ex]
(b)\quad \,\,$ \int^\infty_{-\infty}|f(x+iy)|^p dx 
\leqq e^{p \tau |y|} \int^\infty_{-\infty}|f(x)|^p 
dx$  for $y \not= 0$\,;\\[.63ex]
(c) \quad $f(z) = o(1)$ on every horizontal strip
 $\{|\Imm(z)|\leqq \Delta\}$;\\[.63ex]
(d) \quad $ \sum^\infty_{n=1} |f(z+n)|^p$ converges
 uniformly on $\{0\leqq \Ree(z)\leqq 1,
 |\Imm(z)|\leqq \Delta\}$;\\[.84ex]
(e) \quad $f^{(k)}(x) \in L_p({\mathbb R})$
 for every $k \geqq 1$.
\end{lemma}

\begin{proof} Assertion (b) is a classical 
result of Plancherel and P\'{o}lya; 
see$\MSA$ [PlP, pp.120--124] and 
[Lev, pp.$\MSA$50, 51, 38 (theorem 2)]. 
(The proof hinges on sub-{\linebreak}harmonicity and some 
related Phragm\'{e}n-Lindel\"{o}f type estimates.) 
Straightforward use of the subharmonicity of $|f|^p$ 
then gives
\[
|f(z)| = O_\Delta (1) \Big \{ \int^\infty_{-\infty}
 |f(t)|^pdt\Big \}^{1/p}
\]
on every $\{|\Imm(z)|\leqq \Delta\}$. Cf. [Lev, p.51 (bot)].
 Assertion (c) follows by a minor adaptation of the same
 manipulation; cf. [Lev, p.138 (top)]. Assertion (a) 
is then obvious. 
By using a second (more astute!) adaptation, 
one obtains (d). Cf. [Lev, p.138 (bot)], [PlP, p.126 (top)], 
and [Boa, p.101 (lines 13--23)]. 

To establish (e), 
we follow [PlP, p.127]. It suffices to treat $k=1$.
 For $x_0 \in {\mathbb R}$, the function 
$g(z) \equiv z^{-1} [f(x_0+ z)-f(x_0)]$ 
is entire. Accordingly:
\begin{align*}
|g(0)|^p \,& \leqq \, \frac{1}{2\pi} \int^{2\pi}_0 
|g(\delta e^{i\phi})|^pd\phi\\[.31ex]
|f'(x_0)|^p \,& \leqq \, \frac{1}{2\pi\delta^p}
 \int^{2\pi}_0 |f(x_0+\delta e^{i\phi})-f(x_0)|^pd\phi\\[.63ex]
& \,\;\quad \{|u+v|^p \leqq 2^p \max (|u|^p, |v|^p)\}\\[.63ex]
|f'(x_0)|^p \,& \leqq \, \frac{2^p}{2\pi\delta^p}
 \int^{2\pi}_0 \Big[ |f(x_0+\delta e^{i\phi})|^p 
+ |f(x_0)|^p \Big]d\phi\\[.63ex]
|f'(x_0)|^p \,& \leqq \, \frac{2^{p+1}}{2\pi\delta^p}
 \int^{2\pi}_0 |f(x_0+\delta e^{i\phi})|^pd\phi \, .\\[-1.55ex]
\end{align*}
To conclude, one simply integrates over $x_0$ 
and applies (b). Upon taking $\delta=\tau^{-1}$,
 we find that
\[
\int^\infty_{-\infty} |f'(x)|^pdx \leqq 2(2e\tau)^p
 \int^\infty_{-\infty} |f(x)|^p dx\,.
\]
Compare [Boa, p.211 (bot)]. \bx
\end{proof}

\begin{theorem}
Let ${\mathcal M}^+$ 
be the set of all entire functions of exponential
 type $\leqq 2\pi$ which majorize $\sgn(x)$ along 
${\mathbb R}$. Let ${\mathcal M}^-$ be the 
counterpart for $f(x) \leqq sgn(x)$. We then have
\\[.63ex]
\[
\inf_{f\in {\mathcal M}^+} \int_{{\mathbb R}}
(f(x)-\sgn(x))dx = 1 = \inf_{h\in {\mathcal M}^-}
 \int_{{\mathbb R}} (\sgn(x)-h(x))dx\,.
\]
\\[.63ex]
The extremal functions are respectively $B(z)$ 
and $b(z)$; they are {\em unique}.
\end{theorem}
\begin{proof}
Thanks to the transformation 
$h(z)=-f(-z)$, it is enough to look at ${\mathcal M}^+$.
 There is clearly no loss of generality if we 
also restrict ourselves to functions which satisfy
 $f(x)-\sgn(x)\in L_1({\mathbb R})$. The differences 
$f(z)-B(z)$ will \linebreak then satisfy the hypotheses of
 Lemma 2.6 with $(p,\tau)=(1, 2\pi)$. By combining 
this information with the asymptotic of $W(z)$
 mentioned after (2.8), we immediately see that 
$f(z)-B(z)= o(1)$ on every strip $\{|\Imm(z)|   
\leqq \Delta\}$ and \linebreak that
\begin{equation}\label{2.11}
\int^\infty_{-\infty} |f'(x)|dx \leqq 
\int^\infty_{-\infty} |f'-B'|dx + 
\int^\infty_{-\infty} |K'|dx+\int^\infty_{-\infty}
 |W'|dx<\infty\,.
\end{equation}
Likewise for $f^{(k)}(x)$ with $k \geqq 2$.

We now follow Beurling ([Beu]) and  
look at matters in the framework  
of a judiciously chosen integration by parts. 
To set the stage, we first write
\[
\theta(x)=f(x)-\sgn(x),\;\, \theta^\ast(x)=
B(x)-\sgn(x),\;\, k(x)= x -\llbracket x \rrbracket -\sfrac{1}{2}
\]
and recall that
\\[-1.26ex]
\[
k(x) = -\sum^\infty_{n=1} \frac{\sin(2\pi nx)}
{\pi n} \, \mbox{ \,for\, } x \notin {\mathbb Z}\,.
\]
\\[-.63ex]
The decomposition $\theta= \theta^\ast
+ (f-B)$ assures us that $\theta (x)=o(1)$. (Cf. Theorem 2.2(f)   
\`{a} propos $\theta^\ast(x)$.)

Fix any $y \notin {\mathbb Z}$. For large $N$
 and small $\varepsilon$, we then have:
\begin{align*}
&  \int^{N-\varepsilon}_{y+\varepsilon} 
k(u)d\theta(y-u) + \int^{N-\varepsilon}_{y+\varepsilon}
 \theta (y-u)dk(u) =
  \Big[k(u)\theta (y-u)
\Big ]^{N-\varepsilon}_{y+\varepsilon}\\[.63ex]
&  \int^{N-\varepsilon}_{y+\varepsilon}
 k(u)f'(y-u)(-1)du + \int^{N-\varepsilon}_{y+\varepsilon}
 \theta (y-u)du\\
&  \hspace{11em}= \sum_{y+\varepsilon < n<N}\theta(y-n)+
 \Big[k(u)\theta (y-u) 
\Big ]^{N-\varepsilon}_{y+\varepsilon}\,. 
\end{align*}
Upon letting $N \rightarrow \infty$,
\[
 -\int^\infty_{y+\varepsilon}k(u)
 f'(y-u)du + \int^\infty_{y+\varepsilon}
 \theta (y-u)du = -k(y+\varepsilon)
\theta(-\varepsilon)+ \sum_{n>y}\theta(y-n)\,. 
\]
In a similar way,
\[
 -\int^{y-\varepsilon}_{-\infty} k(u)f'(y-u)du +
 \int^{y-\varepsilon}_{-\infty} \theta
 (y-u)du = k(y-\varepsilon)\theta(\varepsilon)
+ \sum_{n<y}\theta(y-n)\,.
\]
Accordingly:
\[
\int_{|u-y|>\varepsilon} \Big( \hspace{-.63ex}-k(u)f'(y-u)
+\theta(y-u) \Big )\MSH du \MSA = \MSA \Big[k(y-v) \theta(v) 
\Big]^{\varepsilon}_{-\varepsilon}
+\sum^\infty_{-\infty} \theta(y-n)\,.
\]

By passing to the limit in $\varepsilon$, we conclude
that
\[
2k(y) = \sum^\infty_{-\infty} \theta(y-n)+
\int_{{\mathbb R}} k(u)f'(y-u)du -
 \int_{{\mathbb R}}\theta(v)dv\,.
\]
(The {\it format} of the final integral clearly depends 
crucially on the fact that  
$k'(u)=1$ almost everywhere.)

Observe now that the Paley-Wiener theorem applies
 to $f'(z)$. Cf. (1.3) with $\alpha=1$. Since 
$f'(x)\in L_2 \cap L_1$, its Fourier transform is
 continuous on all of ${\mathbb R}$. 
As such, the ``{\it g\,}'' in (1.3) necessarily 
{\it vanishes} at the endpoints.\footnote{Cf. 
[Z2, pp.249$\MSD$(para 4), 250$\MSD$(2.17)]$\MSH$; also 
[Z1, p.26 (11.6)].} 
By virtue of the Poisson 
summation formula (1.2), we therefore have
\[
\sum^\infty_{m=-\infty} f'(x+m)=  
\int^\infty_{-\infty} f'(y)dy + 0 \MSC=
\MSC 2\MSB,\vspace*{1ex}
\]
from which it follows that
\[
\int_{{\mathbb R}} k(u)f'(y-u)du=\int^1_0 k(u)\Big[
 \sum^\infty_{-\infty} f'(y-u+m)\Big] du=0\,.
\]
In other words:
\begin{equation}\label{2.12}
2k(y)= \sum^\infty_{-\infty} \theta(y-n)-\int_{{\mathbb R}}
\theta(v)dv\,, \; y\notin {\mathbb Z}\,.
\end{equation}

Since $\theta(v) \geqq 0$, making $y\rightarrow 0^+$ yields
\[
-1 = [\mbox{non-negative number}]
-\int_{{\mathbb R}}\theta(v)dv\,;
\]
i.e., $\int_{{\mathbb R}}\theta(v)dv \geqq 1$.
Equality holds {\it only if}
\[
\lim_{y\rightarrow 0^+} \,\sum^N_{-N} \theta(y-n)=0
\]
for each large $N$. This necessitates that 
$f(n)=B(n)$ for each fixed $n$. 
Since $f(x)-\sgn(x) = \theta(x) \geqq 0$,
 we would also need to have $f'(n)=0, \, n\neq 0$, and
 $f'(0) \geqq 0$. Utilizing (1.11)\,\footnote{or, even 
better, the $\{ G_1,G_2 \}$ formalism in \S 1.5 for 
$f(z)-B(z)$  }, we finally deduce that
\\[.21ex]
\[
f(z) - B(z) = \Big (\frac{\sin \pi z}
{\pi}\Big )^2 \frac{c}{z}
\]
\\[.21ex]
with $c=f'(0)-B'(0)=f'(0)-2$. Since $f-B 
\in L_1({\mathbb R})$, the constant $c$ 
must be zero; hence $f(z) \equiv B(z)$. \bx 
\end{proof}

Though Theorem 2.7 is due to Beurling, the 
following one is perhaps best attributed 
to Selberg. (Cf. [Sel2, p.226 (lines 5--10)].)

\begin{theorem} Let $\ell$ be a 
positive integer and ${\mathcal M}^\pm_\ell$
 be the counterpart of ${\mathcal M}^\pm$ 
for $\chi_{[0,\ell]}(x)\,$. We then have
\\[.63ex]
\[
\inf_{f\in {\mathcal M}^+_\ell} \int_{{\mathbb R}}
 (f(x)-\chi_{[0,\ell]}(x))dx = 1 
= \inf_{h\in {\mathcal M}^-_\ell}\int_{{\mathbb R}}
 (\chi_{[0,\ell]}(x)-h(x))dx\,.
\]
\\[.63ex]
The functions $S_\ell(z)$ and $\sigma_\ell(z)$ 
are extremal, but they are {\em not} unique. 
\end{theorem}

\begin{proof} We begin with ${\mathcal M}^+_\ell$ 
and again restrict ourselves to $f$ which also 
satisfy $\|f-\chi_{[0,\ell]}\|_1<\infty$. $\MSB$Just 
as in the proof 
of Theorem 2.7 \vspace*{-.62ex}
(but with $f'\MSB \hookrightarrow \MSA f$), 
\[
\sum^\infty_{m=-\infty} f(x+m)=\hat{f}(0)+0\,.
\]
Putting $x=0$ immediately gives $\hat{f}(0)
\geqq \ell+1$. By Theorem 2.3(c), equality holds
 for $f=S_\ell$. Any $\MSZ$\emph{other} extremal function 
$F$ necessarily satisfies
\[
F(m) = \chi_{[0,\ell]}(m)\quad \mbox{ and }\;\quad 
F'(k)=0\, \mbox{ if }\, k\neq 0,\ell\,\,.
\]
By (1.11) and $F \in L_1$, we then get
\begin{equation}\label{2.13}
F(z) = S_\ell (z) + \eta \Big (\frac{\sin \pi z}{\pi}
\Big )^2 \frac{\ell}{z(\ell-z)}\;\,,
\end{equation}
wherein $\eta = F'(0) - S'_\ell (0) = F'(0)-1$. 
 Nonuniqueness stems from the fact that the right-hand
 side of (2.13) continues to lie in 
${\mathcal M}^+_\ell$ for any sufficiently
 small $\eta \in {\mathbb R}-\{0\}$. Cf. 
[Sel2, p.217 (line 3)]. Observe too that, 
for $\ell \in {\mathbb Z}$\,,
\[
\int^\infty_{-\infty} \Big (\frac{\sin \pi x}
{\pi} \Big )^2 \frac{\ell}{x(\ell-x)} dx 
= \lim_{R\rightarrow \infty} \int^R_{-R} 
\Big (\frac{\sin \pi x}{\pi} \Big )^2 \Big [
 \frac{1}{x} + \frac{1}{\ell-x}\Big ] dx=0\,.
\]
The analysis for ${\mathcal M}^-_{\ell}$ runs 
similarly; see [Sel2, p.217 (line 10)] for the 
pertinent $\eta$-condition. \,\bx
\end{proof}

If $\MSB\ell \notin {\mathbb Z}\MSB$, the qualitative 
reasoning in [Sel2, pp.218 (bot), 219 (top)] immediately 
shows that the functions $S_\ell$ and $\sigma_\ell$ 
are \emph{no longer} extremal: $\MSA$one already does better, 
in fact, with certain multiples 
$\MSA(1-\varepsilon)S_\ell (z)\MSA$ and $\MSA(1+u)
\sigma_\ell(z)\MSA$. (Cf. the second clause of Theorem 2.2(c).
Take $u = -1$ if $0< \ell <1$.)

For further information on this, consult
[DL, Log]\,\footnote{in the case of ${\mathcal M}^+_\ell$} 
and the two graphs given in [Sel2, p.219]. The fact
 that $S_\ell$ and $\sigma_\ell$ are suboptimal
 in regard to $L_1$ will turn out to be of
 relatively little consequence for subsequent purposes.

\begin{remark} In constructive function theory,
 approximants (of various $L_p$ types) for a
 given function $m(x)$ along ${\mathbb R}$ are 
often obtained simply by taking \linebreak convolutions with a 
rescaled Fej\'{e}r kernel $\alpha K(\alpha z)$ or 
similar. See [Tim, Ach] and, e.g., [Kat, p.$\MSA$125 
(theorem 1.10)]. Since $\{B, b, W, S_\ell, \sigma_\ell\}$ 
are basically \linebreak \emph{defined} via  
(1.11), they are not really amenable to being 
re-interpreted in a natural way as convolutions 
over ${\mathbb R}$ of the aforementioned special type. 
The formal $L_2$ expansions for $g(v)$ 
and $g(v-\alpha)$ derived in \S1.5 offer 
valuable insight on this last point.
\end{remark}

\subsection{} We close \S2 by determining what
 the analog of representation (1.3) is for
 the function $W(z)$. Since \vspace*{-.62ex}
\begin{equation}\label{2.14}
\sgn(x) = \int^\infty_{-\infty} \frac{1}
{\pi v}\sin (2\pi xv)dv\,,
\end{equation}
there is a natural suspicion that
\begin{equation}\label{2.15}
W(z) = \int^1_{-1} \frac{Q(v)}{v} \sin (2\pi zv)dv
\end{equation}
holds for some nice, even, continuous function
 $Q$ satisfying $Q(0)=1/\pi$.

Following the hint provided by the proof of 
Theorem 2.7, we first pass to $W'(z)$ and write
\[
W'(z) = \int^\infty_{-\infty} A(v)\cos(2\pi zv)dv
\]
with an even continuous $A(v)$ having support
 $\subseteqq [-1, 1]$. 
The asymptotics of $W(z)$ cited near (2.8)
 assure us that $W'(z)=O[(1+|z|)^{-3}]$ on every 
strip $\{|\Imm(z)|\leqq \Delta\}$. 
By Fourier inversion,
\[
A(v) = \int^\infty_{-\infty} W'(x)\cos (2\pi xv)dx\,.
\]
Letting $v=0$ gives $A(0)=2$. Since
\[
A'(v)=-2\pi \int^\infty_{-\infty} xW'(x)\sin (2\pi xv)dx\,,
\]
the function $A$ necessarily belongs to 
$C^1({\mathbb R})$. In view of the restriction on $\supp (A)$,
\\[-1.26ex]
\[
A(1)=A'(1)=0\,.\vspace*{-3.5ex}
\]
\\

We now write $W(x)=\int^x_0 W'(t)dt$ and exploit Fubini's 
theorem. This gives
\begin{align*}
W(x) \,& = \, \int^x_0 \int^1_{-1} A(v) 
\cos (2\pi tv)dv dt\\
\,& = \, \int^1_{-1} \int^x_0  A(v) \cos (2\pi tv)dt dv\\
\,& = \, \int^1_{-1} A(v) \frac{\sin (2\pi xv)}{2\pi v} dv\,,
\end{align*}
from which it is evident that 
\[
Q(v)=\frac{1}{2\pi}A(v)
\]
works in (2.15). (That $Q$ is {\it unique} in (2.15) follows
 immediately from Leibnitz's rule and the Plancherel theorem.)

\begin{theorem}
The function $W(z)$ 
is representable in the form (2.15) with
\[
Q(v) = \frac{1}{\pi} |v| + (1-|v|)\MSA v \ctn (\pi v)\,.
\]
Extending $Q$ to be zero off $[-1, 1]$ produces an \emph{even}
 $C^1({\mathbb R})$ function (call it $q$) which satisfies
\[
\left \{ \begin{array}{l}
\; \hspace{2em} q(0) = \frac{1}{\pi}, \quad q(1)=0, \quad q'(1)=0\\
\\
q''(0\mspace{1mu}\pm)= -\frac{2}{3} \pi, \;\; q''(1-)=
\frac{2}{3}\pi, \;\; q''(1+)=0\\
\\
\; \hspace{2.9em} q'(v) < 0 \quad\MSC 
\mbox{ for }\MSC\quad 0 < v < 1
\end{array}\right \}\,.
\]
\end{theorem}
\begin{proof}
Consider $W(z)$ on the disk $\{|z|\leqq R\}$. One has
\[
W(z) = \lim_{N\rightarrow \infty} \Big (
\frac{\sin \pi z}{\pi} \Big )^2 \Big \{ 
\sum^N_{m=1} \frac{1}{(z-m)^2} - \sum^N_{m=1}
 \frac{1}{(z+m)^2} + \frac{2}{z}\Big \}
\]
as a uniform limit. But,
\begin{align*}
K(p-m) \,& = \, \int^1_{-1} (1-|v|)
e^{2\pi imv} e^{-2\pi ipv}dv\\
pK(p) \,& = \, \int^1_{-1} \frac{i}{2\pi}
 sgn(v) e^{-2\pi ipv}dv
\end{align*}
via \S1.2. Accordingly (note the ``$-z$''):
\begin{align*}
W(z)\,&= \, \lim_{N\rightarrow \infty} \int^1_{-1
} \Big( \sum^N_{-N} (1-|v|)\sgn(m)e^{2\pi imv} 
+ \frac{i}{\pi} \sgn(v) \Big )e^{-2\pi izv}dv\\
\,& = \,\lim_{N\rightarrow \infty} i\int^1_{-1} 
\Big [\frac{2(1-|v|)}{2\sin (\pi v)} 
\Ree( e^{\pi iv} - e^{2\pi iv(N+\frac{1}{2})} ) 
 + \frac{1}{\pi} \sgn(v)\Big ] e^{-2\pi izv}dv
\\[.31ex]
\,& = \, \int^1_{-1} \Big[(1-|v|)\ctn (\pi v)
+\frac{1}{\pi}\sgn(v)\Big] \sin (2\pi zv)dv\\
& \quad -\lim_{N\rightarrow \infty} \int^1_{-1} 
\frac{1-|v|}{\sin(\pi v)} \cos 2\pi 
v(N+ \vsfrac{1}{2})\sin (2\pi zv)dv\,.
\end{align*}
Since $\frac{1-|v|}{\sin(\pi v)}\sin (2\pi zv)$
 is a continuous function of $(v,z)$, the final
$N$-limit is 0 (uniformly w.r.t. $\!\!z$) by the 
Riemann-Lebesgue lemma. Representation (2.15)
 follows at once.

To conclude the proof, one simply uses the relation 
$q=\frac{1}{2\pi}A$ and a bit of 
elementary calculus.  Observe that $q(v)+q(1-v)
= \frac{1}{\pi}$ holds on $[0,1]$. \;\bx
\end{proof}

\section[Esseen's Lemma over ${\mathbb R}$]
{Esseen's Lemma for Probability Distributions over ${\mathbb R}$}
\renewcommand{\theequation}{3.\arabic{equation}}
\setcounter{equation}{0}

\subsection{}
Our primary focus will now shift to probability 
distributions $F(x)$ and their characteristic functions 
\[
\varphi(\zeta) \equiv \int_{{\mathbb R}} e^{i\zeta x}dF(x)\MSB,
\]
as discussed, for instance, in [Fel2, chap.~15], [Z2, p.262], and
[Bil, \S26].$\MSA$\footnote{We make the tacit assumption 
that all probability distributions 
on ${\mathbb R}$ are taken to be \emph{right} continuous. 
\,Similarly for ${\mathbb R}^k$.}

\begin{lemma}
Let $F$ and $G$ be any two probability distributions 
on ${\mathbb R}$ having continuous densities $f(x)$ 
and $g(x)$, respectively. Assume that 
$0 \leqq g(x) \leqq m$ and that\vspace*{-0.31ex}
\[
\int_{{\mathbb R}} |x|^\alpha dF(x) < \infty, \quad 
\int_{{\mathbb R}} |x|^\alpha dG(x) < \infty \vspace*{.31ex}
\]
for some positive $\alpha\MSH$. $\MSB\MSH$Take   
$\Omega > 0$ and put
\[
\varphi (\zeta) = \int_{{\mathbb R}} e^{i\zeta x}dF(x), \quad
 \psi (\zeta) = \int_{{\mathbb R}} e^{i\zeta x} dG(x)\,.
\]
The a priori inequality
\begin{equation}\label{3.1}
|F(t)-G(t)| \,\leqq\,  c_1\! \int^\Omega_{-\Omega}
 \bigg | \frac{\varphi (\zeta) - \psi(\zeta)}{\zeta}\bigg | 
\,d\zeta \,+\, c_2 \mspace{1mu} \frac{m}{\Omega}
\end{equation}
will then hold on ${\mathbb R}$ for certain universal 
constants $c_1$ and $c_2$. (The $d\zeta$-integral is 
automatically convergent as an improper Riemann integral.)
\end{lemma}

\begin{proof} 
By considering $F(x_0+u)$ and $G(x_0+u)$, it suffices 
to treat the case $t=0$. To expedite matters, 
we refer to Theorem 2.11 and set
\begin{align*}
T(v) & = \left \{ \begin{array}{l}
0\,, \quad \quad v=0\\[.63ex]
\dfrac{Q(v)-Q(0)}{v}\;, \quad 0<|v| \leqq 1
\end{array} \right \}
\in C[-1, 1]\,;\\[1.26ex]
R_B(v) & = \sfrac{1}{i} T(v) + (1-|v|) \,\in\, C[-1, 1]\,;\\[1.26ex]
R_b (v) & = \sfrac{1}{i} T(v) - (1-|v|) \,\in\, C[-1, 1]\,;\\[1.26ex]
\lambda & = \sup_{0\leqq \xi\leqq 1} \left|-iQ(\xi) + \xi (1-\xi)\right|\,.
\end{align*}
For $A>0$ and $h \in C\{0<|v| \leqq A\}$, we also set
\\[.21ex]
\begin{equation}\label{3.2}
\oint^A_{-A} h(v) dv \,=\, \lim_{\varepsilon \rightarrow 0^+}\, 
 \int_{\varepsilon\leqq |v|\leqq A}\, h(v)dv
\end{equation}
\\[.21ex]
whenever the right-hand limit exists.

By combining Theorems 2.2 and 2.11, it is virtually self-evident that
\begin{align*}
\hspace{0em}B(x) & = \oint^1_{-1} \big [\sfrac{Q(v)}{iv} +
  (1-|v|)\big ]e^{2\pi ixv} dv\hspace{0em}\\
\hspace{0em}b(x) & = \oint^1_{-1} \big [\sfrac{Q(v)}{iv} -
  (1-|v|)\big ]e^{2\pi ixv} dv\hspace{0em}\\[.5ex]
\hspace{7.74em}B(x) & = \oint^1_{-1} \big [ \sfrac{1}{\pi iv} +
 R_B(v)\big ] e^{2\pi ixv}dv\hspace{7.74em}&\mbox{(3.3a)}\\
\hspace{7.74em}b(x) & =  \oint^1_{-1} \big [ \sfrac{1}{\pi iv} +
 R_b(v)\big ] e^{2\pi ixv}dv\, .\hspace{7.74em}&\mbox{(3.3b)}
\end{align*}
One knows that
\begin{align*}
B(x) & = \sgn(x) + \theta(x), \quad\;\, 0 \leqq 
\theta (x) \leqq 2K(x)\,;\\[.31ex]
b(x) & = \sgn(x) + \sigma(x), \quad -2K(x) \leqq \sigma (x) \leqq 0 \,.
\end{align*}

Observe now that:
\begin{align*}
2\big(F(0)-G(0)\big) & = 2\big(1 - G(0)\big) \MSA -\MSB 
2\big(1 - F(0)\big)\\[.46ex]
& = \int^\infty_{-\infty} \big(g(x)-f(x)\big) 
   \big[1+\sgn (\Omega x)\big]\,dx\\[.31ex]
& = \int^\infty_{-\infty} \big(g(x)-f(x)\big) 
     \sgn (\Omega x)\, dx\\[.31ex]
& = \int^\infty_{-\infty} \big(g(x)-f(x)\big)
    B(\Omega x)\,dx - \int^\infty_{-\infty}
    g(x)\theta(\Omega x)\,dx\\[.31ex]
& \quad + \int^\infty_{-\infty} f(x) 
    \theta(\Omega x)\,dx\\[.31ex]
& \geqq \int^\infty_{-\infty} \big(g(x)-f(x)\big)
    B(\Omega x)dx - m \int^\infty_{-\infty}
   \theta(\Omega x)dx\\[.31ex] 
& = \int^\infty_{-\infty} \big(g(x)-f(x)\big)
    B(\Omega x)dx - \frac{m}{\Omega}
\end{align*}\\ 
by Theorem 2.2$\MSB$(g). \,A trivial substitution for $B(\Omega x)$ 
gives 
\begin{align*}
\int^\infty_{-\infty}\big(g(x)-f(x)\big)B(\Omega x)dx
& = \int^\infty_{-\infty} \big[g(x)-f(x)\big]
  \Big ( \oint^\Omega_{-\Omega}
 \frac{e^{2\pi ixw}}{\pi iw}\, dw\Big )dx\\
  + \int^\infty_{-\infty} \big[ & g(x)-f(x)\big]
  \Big (\oint^\Omega_{-\Omega} \frac{1}{\Omega}
  R_B \big (\frac{w}{\Omega}\big )
  e^{2\pi ixw}dw \Big )dx\,.
\end{align*} 
Notice, however, that one has
\begin{align*}
\int_{\varepsilon \leqq |w|\leqq \Omega} 
  \msfrac{e^{2\pi ixw}}{\pi iw}\, dw
& = \int_{\varepsilon \leqq |w|\leqq \Omega}
  \msfrac{\sin (2\pi x w)}{\pi w}\, dw\\[.31ex]
& = \msfrac{2}{\pi} \int^\Omega_{\varepsilon}
  \msfrac{\sin (2\pi xw)}{w}\, dw\\[.48ex]
& = \msfrac{2}{\pi}\int^{2\pi x \Omega}_{2\pi x \varepsilon}
  \msfrac{\sin u}{u} \, du \, = \, O(1)
\end{align*}
with an implied constant that is absolute (since
$\int^\infty_{0} u^{-1}\sin(u)\, du$ converges to
$\pi /2$). By the Lebesgue dominated convergence 
theorem and Fubini, it follows that
\small
\begin{align*}
\int^\infty_{-\infty} \big(g(x)-f(x)\big)
 B(\Omega x)dx
& = \oint^\Omega_{-\Omega} 
   \frac{1}{\pi iw} \int^\infty_{-\infty} \big[g(x)-f(x)\big]
   e^{2\pi ixw} dx\, dw\\[.31ex]
& \quad + \oint^\Omega_{-\Omega} \frac{1}{\Omega}
  R_B \big (\frac{w}{\Omega}\big )
  \int^\infty_{-\infty} \big[g(x)-f(x)\big]
  e^{2\pi ixw} dx\, dw\\[.31ex]
& = \oint^\Omega_{-\Omega} r(w)\big[\psi (2\pi w)-
  \varphi (2\pi w)\big] dw
\end{align*} 
\normalsize
wherein
\begin{align*}
r(w) & \equiv \msfrac{1}{\pi iw} + \msfrac{1}{\Omega} R_B 
  \big (\msfrac{w}{\Omega}\big )\\[.63ex]
& = \msfrac{1}{\Omega} \big (\msfrac{1}{\pi iv}+R_B(v)\big )
   \quad \mbox{\{with $v=w/\Omega$\}}\\[.63ex]
& = \msfrac{1}{\Omega} \Big [\msfrac{Q(v)}{iv}+ (1-|v|)\Big ]
  = \msfrac{1}{\Omega}\Big [\msfrac{-iQ(v)+v(1-|v|)}{v}\Big ]\,.
\end{align*}
Since $Q$ is even, we clearly have
\[
|r(w)| \leqq \frac{\lambda}{|w|} 
  \quad \mbox{for\,\, $0<|w|\leqq \Omega$}\,.
\]
Putting everything together, we find that:
\begin{align*}
2\big(F(0)-G(0)\big) & \geqq \int^\infty_{-\infty} 
  \big(g(x)-f(x)\big) B(\Omega x)dx - \frac{m}{\Omega}\\[.31ex]
& = \oint^\Omega_{-\Omega} r(w) 
   \big [\psi (2\pi w)-\varphi (2\pi w)\big ]dw-\frac{m}{\Omega}\\[.31ex]
& \geqq -\oint^\Omega_{-\Omega} |r(w)| \big |\varphi (2\pi w)
  - \psi (2\pi w)\big |dw - \frac{m}{\Omega}\\[.31ex]
& \geqq -\lambda \oint^\Omega_{-\Omega}
  \frac{|\varphi (2\pi w)-\psi (2\pi w)|}
  {|w|}\, dw - \frac{m}{\Omega}\\[.31ex]
& \geqq - \lambda \oint^{2\pi \Omega}_{-2\pi \Omega}
  \frac{|\varphi (w)-\psi(w)|}{|w|}\, dw - \frac{m}{\Omega}\,\,,
\end{align*}
where, in the last three lines, a natural convention is made if
the $dw$-integral diverges.
A parallel manipulation with $\sgn (\Omega x)=
b(\Omega x)-\sigma (\Omega x)$ gives
\[
2\big(F(0)-G(0)\big) \,\leqq\, \lambda \oint^
{2\pi \Omega}_{-2\pi \Omega}
   \frac{|\varphi(w)-\psi(w)|}{|w|}\, dw  +  \frac{m}{\Omega}\,\,.
\]
Dividing by ${\TS 2}$ and 
replacing $\Omega$ by $\Omega/2\pi$ finally yields
\[
|F(0)-G(0)| \,\leqq\, \frac{\lambda}{2}
 \oint^\Omega_{-\Omega} \frac{|\varphi(w)-\psi(w)|}{|w|}\,
  dw + \frac{\pi m}{\Omega}\,\, .
\]

To finish up, we set   
$\MSA\widetilde{\alpha} = \min \{\alpha, 1\}\MSA$ and 
simply note that
\begin{align*}
\varphi (w)-1 & = \int_{{\mathbb R}} (e^{ixw}-1)\MSH dF(x)  
= O(1)|w|^{\widetilde{\alpha}} \int_{{\mathbb R}} 
|x|^{\widetilde{\alpha}} dF(x)\MSH,
\end{align*}
the implied constant being at most 2. 
\,Similarly for $\psi (w)-1$. \vspace{0.78ex} \;\bx  
\end{proof} 

Estimate (3.1) corresponds to [Ess, p.32] and 
[Fel2, p.538 (3.13)] and is often referred to 
(together with Theorem 3.3 below) as the 
\emph{Esseen smoothing lemma}. \linebreak  The 
traditional proof of (3.1) makes use of an auxiliary 
convolution. The approach adopted 
here, based on the Beurling 
function (i.e., eq. $\MSZ$(3.3)),
\pagebreak is somewhat more direct. $\MSH\MSD$Since 
\\[-.20ex]
\[{\TS{ 
\sqrt{\MSAH 0 + \frac{1}{\pi^2}} \leqq \lambda
 \leqq \sqrt{\frac{1}{16}+\frac{1}{\pi^2}}<\frac{1}{2}\,,}}
\] 
\\[-1.80ex]
taking $\MSA (c_1, c_2)= (\frac{1}{4}, \pi)\MSAH$ will 
certainly be admissible.\footnote{Readers having computer
experience will readily check that $\lambda = .3263598$ to
7 decimal places. Those familiar with [Ess, chap. II]
may also wish to note 
that $\lambda/2 \approx (1.025)/2\pi \MSH.$} 

\begin{theorem}
\,Inequality (3.1) remains true without the simplifying 
hypothe-{\linebreak}sis 
that $F$ have a continuous density function on ${\mathbb R}$.  It 
is also valid when 
the \linebreak initial $\MSD C^1$ function $\MSA G \MSA$ is 
only known to be real and to satisfy 
\\[-0.93ex]
\setcounter{equation}{3} 
\begin{equation}\label{3.4} 
G(-\infty)=0\,,\:\: G(\infty)=1\,,\:\:  
|G'(x)|\leqq m\,,\:\: 
\TS \int^\infty_{-\infty}\,|x|^\alpha \MSA |G'(x)|\MSA dx < \infty\,.
\\[.93ex]
\end{equation}
\end{theorem}

\begin{proof}
The extension to a non-monotonic $G$ is self-evident upon
reviewing the earlier manipulations.  With that augmentation in 
hand, widening the class of admissible $F$ is most easily
achieved by way of approximation. To this end, \linebreak one 
forms the (Riemann-Stieltjes) convolution 
${\mathcal N}_\varepsilon \ast F$ with
\[
{\mathcal N}_\varepsilon (x) = 
\frac{1}{\varepsilon \sqrt{2\pi}}
 \int^x_{-\infty} e^{-u^2/2\varepsilon^2} du
\]
and then passes to the $\varepsilon =0$ limit in 
(3.1). Compare [Fel2, pp.507 (bottom), 146 (theorem 4)]. 
The pertinent density and characteristic functions are
\[
\frac{1}{\varepsilon \sqrt{2\pi}} \int^\infty_{-\infty}
 e^{-(x-u)^2/2\varepsilon^2} dF(u) \quad 
\mbox{and}\quad\ e^{-\varepsilon^2\zeta^2/2}\varphi (\zeta)\,,
\]
respectively. The $dF$-integral is manifestly 
continuous w.r.t. \!$x\MSB$; one also knows that\vspace*{1.0ex}
\[
{\mathbb E} (|Q+X|^\alpha) \leqq 2^\alpha 
{\mathbb E} (|Q|^\alpha)+2^\alpha {\mathbb E} 
(|X|^\alpha) \vspace*{1.0ex}
\]
for general random variables $Q$ and $X$. That being said, 
to finish the 
proof, one simply notes that
\[
\lim_{\varepsilon \rightarrow 0^+}
\big [({\mathcal N}_\varepsilon \ast F)(t)-F(t)\big ]=
  \lim_{\varepsilon\rightarrow 0^+}
\int_{{\mathbb R}} \big [F(t-v)-F(t)\big ] 
d{\mathcal N}_\varepsilon (v) = 0\vspace*{-0.3ex}
\] 
holds at every point of continuity of $F$. 
Since such $t$ are everywhere dense along 
${\mathbb R}$, \,inequality (3.1) follows with 
the same constants $c_j$ as before. \bx
\end{proof}

Esseen's lemma admits an important 
$\MSA$(and very well-known)$\MSA$   
corollary in connection with 
convergence of probability distributions.

\begin{corollary}
Given $\{G, g, \alpha, \psi\}$ as in Lemma 3.2. Let 
$\{F_n\}^\infty_{n=1}$ be any \linebreak sequence of probability 
distributions on ${\mathbb R}$ for which\\[.63ex] 
(i)\phantom{i}  
\quad $\int_{{\mathbb R}} |x|^\alpha dF_n(x) = O(1)$;\\[.63ex] 
(ii) \quad $ \int_{{\mathbb R}} e^{i\zeta x} dF_n(x) 
\rightarrow \psi(\zeta)$ pointwise on ${\mathbb R}$ as $n 
\rightarrow \infty$.\\[.63ex]
We then have\vspace*{0.63ex}
\[
F_n(x) \rightarrow G(x)\vspace*{0.9ex}
\]
uniformly on ${\mathbb R}$. An analogous result holds 
when $n$ is replaced by a continuous variable $\xi$.
\end{corollary}

\begin{proof}
An immediate consequence of Theorem 3.3 and the 
Lebesgue dominated convergence theorem. Cf. (i) and 
the last three lines in the proof of Lemma 3.2. 
A trivial modification of this shows that the 
characteristic functions $\{\varphi_n(\zeta)\}^
\infty_{n=1}$ 
satisfy a \emph{uniform} H\"{o}lder 
$\widetilde{\alpha}$-condition on ${\mathbb R}$. The convergence 
hypothesized in (ii) will thus be uniform on every 
$\zeta$-interval $[-\Omega, \Omega]$\,. 

To treat $F_\xi$, one can either exploit the H\"{o}lder 
$\widetilde{\alpha}$-condition or simply 
reason by contradiction. \bx
\end{proof}

\subsection{}
In terms of applications, the single most common one for
Corollary 3.4 is undoubtedly its use in establishing  
the central limit theorem for sums of $N$ 
identically distributed, independent random variables
$X_j$.  See, for instance, [Fel2, p.515]. \,In Corollary 5.5 
below, we look at a somewhat more involved case. $\MSAH$(See
[Fel2, pp.542--543] for a prototypical \emph{effective} 
example.)

\subsection{}
Hypothesis (i) in Corollary 3.4 is of course very mild. 
By a change in method, it can simply be expunged. 

To appreciate this, it suffices to recall two standard 
facts from basic probability theory. $\MSA$First: 
that (ii) is tantamount to $F_n$ approaching $G$ weakly, 
i.e., in distribution. $\MSA$Second: that distributional 
convergence $F_n \rightarrow {\mathcal P}$ is automatically 
upgradable to uniform convergence on ${\mathbb R}$ anytime 
the limiting probability distribution 
${\mathcal P}$ is everywhere continuous. (Cf. 
[Fel2, p.285 (problem 5)] or [KeS, \S4.11] as regards
the latter assertion.) $\MSA$Insofar as the ``target'' 
distribution $G$ is continuous, hypothesis (ii) will 
thus furnish both a necessary and sufficient condition 
for the uniform $x$-limit articulated in 
Corollary 3.4.$\MSA$\footnote{Consideration 
of ${\mathcal N}_\varepsilon \ast G$ 
provides a good illustration of what can go 
wrong if $G$ has atoms.}

The downside to switching over to this much   
more rudimentary viewpoint [based ultimately 
on just the Helly selection principle and 
Fourier-Stieltjes \linebreak[2]inversion] is that  
the \emph{rate} of convergence will 
typically not be very transparent.\footnote{$\MSH$It
is worth noting that the technique used in 
[Fel2, pp.257--260$\MSA$(top)] runs into similar issues. 
The second line in (4.7) should be deleted there.}

\section[Esseen's Lemma over ${\mathbb R}^k$]
{The Esseen Smoothing Lemma in Several Variables}

\renewcommand{\theequation}{4.\arabic{equation}}
\setcounter{equation}{0}

\subsection{}
It is only natural to wonder if there exists a reasonably 
simple counterpart of (3.1) 
and Theorem 3.3 \emph{in higher dimensions}.  We examine 
this question\linebreak   
in the present section, placing a premium 
on retaining some measure of formal\linebreak  
similarity.   
Because of certain algebraic difficulties, the matter is not 
as straightforward as one might initially expect. 

To streamline our exposition, it is helpful to begin with several 
preliminaries and a bit of notation.

\subsection{}
When $k>1$, the one-variable inequalities 
$f_j(x_j) \leqq \chi_{{\mathscr A}_j}(x_j)$ do \emph{not} 
in general concatenate w.r.t. $j$ to produce the relation \vspace*{0.5ex}
\[
{\TS \prod^k_{j=1}}\MSC f_j(x_j) \,\leqq \, \chi_{{\mathscr A}}(x_1, \dots, x_k)
\vspace*{0.5ex}
\]
with ${\mathscr A}= {\mathscr A}_1 \times \dots \times {\mathscr A}_k$ 
over ${\mathbb R}^k$. Fortunately, an elementary ring-theoretic 
lemma [whose form I owe to a 1990 conversation with A. Selberg]
enables one to circumvent this difficulty $\MSZ$with $\MSA$minimal 
fuss$\MSA$ in a wide variety of technical 
settings.\footnote{that of 
\S 2 being fairly typical}

\begin{lemma}
Given any integer $k\geqq 2 \MSB$. Define symbolic functions
\[
f_j = \chi_j-\delta_j\MSB , \quad g_j = \chi_j + \varepsilon_j
\]
for $1\leqq j \leqq k$. We then have
\begin{align}\label{4.1}
(1-k)g_1 \cdots g_k \MSC\MSC+\MSC\MSC f_1&g_2 \cdots  g_k   
\MSC\MSC+ \MSC\MSC g_1 f_2 \cdots g_k \MSC\MSC+ \MSE\MSE .\,\,.\,\,. \MSE\MSE
+\MSC\MSC g_1 \cdots g_{k-1} f_k\\[0.5ex] 
& = \chi_1 \cdots \chi_k \MSAH-\MSAH S \MSC, \nonumber
\end{align}
wherein $S$ is a nonempty sum of monomials $\omega_1 \cdots \omega_k$ 
satisfying the conditions\\[0.50ex]
(i)\phantom{iiiii} $\omega_j \in \{\chi_j, \varepsilon_j, \delta_j\}$ 
for each index $j$\,;\\[0.56ex]
(ii)\phantom{iiii} $\omega_\tau \not= \chi_\tau$ for at 
least one $\tau$.\\[0.53ex]
When $\MSA k > 2$, certain of the products   
$\MSA\omega_1 \cdots \omega_k$ will appear in $S$ with a 
multiplicity larger than one$\MSH$; $\MSAH$e.g., $\varepsilon_1 \cdots 
\varepsilon_k \MSA$.

\end{lemma}

\begin{proof}
Put
\\[-1.0ex]
\[
G_j(m) = \left \{ \begin{array}{ll}
\chi_j, & m=0\\[1.24ex]
\varepsilon_j, & m=1
\end{array}\right \}
\]
\\
and then keep $m_j \in \{0, 1\}$. \,Since the LHS of (4.1) is just
\\
\[
g_1 \cdots g_k \,\, + \,\, (f_1 - g_1)g_2 \cdots g_k \,\, 
+ \MSE\MSE .\,\,.\,\,. \MSE\MSE
+ \,\, g_1 \cdots g_{k-1}(f_k - g_k)\,,
\]
\\ 
one is free to re-express things as
\begin{align*}
\sum_{\{1,..., k\}} G_1 (&m_1) \cdots G_k(m_k)\, - \, (\delta_1 + \varepsilon_1) 
\sum_{\{2,..., k\}} G_2 (m_2) \cdots G_k(m_k) -\, \cdots  \\[.31ex]
& - (\delta_k+\varepsilon_k) \sum_{\{1,..., k-1\}} G_1(m_1)
 \cdots G_{k-1}(m_{k-1})\,,
\end{align*}
$\{a,\dots, b\}$ serving here as an obvious shorthand for 
$(m_a,\dots, m_b)$. \vspace*{.31ex}\,But:
\begin{align}\label{4.2}
\sum_{\{1,..., k\}} G_1 (&m_1) \dots G_k (m_k)\\[-.21ex]
  & = \chi_1 \cdots \chi_k  + \sum_{m_1+\cdots+m_k>0} 
G_1(m_1) \cdots G_k (m_k)\,.\nonumber
\\[-2.00ex]\nonumber
\end{align}
The last sum is clearly ``subsumed'' algebraically by
\begin{align*}
& G_1(1) \sum_{\{2,..., k\}} G_2(m_2) \cdots G_k(m_k) + \,\cdots\\[.63ex]
& \phantom{mmi}+ G_k(1) \sum_{\{1,..., k-1\}} G_1(m_1) \cdots 
G_{k-1}(m_{k-1})\\[.63ex]
& \equiv \,\varepsilon_1 \sum_{\{2,..., k\}} G_2(m_2) \cdots 
G_k(m_k) + \,\cdots \\[.63ex]
& \phantom{mmi}+ \varepsilon_k \sum_{\{1,..., k-1\}} G_1(m_1) \cdots 
G_{k-1}(m_{k-1})\, .
\end{align*}
Inserting a ``redacted'' form of the above into (4.2) and 
then backtracking, it quickly becomes apparent that $S$ has 
the properties we've posited for it. $\blacksquare$

\end{proof}

For later use, we'll abbreviate equation (4.2) as
\[
\hspace{10.2em} g_1 \cdots g_k \MSB=\MSB \chi_1 \cdots 
\chi_k \MSB+\MSB \widetilde{S}\,, 
\hspace{9.5em}\mbox{(4.1$\MSB$bis)}
\]
wherein $\widetilde{S}$ is now viewed as an $\{\chi_j,
\varepsilon_j\}$ \!--\! counterpart of the expression $S$ that 
occurs in (4.1).

\subsection{}
Moving onward, suppose next that $n\geqq 1$ and that the 
variables $\{x, v \} \MSA\cup\MSA \{x_j, v_j\}^n_{j=1}$ are all real. 
Keep $\Delta$ and $\Delta_j$ positive, and write
\[
E=[-\Delta, \Delta]\:\: , \,\; E_j = [-\Delta_j, \Delta_j]\,.
\]
Also set
\[
{\mathscr E} = E_1 \times \cdots \times E_n
\]
and let $d\vec{v}$ be an abbreviation for the standard Euclidean 
volume element $dv_1 \cdots dv_n$.

For functions $f \in C(E)$, we introduce a \emph{difference} 
operator $D$ by writing
\begin{equation}\label{4.3}
(Df)(v) = \sfrac{1}{2} [f(v)-f(-v)]\,\,.
\end{equation}
For $G \in C({\mathscr E})$, we then let $D_jG$ denote the 
obvious partial difference. It is easily seen that the 
operators $D_j$ commute; moreover, $D^2_j = D_j$. The 
vector space $C({\mathscr E})$ therefore splits into a direct 
sum of $2^n$ simultaneous $D_j$ \!--\! eigenspaces $C({\mathscr E})_\sigma$,
wherein $\sigma$ corresponds to $\MSB\prod^n_{j=1} \{0, 1\}\MSB$ in 
an obvious fashion.

On a related note, recall that a function $h \in C(E)$ 
is said to be \emph{Hermitian} when $h(-x)= \overline{h(x)}\MSB$. 
Similarly for $H \in C({\mathscr E})$; here
\[
H(-x_1,\dots, -x_n) = \overline{H(x_1,\dots, x_n)}\,.
\]
By the symmetry properties of $d\vec{v}$, one immediately sees that
\[
\int_{{\mathscr E}} H(v_1, \dots, v_n) d\vec{v} \,\in\, {\mathbb R}\,.
\]
Similarly for $h$.

A moment's thought shows that functions in $C({\mathscr E})$ are 
Hermitian precisely when their $\sigma$ \!--\! components in 
$C({\mathscr E})$  
$(G_\sigma$, say) are respectively either real or purely 
imaginary depending on the parity of 
$\|\sigma\| \equiv \sigma_1 + \cdots + \sigma_n$. 

In line with this, it is also helpful to observe that:\\[.31ex]
(i) the characteristic function $\varphi(v_1,\dots, v_n)$ 
of any probability distribution $F$ (or totally finite 
signed measure $\eta$) on ${\mathbb R}^n$ is 
automatically Hermitian;\\[.45ex]
(ii) likewise for any ${\mathbb R}$-linear combination 
$(D_{\iota_1} ... D_{\iota_r})\varphi$ with $\iota_m \uparrow$ 
and $1\leqq r \leqq n$;\\[.45ex]
(iii) and -- yet again -- for the functions 
$\frac{1}{\Delta} R_B \big (\frac{w}{\Delta} \big )$ and 
$\frac{1}{\Delta} R_b \big (\frac{w}{\Delta}\big )$ in 
$C(E)$ that arose in the proof of (3.1).

\subsection{}
Our subsequent continuity considerations and crude estimates 
for functions of types (i) and (ii) will generally rest
either implicitly or explicitly on some combination of 
three basic inequalities; \,viz.,
\begin{align}\label{4.4}
|e^{2i(\tau+\phi)} &- e^{2i\phi}| \MSA=\MSA |e^{2i\tau}-1| \MSA=\MSA 
2|\sin \tau| \MSA \leqq \MSA
 2\min (1, |\tau|) \leqq 2\min (1, |\tau|^\omega)\MSB,\\[1.88ex]
& \MSE {\TS \prod^n_{j=1}} (|x_j| + |y_j|)^{\beta_j} \leqq (\|X\|_\infty
+ \|Y\|_\infty)^{\beta_1+ \cdots + \beta_n} \MSC, \\[1.88ex]
  |{\TS \prod^n_{j=1}}\MSA &w_j \MSC-\MSC {\TS \prod^n_{j=1}}\MSA \xi_j| 
\MSA\leqq\MSA {\TS \sum^n_{j=1}} |w_j-\xi_j| \quad \mbox{\,whenever\,} 
\quad |w_j| \MSC,\MSA |\xi_j| \leqq  1\MSD .
\end{align}
It is understood herein that $\phi$ and $\tau$ are real, $0 < \omega < 1$, 
and $\beta_j \geqq 0$.

\subsection{}
Finally, in connection with both ${\mathbb R}$ and 
${\mathbb R}^n$, it is convenient to have at one's disposal 
a symbolic Dirac delta function $\delta(\cdot)$ such that, 
for any $r \in [1, n]$ and $G \in C({\mathscr E})$,
\begin{align}\label{4.7}
\int_{{\mathscr E}} \delta(v_1&) \cdots \delta(v_r) 
G(v_1,\dots, v_n) \MSA d\vec{v}\\
& = \left \{ \begin{array}{ll}
G(0,\dots, 0) & \mbox{if $r=n$}\\[1.50ex]
\int_{E_{r+1} \times \cdots \times E_n} G(0,\dots, 0, v_{r+1},\dots, v_n) \MSA 
\TS \prod^n_{j=r+1} dv_j & \mbox{if $r<n$}
\end{array}\right \} . \nonumber
\end{align}
Similarly for $f \in C(E)$ and permuted sets of indices.  Also  
for functions $\widetilde{G}$ whose continuity is assumed 
only over, e.g.,
\[
\widetilde{{\mathscr E}} \equiv {{\mathscr E}} \cap \{v_{r+1} \cdots v_n \not= 0\},
\]
the behavior elsewhere being such that one at least has
\begin{equation}\label{4.8}
\sup_{E_1 \times \cdots \times E_r} |\widetilde{G} (\xi_1, \dots, 
\xi_r;\, v_{r+1},\dots, v_n)| \in L_1 (E_{r+1} 
\times \cdots \times E_n)\,,
\end{equation}
the latter condition serving to ensure that the 
$dv_{r+1} \cdots dv_n$ ``cross-sectional integral'' of 
$\widetilde{G}$ remains nicely continuous as the level 
$(\xi_1, \dots, \xi_r)$ varies.

\subsection{}
We are now able to state a prototypical \vspace*{-.18ex}
version of our $\MSH k \MSB$-$\MSAH$variable 
counterpart of $\MSA$(3.1)$\MSA$. 

Matters are facilitated in this by letting ${\mathscr P}$ 
denote a generic partition of $\{1,2,..., k\}$ into three 
disjoint \,(possibly empty!) subsets $\langle {\mathscr B}, 
{\mathscr C}, {\mathscr D}\rangle$ and then writing 
\\[-1.5ex]
\[
D_{{\mathscr C}} \MSA G \equiv \big ( 
\TS \prod_{j\in {\mathscr C}} D_j \big )\MSA G\,.
\] 
\\[-0.9ex]
(Here and below, we follow the standard convention that 
empty products are understood to be either 1 or the 
identity operator. When the context is clear, we also 
agree that `` \raisebox{-.56ex}{$\widehat{\phantom{n}}$} ''  
atop anything signifies expungement.)

\begin{theorem}
Given \,$k \geqq 1$ and any list of positive numbers  
$\{\Omega_1 ,\dots,\Omega_k \}$. \linebreak
Let $F$ be any probability distribution 
on ${\mathbb R}^k$ and $\MSB$$dG= g(x_1,\dots, x_k)\MSA d\vec{x}\MSB$
any totally finite \emph{signed} measure having
$g\in C({\mathbb R}^k)$. \,Suppose that
\begin{equation}\label{4.9}
\int_{{\mathbb R}^k} \big (\max_{1\leqq j\leqq k}
 |x_j| \big )^\alpha \MSA dF < \infty\, , \MSD\MSD 
\int_{{\mathbb R}^k} \big (\max_{1\leqq j\leqq k}
 |x_j| \big)^\alpha
\MSA |g(x_1,\dots, x_k)|\MSA d\vec{x} < \infty
\end{equation}
for some $\alpha > 0$ and that, in addition,
\begin{equation}\label{4.10}
\int_{{\mathbb R}^{k-1}} |g(x_1,\dots, x_k)|\MSA  dx_1 \cdots 
\widehat{dx_\ell} \cdots dx_k \leqq m_\ell
\end{equation}
for every $\ell \in [1,k]$ (the integral sign simply 
being absent if $k=1$). Let $\MSB\varphi (\vec{v})\MSB$ and 
$\MSB\psi(\vec{v})\MSB$ be the usual multivariate characteristic
 functions of $F$ and $dG\MSA$. (Cf. \S 3.1.)  With
\begin{equation}\label{4.11}
G(y_1,\dots, y_k) \equiv \int^{y_1}_{-\infty} \cdots 
\int^{y_k}_{-\infty} g(\xi_1,\dots, \xi_k) \MSA d\vec{\xi}\,,
\end{equation}
there then exist$\MSD$\footnote{$\MSAH$explicitly computable} positive 
constants $c_1$ and $c_2$ depending 
solely on $k$ such that one has the a priori bound
\begin{align}\label{4.12}
&|F(t_1,\dots, t_k) - G(t_1,\dots, t_k)|\\[2.2ex]
&\,\leqq\,c_1 \sum_{{\mathscr P}} \int^{\Omega_1}_{-\Omega_1} \cdots 
\int^{\Omega_k}_{-\Omega_k} |D_{{\mathscr C}} (\varphi-\psi)|
\prod_{j\in{\mathscr C}} \sfrac{1}{|v_j|}
\prod_{j\in {\mathscr D}} 
\Big ( \sfrac{1}{\Omega_j} + 
\sfrac{|\sin (t_jv_j)|}{|v_j|}\Big ) 
\prod_{j\in {\mathscr B}} \delta(v_j)d\vec{v}\nonumber\\[.31ex]
& \quad\MSE +\MSA c_2 
\sum^k_{\ell=1} \msfrac{m_\ell}{\Omega_\ell}\vspace*{-.46ex}\nonumber
\end{align}
at every point $(t_1,\dots, t_k)\in {\mathbb R}^k$. (Though (4.12)
 does \emph{not} require $\psi (\vec{0}) = 1$, cases having 
$\psi (\vec{0})\not= 1$ will generally be of at most 
subsidiary interest from a probabilistic standpoint; 
cf. the term with ${\mathscr C} = {\mathscr D} = \phi$.)
\end{theorem}

\begin{proof}
Apart from some additional bookkeeping, the basic procedure 
is largely a mimic of the approach that was used for Lemma 3.2 
and Theorem 3.3. 

We start by taking $\alpha < 1$ w.l.o.g and writing 
$\beta=\alpha/k$. 
By (4.4) and a trivial manipulation, \pagebreak[4]
\begin{align}\label{4.13}
& |\sin u| \leqq \min \{1, |u|\} \leqq \min 
\{1, |u|^\beta\}\MSA;\nonumber\\[1.50ex]
D_{{\mathscr C}} \varphi &= \int_{{\mathbb R}^k} {\TS \prod_{j\in {\mathscr C}}}
 \MSA \big (i \sin (v_jx_j) \big )\MSA {\TS \prod_{j\in {\mathscr B}\cup {\mathscr D}}}
\MSB e^{i v_jx_j} \MSB dF(\vec{x})\,.
\end{align}
(Similarly for $D_{{\mathscr C}}\psi$.) Relation (4.8) is thus fulfilled 
by a large margin for every ${\mathscr P}$ \!--\! summand in (4.12)$\MSA$; cf. 
(4.9) and \S 4.5. Notice {\it too} that there is no loss
 of generality if a large positive constant ${\mathfrak N}$ is 
added to each ``max'' in (4.9).

It will be convenient to first prove (4.12) in the case 
where $(t_j) = (0)$ and $F$ has a continuous density 
function $f(x_1,\dots, x_k)$.
To this end, we write
\[
\chi(x) = \chi_{(-\infty, 0]}(x)
\]
and then use Theorem 2.2 to get
\[
\sfrac{1}{2} [1-B (\Omega x)] \leqq \chi(x) \leqq 
\sfrac{1}{2} [1-b(\Omega x)]
\]
for each $\Omega > 0$. This quickly leads to
\begin{equation}\label{4.14}
\widetilde{f}_\Omega (x) \leqq \chi (x) \leqq \widetilde{g}_\Omega (x)
\end{equation}
with
\begin{align*}
\widetilde{f}_\Omega(x) & \equiv - \msfrac{1}{2}
 \oint^\Omega_{-\Omega} \msfrac{D_w(e^{2\pi ixw})}{\pi i w}\MSA dw +
\msfrac{1}{2} \int^\Omega_{-\Omega}
e^{2\pi ixw} \big [\delta (w) - \frac{1}{\Omega}
R_B  \big (\msfrac{w}{\Omega}\big )\big ]\MSA dw\\[1.5ex]
\widetilde{g}_\Omega(x) & \equiv -\msfrac{1}{2}
\oint^\Omega_{-\Omega} \, \msfrac{D_w(e^{2\pi ixw})}{\pi iw} \MSA dw
+ \msfrac{1}{2} \int^\Omega_{-\Omega} 
e^{2\pi ixw} \big [ \delta (w) - \msfrac{1}{\Omega} R_b 
\big ( \frac{w}{\Omega}\big ) \big ] \MSA dw
\end{align*}
thanks to equation (3.3).  
$\MSB$By Theorem 2.2(a), one knows that
\[
\widetilde{g}_\Omega (x) - \widetilde{f}_\Omega (x) = K(\Omega x)\,.
\]
The situation of Lemma 4.3 is thus applicable with
\begin{equation}\label{4.15}
\left \{\begin{array}{l}
\chi_j = \chi (x_j), \MSE f_j = \widetilde{f}_{\Omega_j}(x_j),\MSE 
g_j = \widetilde{g}_{\Omega_j}(x_j)\\[1.55ex]
0 \leqq \delta_j \leqq K(\Omega_j x_j), \MSE 0 \leqq 
\varepsilon_j \leqq K(\Omega_j x_j)
\end{array}\right \} .
\end{equation}

To continue, we let
\begin{align*}
T^- ( x_1,\dots, x_k) & = \mbox{the LHS of (4.1)}\\[.46ex]
T^+ (x_1,\dots,x_k) & = g_1 \cdots g_k
\end{align*}
and then exploit the fact that\\
\begin{equation}\label{4.16}
T^{-} + S \,=\, \chi_1 \cdots \chi_k \,=\, T^{+} - \widetilde{S}
\end{equation} \\
in accordance with (4.1) and (4.1$\MSB$bis). Since
\[
F(\vec{0}) - G(\vec{0}) = \int_{{\mathbb R}^k} (f-g) \MSB\chi_1 \cdots
 \chi_k \MSB d\vec{x}
\]
and $f \geqq 0$, one immediately obtains
\begin{align*}
\hspace{5.08em}F(\vec{0}) - G(\vec{0}) 
& \geqq \MSB\int_{{\mathbb R}^k} (f-g)
T^- \MSA d\vec{x} \MSB-\MSB \int_{{\mathbb R}^k} 
|g| S\MSA d\vec{x}\hspace{5.08em} 
& \mbox{(4.17a)}\\[.93ex]
F(\vec{0})-G(\vec{0}) 
& \leqq \MSB\int_{{\mathbb R}^k} (f-g)
T^+ \MSA d\vec{x} \MSB+\MSB \int_{{\mathbb R}^k} 
|g| \widetilde{S} \MSA d\vec{x}\MSB. 
& \mbox{(4.17b)}
\end{align*}
Each $g_j$ or $f_j$ appearing in $T^\pm$ can and \emph{will} 
be viewed as the sum of three obvious chunks; cf. the formulae
following (4.14).

Prior to pushing onward with this, we note that, by (4.10), 
(4.15), and \linebreak the algebraic format of $S$ and 
$\widetilde{S}$, the $\MSA$final$\MSA$ terms in (4.17a) and (4.17b) 
are \linebreak manifestly dominated by
\setcounter{equation}{17}
\begin{equation}\label{4.18}
\sum^k_{\ell=1} c_3 \int^\infty_{-\infty} 
\MSA K(\Omega_\ell x_\ell)\MSA m_\ell \MSA dx_\ell 
\MSE=\MSE  c_3 \sum^k_{\ell=1} \, \msfrac{m_\ell}{\Omega_\ell}
\end{equation}
for some choice of $c_3$ ($= c_3(k)$).

The central issue now becomes one of substitution and 
straightforward justification of the applicability of 
Fubini's theorem. For the latter, it suffices to 
combine (4.9), (4.5), and our earlier observations 
about $|\sin u|$ and ${\mathfrak N}$.  $\MSD$Cf. also 
here the discussion following (3.3), especially 
as regards the uniform boundedness of the integral
\\[-1.0ex]
\[
\int^{\gamma_2}_{\gamma_1} \,\, 
\msfrac{\sin (2\pi xw)}{\pi w} \MSA dw
\]
\\[-1.0ex]
and related use of the Lebesgue dominated convergence theorem.

On the basis of these remarks, it quickly becomes apparent 
that all relevant manipulations 
for ``$\MSB$case $\vec{0}\MSC$'' are 
easily carried out --- in fact, separately 
so for both $\varphi$ and $\psi$, and with 
\emph{good majorants} throughout. One clearly obtains 
(4.12) with $\varphi(2\pi \vec{w}),\, \psi (2\pi \vec{w})$ 
in place of $\varphi(\vec{v}), \, \psi (\vec{v})$. To finish 
up format-wise, one simply replaces each $\Omega_j$ by 
$\Omega_j/2\pi$ and writes $v_j = 2\pi w_j$.

Keeping $dF=fd\vec{x}\MSA$, we turn next to the case of an 
$\MSA$arbitrary$\MSA$ point $(t_j) \in {\mathbb R}^k$. One 
freezes $(t_j)$ and introduces the new densities
\[
f^\ast (x_1,\ldots, x_k) = f(t_1+x_1,\ldots, t_k+x_k), \MSE 
g^\ast (x_1,\ldots, x_k)= g(t_1+x_1,\ldots, t_k+x_k).
\]
Conditions (4.9) and (4.10) are still satisfied. Working 
separately with the analogs of (4.17a) and (4.17b) in 
accordance with the first two sentences of the 
preceeding paragraph [cf. also here (4.18)], it is 
immediately evident that, up to trivial 
${\mathbb Q}$-coefficients, one obtains a 
specific sum of $d\vec{w}$ \!--\! integrals over
\[
{\mathscr E} = [-\Omega_1, \Omega_1] 
\times \cdots \times [-\Omega_k,\Omega_k]
\]
whose \emph{integrands} have the form
\begin{align*}
& \prod_{j\in {\mathscr C}} \msfrac{1}{\pi w_j}
 \prod_{j \in {\mathscr D}}
\msfrac{1}{\Omega_j} R_\ast \Big (\msfrac{w_j}{\Omega_j}\Big )
\prod_{j\in {\mathscr B}} \delta (w_j)\\[.62ex]
& \quad \cdot \int_{{\mathbb R}^k} 
\prod_{j\in {\mathscr C}} \sin (2\pi w_j x_j)\MSA 
\prod_{j\notin {\mathscr C}} e^{2\pi i w_j x_j} \MSB
(f^\ast-g^\ast)\MSA d\vec{x}\,.
\end{align*}
(Cf. (4.14).) Writing $x_j= y_j-t_j$, each such integrand 
then becomes, again up to trivial 
${\mathbb Q}$-coefficients, a \emph{specific} sum of \vspace*{.75ex} 
expressions like
\begin{align}\label{4.19}
& \prod_{j\in {\mathscr D}} \msfrac{1}{\Omega_j} R_\ast
\Big (\msfrac{w_j}{\Omega_j}\Big )
\prod_{j\in {\mathscr B}} \delta (w_j)
\prod_{j\notin {\mathscr C}} e^{-2\pi i w_j t_j}\\[.62ex]
& \quad \cdot \int_{{\mathbb R}^k} \prod_{j\in {\mathscr C}_1}
\Big ( \msfrac{\sin (2\pi w_j y_j)}{\pi w_j}\Big )
\cos (2\pi w_jt_j) \prod_{j\in {\mathscr C}_2}
\Big ( \msfrac{\sin (2\pi w_j t_j)}{\pi w_j} \Big )
\cos (2\pi w_j y_j)\nonumber\\[.62ex]
& \quad\quad\quad \cdot \prod_{j\notin {\mathscr C}}
e^{2\pi i w_j y_j} \MSB (f-g) \MSA d\vec{y}\,,
\vspace*{-.46ex}\nonumber
\end{align}
wherein $\langle {\mathscr C}_1, {\mathscr C}_2\rangle$ 
is a partition of ${\mathscr C}$. Since 
$\max\{|y_j|, |t_j|\}\leqq |y_j| + |t_j|$ 
for $j \in {\mathscr C}$ and
\\
\[
\int_{{\mathbb R}^k} \big ( 1+ \|\vec{y}\MSA\|_\infty +
\|\vec{t}\MSB\|_\infty \big )^\alpha (f+|g|)\MSA d\vec{y} < +\, \infty\,,
\]
\\
each of these new integrands clearly admits a good 
majorant over the set ${\mathscr E}$.
Observe now that one can successively ``flip'' each 
$\cos (2\pi w_j y_j)$ (with $j \in {\mathscr C}_2$) 
into $\exp(2\pi i w_j y_j)$ without changing the 
\emph{numerical} value of the corresponding iterated 
$d\vec{w}$ \!--\! integral over ${\mathscr E}$. Similarly for 
$\cos (2\pi w_j t_j)$ with $j \in {\mathscr C}_1$.

The upshot of course is that each expression in 
(4.19) can thus be replaced by
\begin{align*}
& \prod_{j\in {\mathscr D}} \msfrac{1}{\Omega_j}
R_\ast \Big ( \msfrac{w_j}{\Omega_j}\Big )
\prod_{j\in {\mathscr B}} \delta (w_j)
\prod_{j\notin {\mathscr C}} 
e^{-2\pi i w_j t_j}\\[.62ex]
& \quad \cdot \int_{{\mathbb R} ^k} \prod_{j\in {\mathscr C}_1}
\Big (\msfrac{\sin (2\pi w_j y_j)}{\pi w_j} \Big )
e^{-2\pi i w_j t_j} \prod_{j\in {\mathscr C}_2}
\Big (\msfrac{\sin (2\pi w_j t_j)}{\pi w_j}\Big )\\[.62ex]
& \quad \quad \quad \cdot \prod_{j\in {\mathscr C}_2}
e^{2\pi i w_j y_j} \prod_{j \notin {\mathscr C}}
e^{2\pi i w_j y_j} \MSB (f-g)\MSA d\vec{y}\,.
\end{align*}
It makes sense to rewrite this \vspace*{.62ex} as
\begin{align}\label{4.20}
& \prod_{j\in {\mathscr D}} \msfrac{1}{\Omega_j}
R_\ast \Big (\msfrac{w_j}{\Omega_j} \Big )
\prod_{j \in {\mathscr B}} \delta (w_j)
\prod_{j\notin {\mathscr C}_2}
e^{-2\pi i w_j t_j} \prod_{j\in {\mathscr C}_2}
\Big ( \msfrac{\sin (2\pi w_j t_j)}{\pi w_j} \Big )\\[.62ex]
&\quad  \cdot \int_{{\mathbb R}^k} \prod_{j\in {\mathscr C}_1}
\Big (\msfrac{i \sin (2\pi w_j y_j)}{i\pi w_j}\Big )
\prod_{j\notin {\mathscr C}_1}
e^{2\pi i w_j y_j} \MSB(f-g)\MSA d\vec{y}\,, \nonumber
\end{align}
the $d\vec{y}$ \!--\! integral herein simply being
\\
\[
\msfrac{D_{{\mathscr C}_1} (\varphi-\psi)(2\pi \vec{w})}
{\prod_{j\in {\mathscr C}_1} (i\pi w_j)}\,\,.
\]
\\
Notice incidentally that (4.20)  is 
manifestly Hermitian w.r.t. $(w_1,\dots, w_k)$.

Writing ${\mathscr D}_a = {\mathscr D} \cup {\mathscr C}_2$, the 
corresponding contribution to the overall estimate for 
$|F(t_j)-G(t_j)|$ thus becomes some real 
number having absolute value\pagebreak
\vspace*{.31ex}
\[
\leqq c_4 \int_{{\mathscr E}} 
\prod_{j \in {\mathscr D}_a} 
\Big ( \msfrac{1}{2\pi \Omega_j} + 
\msfrac{|\sin (2\pi w_jt_j)|}{2\pi |w_j|}\Big )
\prod_{j\in {\mathscr B}} \delta (w_j)
\cdot \msfrac{|D_{{\mathscr C}_1}(\varphi-\psi)(2\pi \vec{w})|}
{\prod_{j\in {\mathscr C}_1} |2\pi w_j|}\,\, d\vec{w}.
\]
\\
Needless to say: $c_4 = c_4(k)$ and 
$\{1,2,\dots, k\}\,=\, {\mathscr C}_1 
\cup {\mathscr D}_a \cup {\mathscr B}$\,.

Upon taking $v_j = 2\pi w_j$, $\MSAH$temporarily setting 
$\Omega_j = E_j/2 \pi$, and looking at \vspace*{-0.62ex}
\[
{\mathscr P}^\ast \,=\, \langle {\mathscr B}, 
{\mathscr C}_1, {\mathscr D}_a\rangle \vspace*{-0.62ex}
\]
in lieu of ${\mathscr P}$, relation (4.12) follows.\footnote{
Note that the form of the product for $j\in {\mathscr D}_a$
 already accounts for all \emph{choices} of ${\mathscr C}_2$
 associated with a given ${\mathscr D}_a$.}

To establish (4.12) for a perfectly general $F$, one simply 
convolves $F$ with a multivariate Gaussian 
(${\mathcal N}^{\MSA\langle k \rangle}_\varepsilon$, say) 
having density \vspace*{-0.62ex}
\[
\big ( \msfrac{1}{\varepsilon \sqrt{2\pi}} \big )^k \,\,
{\TS \prod^k_{j=1}}\, \exp 
\big ( - \msfrac{{u_j}^2}{2 \MSB \varepsilon^2}\big ) \vspace*{-0.62ex}
\]
and then lets $\varepsilon \rightarrow 0$ in obvious 
analogy to what was done earlier for $k=1$ and 
Theorem 3.3. \,Since \vspace*{1.08ex}
\[
{\TS \prod^k_{j=1}}\, \exp 
\big (-\sfrac{1}{2}\MSA\varepsilon^2\MSA {v_j}^2 \big ) \vspace*{.93ex}
\]
is \emph{even} w.r.t. $\MSZ\MSZ$each $v_j$, $\MSD$the action of 
$D_{\mathscr C}$ is trivially visualized and there is
 no difficulty securing a good majorant for each
${\mathscr P}$ \!--\! summand over ${\mathscr E}$. $\blacksquare$ 
\vspace*{-1.9ex}
\end{proof}

\subsection{}
When $k=1$, (4.12) is readily seen to provide
a very slight improvement in (3.1)\, [apart 
from choice of constants]. 
For $k \geqq 2$, however, matters are less satisfactory. 
There (4.12) turns out to have two features that 
tend to be somewhat {\it problematic} for purposes of applications.

The most egregious of these is the fact that, when $k>1$, 
the RHS of (4.12) typically tends to blow up anytime 
one or more of the entries in $(t_1,\dots,t_k)$ 
diverges to $\pm \MSA\infty$.

The difficulty stems from the ``sine'' portion of the 
${\mathscr D}$-terms. It suffices to look at the 
partition for which ${\mathscr D}=\{1,2,..., k\}$. 
Assume, for simplicity, that $\alpha$ can be taken
very large in (4.9). 
After utilizing the well-known relation\vspace*{-.125ex} 
\begin{equation}\label{4.21}
e^{iy} = \sum^N_{j=0} \frac{(iy)^j}{j!} + \Theta\, 
\frac{|y|^{N+1}}{(N+1)!}\vspace*{-.125ex}
\quad \mbox{(for $y \in {\mathbb R}$)}\MSA\footnote{
Cf. equation (5.1) below.}  
\end{equation}
to form obvious Taylor polynomials and rescaling things a bit, 
the operative ``inflationary'' mechanism is most easily 
appreciated [in a model computation 
with $k=3\MSA$] $\MSA$by contemplating
the fact that one has \vspace*{.26ex}
\begin{align*}
& \int^1_{1/T_1} \int^1_{1/B_2} \int^1_{1/B_3}
[u^c + v^c + w^c - \lambda (u^{\ell+1} +
v^{\ell+1} + w^{\ell+1})]\,\,
\frac{dw dv du}{w v u}\\[.62ex]
& \quad \geqq \int^1_{1/T_1} \int^1_{1/B_2}
\int^1_{1/B_3} [u^\ell + v^\ell + w^\ell - (\lambda u) u^\ell -
(\lambda v) v^\ell - (\lambda w) w^\ell]\,\,
\frac{dw dv du}{wvu}\\[.62ex]
& \quad \geqq \frac{1}{2} \int^1_{1/T_1}
\int^1_{1/B_2} \int^1_{1/B_3} 
\frac{u^\ell + v^\ell + w^\ell}{uvw} \,\,dw dv du\\[.62ex]
& \quad \geqq \frac{
1-2^{-\ell}}{2\ell}\,\,
[(\log B_2) (\log B_3) + (\log T_1)(\log B_3) +
(\log T_1) (\log B_2)] \vspace*{.62ex}
\end{align*}
anytime $T_1 \geqq 2, \,B_2 \geqq 2, \,B_3 \geqq 2,\,\, 
1\leqq c \leqq \ell, \, \lambda \in [0,1/2]$. \,Here
$T_1$ corresponds essentially to $|t_1|$.

The second snag is more subtle and -- as will be seen
momentarily -- has a  
predominantly ``operational$\MSH$'' nature.   
$\MSB$It arises in connection with those entries 
\[
\frac{D_{\mathscr C}(\varphi - \psi)}
{\prod_{j\in {\mathscr C}} v_j}\qquad 
(\equiv Q_{\mathscr C})
\]
in (4.12) for which $\card ({\mathscr C})\geqq 2$.  
$\MSB$We take ${\mathscr C} = \{1, \ldots,m \}$ 
$\MSA$ w.l.o.g. (and continue to assume $k > 1$). To 
explicate matters\footnote{(some readers may prefer 
to merely $\MSZ\MSZ$\emph{skim}$\MSZ$ the details 
that follow)}, it is 
helpful to first observe that
\begin{equation}\label{4.22} 
(D_1 \cdots D_m)f = \sfrac{1}{2^m}
(v_1 \cdots v_m)\MSB\int_{[-1,1]^m}\MSB
{\mathcal D}^{[m]} f (v_1 u_1,\ldots ,v_m u_m)\,d\vec{u}
\end{equation}
holds whenever $f \in C^m$. $\MSB$Here 
${\mathcal D}^{[m]} f$ is the obvious mixed partial.  
(There \linebreak is clearly no harm in suppressing, as we do, any variables 
that are inactive. \linebreak Notice too 
that the $d\vec{u}$-integral is  
automatically \emph{even} w.r.t. each $v_j$.)

Suppose now that $\vec{v}$ is situated in that portion of
$\prod^k_{j=1}\,[-\Omega_j,\Omega_j]$ on which
\[
\{|v_1| < \tau, \ldots, |v_m| < \tau \}\MSD .
\]
Here $\tau$ is some suitably small constant. \vspace*{.25ex} 
Think of $\|\MSA\vec{t}\MSAH\|$ as being bounded. The need to   
avoid spurious divergences in (4.12) 
growing out of \vspace*{1.5ex}
\small
\[
\int^{\tau}_0 \cdots \int^{\tau}_0 
\frac{\lambda ( v_1^{\ell+1} + \cdots + v_m^{\ell+1}) }
{v_1 \cdots v_m} \,\,
dv_m \cdots dv_1 = +\infty \MSE\MSE\MSC (\lambda >0)
\vspace*{1.5ex}
\]
\normalsize
when relation (4.21) is applied 
crudely over ${\mathbb R}^m$ 
strongly suggests that any Taylor approximations 
for 
$D_{\mathscr C} (\varphi - \psi)$ w.r.t.$\MSZ\MSZ\MSZ\MSZ$ 
$(v_1,\ldots,v_m)$ $\MSH$that one proposes to exploit  
$\MSB$be created$\MSH\MSB$ in a ``pre-factored'' $\MSZ$format 
$\MSZ$by means of \emph{either} (4.22)+(4.21) 
or $\MSH$else$\MSH$ (4.13)$\MSB$+$\MSD$a variant of (4.21) 
$\MSB$[$\MSAH$call it$\MSH$ (4.21)$^\sharp\MSB$]$\MSB$ which 
refers to the product function \vspace{.05ex} 
\[
\frac{\sin(y_1)}{y_1} \cdots \frac{\sin(y_m)}{y_m}
\MSD\MSD .\vspace*{.10ex}
\] 
The remainder term in relation (4.21)$^\sharp\MSA$ is readily
seen (after a bit of manipulation 
with \pagebreak multinomial expansions) to 
have magnitude at most\vspace*{-2.50ex}
\[ 
O_m (1)\MSH (|y_1| + \cdots + |y_m|)^{N+1} \vspace*{.10ex}
\]
for $\MSA y_\nu \in {\mathbb R}\MSA$ once 
the total $\vec{y}\MSD$--$\MSD$degree pushes 
past $N$. A closer analysis reveals that any \emph{fixed}
number $D \geqq 1$ can be inserted as a denominator in
${\TS \sum^{m}_{j=1}}\MSA |y_j|$. 

To ensure that everything remains well-defined here 
[$\MSH$in the first approach as well$\MSH$], one tacitly 
assumes in the foregoing that $\MSH$(4.9)$\MSH$ holds 
with some $\alpha \geqq m + N + 1$.

Matters receive a useful clarification when Taylor's
theorem with remainder written in \emph{integral} form   
is brought into the picture. (Recall that the
${\mathbb R}^m$-version of Taylor's theorem is an easy
consequence of the corresponding result 
over ${\mathbb R}^1$.)
By writing things out for the function 
$G = \prod^m_{j=1} \exp (i\xi_j)$, 
putting $\xi_j = y_j u_j$, and then averaging
w.r.t. $\vec{u}$ over $[-1,1]^{m}$, an \emph{explicit} form of
(4.21)$^\sharp$ immediately ensues.  The essential point here,
of course, is that
\begin{equation}\label{4.23}
\Big( \frac{\sin y}{y} \Big )^{\MSZ(a)} = \msfrac{1}{2} \MSA
\int^1_{-1} \MSA (iu)^a\MSA e^{i y u}\MSB du \quad
\mbox{for $a \geqq 0$}\MSC .
\end{equation}
The equation for $D_{\mathscr C}\varphi / (v_1 \cdots v_m)$ 
produced by this version of (4.21)$^\sharp$ is 
clear-{\linebreak}ly just
the earlier Taylor $\MSAH$manipulation$\MSAH$      
repeated for the augmented function  \linebreak 
$\prod^m_{j=1} \exp (i\xi_j)\MSZ\MSZ
\cdot\MSZ\MSZ\Omega\MSAH$, $\MSB$wherein $\xi_j = (v_j x_j)u_j$
and
\[
\Omega = \TS\prod^k_{j=m+1} e^{i v_j x_j}\cdot
(ix_1)\cdots (ix_m)\MSA dF,
\]
followed by an integration w.r.t. $(\vec{u},\vec{x})$, $\MSB$the part
involving $\MSB\vec{x}\MSB$ being done \emph{last}.  

Since ${\mathcal D}^{[m]} \varphi (w_1, \ldots,
w_m \MSA ;\MSA v_{m+1},\ldots ,v_k)$ can be viewed as a
continuous superposition of $\MSB\prod^m_{j=1} 
\exp(i w_j x_j)\MSZ\MSHZ\cdot\MSHZ\MSZ\Omega\MSB$, it is hardly 
surprising that formula 
(4.22) with $f = \varphi$$\MSC$ appears when exactly 
the same procedure is followed, but the integra-{\linebreak}tion 
w.r.t. $\MSH(\vec{u},\vec{x})\MSB$ is done in reverse order. 
Integrals having the form
\[
\int_{{\mathbb R}^k} \MSA (ix_1)^{a_1} e^{i w_1 x_1} \cdots
(ix_m)^{a_m} e^{i w_m x_m}\MSZ\MSZ\cdot\MSZ\MSZ\Omega
\]
are merely partial derivatives 
of ${\mathcal D}^{[m]}\varphi\MSH$;  $\MSD$as such, it is easily seen
that the Taylor developments of 
$D_{{\mathscr C}}\varphi /\MSC(v_1 \cdots v_m)$ obtained 
by first integrating the one attached to 
(4.21)$^\sharp$ and, then, that of ${\mathcal D}^{[m]}\varphi$
$\MSA$(\`{a} la (4.22)) actually have summands that are 
numerically identical {\it term-by-term}. 
\,Methods 1 and 2 for handling $Q_{{\mathscr C}}$ over the given
portion \vspace*{-.10ex} of $\prod^k_{j=1} [-\Omega_j, \Omega_j]$ 
are thus {\it equivalent}, at least to the extent that the 
sharper version of (4.21)$^\sharp$ is used.  

When $F$ is fixed, none of this presents any serious difficulty.
In settings, however, where $\varphi$ is basically given
as a product of many (scaled down, more \linebreak basic) 
characteristic functions $\varphi_{\nu}$, making any 
effective use 
of this multiplica-{\linebreak}tivity  
vis \`{a} vis $D_{\mathscr C} \varphi /\MSA(v_1 \cdots v_m)$ 
requires a bit of thought, lest unwelcome 
spuri-{\linebreak}ous divergences re-appear.

On a practical level, 
for run-of-the-mill $\varphi_{\nu}$, 
one finds two main options.  
\linebreak First:  viewing $F$ as an $M$-fold convolution, 
one can seek to  
obtain some sort of $\MSA$\emph{a priori}$\MSA$ hold 
on the magnitudes of   
appropriately large (4.9)$\MSB$-$\MSB$style 
moments of $\MSA dF \MSA$.  $\MSE$In certain settings, 
such information may be relatively easy to obtain by \linebreak 
means of a bit of algebra. $\MSC$(With bounds of any 
reasonable quality here, use of even the cruder version     
of $\MSAH$(4.21)$^{\sharp}\MSH$ may already 
turn out to be sufficient for one's needs.)\pagebreak 

The second option would be to exploit method 1; {\it i.e.}, 
come in by way of (4.22).  
Writing $f=\varphi = \prod \varphi_{\nu}$,  
one would simply expand ${\mathcal D}^{[m]}\varphi$ via 
a rule of Leibnitz type. \,This is not unreasonable 
since $m \leqq k$. (We tacitly assume 
here$\MSA$\footnote{Compare [Fel2, p.528 (problem 15)].} 
that each $\varphi_{\nu}\MSZ$ corresponds to a $dF_\nu\MSZ$ 
having finite moments out to order at least $m+N+1$.) 
Proceeding in this way enables one to 
achieve fairly quickly at least \emph{some} orderly use of the 
multiplicativity ``without any division worries.''  Depending on 
the form of $\varphi_\nu$, keeping $\tau$ sufficiently small may 
also facilitate the calculation of any relevant derivatives 
of $\log \varphi_{\nu}$. 

In situations where $\varphi_{\nu}$ and $\varphi 
= \prod \varphi_{\nu}$ are nicely-behaved \emph{analytic}
functions of $\vec{v}$, a third line-of-attack would be to
express the linear combination    
$D_{\mathscr C} \varphi$  
as a Cauchy-type integral w.r.t. $\{ v_j \}^m_{j=1}$.  Doing this 
produces a factorization formula akin to (4.22).

In all three of these approaches, it pays to observe that 
those portions of $\prod^k_{j=1} [-\Omega_j, \Omega_j]$ having, 
e.g.,
\[
\{ |v_1| < \tau \MSB, \ldots, |v_\ell| < \tau 
\MSB;\MSB |v_{\ell+1}| \geqq 
\tau \MSB, \ldots, |v_m| \geqq \tau \}\qquad (1 \leqq \ell \leqq m)
\]
can be easily treated by first expressing $D_{\mathscr C}\varphi$ as
$D_1 \cdots D_\ell \MSH g$, wherein $g$ is the explicit linear combination
$D_{\ell+1} \cdots D_m \varphi$, and then simply repeating the earlier
ideas with ${\mathbb R}^\ell$ in place of ${\mathbb R}^m$.

What is $\MSAH$evident$\MSAH$ from all this is that, in seeking to gain 
some sensible control on $\MSA Q_{\mathscr C}$ in (4.12) by way 
of Taylor approximations when $\varphi = \prod\MSH\varphi_{\nu}\MSAH$, 
there is in-{\linebreak}variably an element of $\MSH$\emph{context}$\MSH$  
that enters the  
picture even if, $\MSA$say, $\MSA\alpha \geqq 1000 k$.

\subsection{} 
This fact, which is not entirely surprising, prompts one 
to ponder the possibility of making a 
fundamental \emph{change in mindset} 
that would enable one to simply $\MSH$sidestep$\MSH$ the 
bulk of any contextual issues with $Q_{\mathscr C}$.

To that end,  
the following corollary of (4.22) 
$\MSA$[$\MSH$taken now with any $m \geqq 1\MSH$]$\MSA$ 
is highly suggestive, particularly when 
specialized to $\varphi - \psi$. $\MSC$Namely: for 
any $f \in C^h$ and $\MSH\ell \in [1,m]\MSH$, one always has
\begin{equation}\label{4.24}
|D_1 \cdots D_m \MSA f| \leqq |v_1|^{h/ \ell} \cdots
|v_{\ell}|^{h / \ell}\MSA \Big ( \TS \prod^{\prime}_
{\sigma}\MSA {\mathscr M}_{\sigma}\MSA f
\Big )^{1/ \binom{\ell}{h}}\vspace*{.20ex}\hspace*{.48em}, 
\end{equation}
wherein $h \in [1,\ell]\MSA$,$\MSC$   
$\sigma$ ranges over all subsets of $[1,m]$ for which 
\[
\card(\sigma) = h \quad\MSE \mbox{and} \quad\MSE \sigma  
\subseteqq \{ 1,2, \ldots,\ell \} \MSE,
\]
and one defines, e.g.,
\[
{\mathscr M}_{\{1,2,...,h \}}\MSA f \equiv
\max_{\varepsilon^{\MSA 2}_{\beta} = 1}\MSA \sup_{|u_\gamma| \leqq 1} 
\Big | \frac{\partial^{h}f}{\partial\xi_1 \cdots \partial\xi_h} 
\Big (
 u_1 v_1,...,u_h v_h\MSA; \MSB\varepsilon_{h+1}v_{h+1},...,
\varepsilon_{m}v_m \Big ) \Big | \MSC .
\]
The index $\gamma$ ranges over $\sigma$, while the 
``\raisebox{.35ex}{\footnotesize {$\MSC\prime\MSZ$} \normalsize}'' 
serves to stress that only $\sigma \subseteqq [1,\ell]$ are 
included in the product. 

The case $\MSA h=0 \MSA$ is readily seen to hold 
trivially if 
the ${\mathscr M}_{\sigma}\MSC$-$\MSC$product is 
given an obvious interpretation. $\MSE$At the 
same time, it is also worth noting that, 
{\it as a technical artifice}, 
taking $\MSA h = 1\MSA$ will 
already be sufficient for most purposes. 

To prove (4.24) per se, we first put $g = D_{h+1} 
\cdots D_m \MSA f$ and then think of $g$ as a linear
combination.  The sum of the absolute
values of its coefficients is $2^{m-h} / 2^{m-h} = 1$.
With $\{v_{h+1},...,v_m\}$ frozen, we now apply (4.22)  
but with $m$ replaced by $h$.  \,This gives:
\begin{align*}
 | D_1 \cdots D_m \MSA f| = |(D_1 \cdots D_h)\MSA g| & \leqq |v_1|
 \cdots |v_h| \MSA \big ( {\mathscr M}_{\{1,...,h\}} \MSA g \big ) \\
 & \leqq |v_1| \cdots |v_h| \MSA \big ( {\mathscr M}_{\{1,...,h\}} 
\MSA f \big ) \MSC . 
\end{align*}
(The first ${\mathscr M}$ is taken on ${\mathbb R}^h$.)  Similarly for 
any other $\sigma \subseteqq [1,\ell]$ having 
cardi-{\linebreak}nality $h$.  
Multiplying over the collection of all 
relevant $\sigma$, one clearly gets
\[ 
 |D_1 \cdots D_m \MSA f|^{\binom{\ell}{h}} \leqq 
 \TS \prod^{\ell}_{j=1}\MSA |v_j|^{\binom{\ell - 1}{h-1}}\MSA\cdot 
\MSA \big ( \TS \prod^{\prime}_{\sigma}
\MSA {\mathscr M}_{\sigma}\MSA f \big )\MSC .
\]
Since
\small
\[
\binom{\ell - 1}{h-1} = \msfrac{h}{\ell} \binom{\ell}{h} 
\quad \mbox{\normalsize whenever} 
\quad  1 \leqq h \leqq \ell \MSE , 
\]
\normalsize 
inequality (4.24) follows  
immediately.\footnote{When $f$ is the restriction of a 
characteristic function $\varphi$ to ${\mathbb R}^m$, there are 
of course bounds similar to (4.24) having 
exponent $\lambda / \ell$ ($0 < \lambda \leqq \ell$) in 
place of $h/ \ell$. $\MSAH$These follow 
immediately from (4.4) and were already referenced, albeit 
implicitly$\MSH$, in the proof of Theorem 4.8. The 
${\mathscr M}_{\sigma}$-portion is best written 
as ${\mathbb E}(\|\vec{x}\|^{\MSH\lambda}_{\star} )$, 
where $\|\vec{x}\|_{\star} \equiv \max \{\MSH 
|x_j|\MSH: 1 \leqq j \leqq \ell \}$. }

There exists a \emph{slight extension} of (4.24) that 
will $\MSAH$also$\MSAH$ turn out to be useful 
in connection with $Q_{\mathscr C}$.    
$\MSE$To explain it, we'll continue to work 
in the setting of \S 4.4
with ${\mathbb R}^n$ replaced by ${\mathbb R}^m$. In addition 
to the operators $\{ D_1,\MSHZ...\MSA,D_m \}$, we propose to 
look at
\[
E_j = I- D_j \MSE,\MSE (P_j f)(\vec{v}) 
= f[\MSA((1-\delta_{hj})v_h)^{m}_{h=1}\MSA] \MSE,\MSE \Delta_j = I - P_j
\]
for $1 \leqq j \leqq m$.  
By straightforward symbol-pushing, each of
these operators on $C({\mathscr E})$ is readily seen to commute with
the other $4m-1\MSH,$ and to be an idempotent   
(i.e., satisfy $T^2 = T$). One also sees that\vspace*{-.62ex}
\begin{align*}
D_j E_j  &= 0 \MSB,\MSD\MSD D_j P_j = 0 \MSB,\MSD\MSD D_j \Delta_j = 
D_j\MSB,\\[.93ex]
 &E_j P_j = P_j \MSE,\quad P_j \Delta_j = 0 \MSE;
\end{align*}
\[
 \hspace{12.5em} I = P_j + D_j + E_j \Delta_j \MSE .\hspace{9.5em}(4.25)
\]
Besides being an idempotent, each summand in (4.25)  
annihilates the other two.  
The associated direct sum 
decomposition of $C({\mathscr E})$ can be viewed as providing a 
kind of ``Boolean'' Taylor   
development w.r.t. $v_j$. Corresponding to (4.22), one easily checks that
\setcounter{equation}{25}
\begin{align*}
 (\Delta_1 \cdots \Delta_m)f &= (v_1 \cdots v_m)\MSA\int_{[0,1]^m}\MSA
 ({\mathcal D}^{[m]}f)(t_1 v_1,...,t_m v_m)\MSA d\vec{t} \\[.31ex]
 ( {\TS \prod^{m}_{j=1}} E_j\Delta_j  ) f &= {\TS \prod^{m}_{j=1}}
(\vsfrac{1}{2} v^{\MSA 2}_{j}) \MSA \int_{ [0,1]^m \times [-1,1]^{m} } 
({\mathcal D}^{[m,m]}f)((t_j u_j v_j))\MSAH t_1 \cdots t_m\MSAH 
d\vec{u}\MSA d\vec{t} \\[.31ex]
({\TS \prod^{m}_{j=1}} E_j\Delta_j ) f &= {\TS \prod^{m}_{j=1}}
 (\vsfrac{1}{2} v^{\MSA 2}_{j} ) \MSA 
\int_{ [-1,1]^m }{\TS \prod^{m}_{j=1}} (1-|\xi_j|)\cdot
({\mathcal D}^{[m,m]} f) ((\xi_j v_j)) \MSA d\vec{\xi}
\end{align*}
for $f \in C^m$ and $C^{2m}$, respectively.  The notations 
${\mathcal D}^{[m,m]}$ and $((t_j u_j v_j))$ are obvious abbreviations.

Let $1 \leqq n \leqq m$ and $r = m - n$.  \,Writing $S_{j} = 
\partial /\partial \xi_j\MSB$, $\MSB$we now define
\small
\[
T_j = \left \{\begin{array}{ll}
S^{\MSA 2}_j, &\mbox{\,if\,\,\,} j \leqq n \\
\\
S_j, &\mbox{\,if\,\,\,} j > n 
\end{array} \right \}\MSC .
\]
\normalsize
Fix any $\ell \in [1,n]$ and $\delta \in [0,r]$. $\MSH$Put
$L = \ell + \delta$.  In 
line with \vspace*{.15ex} 
$\{ 1,\ldots,\ell \}\MSD\cup\MSA$ \linebreak 
$\{ n+1,\ldots,n+\delta \}$ and the actions
of both $\MSA I = P_j + \Delta_j\MSA$ and (4.25)$\MSA$, 
for any $h \in [1,L]$ $\MSA$and  
$f \in C^{h}\MSZ\MSZ,\MSC\MSH C^{2h}$ (resp.), $\MSB$one 
has 
\small
\begin{equation}\label{4.26}
 |\MSA\Delta_{1} \cdots \Delta_{n}\MSA D_{n+1} \cdots D_{m} f\MSA| 
\leqq  2^{m-h} \big (
 {\TS \prod^{\ell}_{j=1}} |v_j| \cdot {\TS \prod^{\delta}_{j=1}}
|v_{n+j}| \big )^{h/L} \MSA \big ( {\TS \prod^{\prime\prime}_{\sigma}}
\MSB {\mathscr M}_{\sigma} \MSA f \big )^{1/ \binom{L}{h}}\vspace*{-.40ex} 
\end{equation}
\begin{equation}\label{4.27}
 |\MSA{\TS \prod^{n}_{j=1}} E_j\Delta_j \MSB {\TS \prod^{m}_{j=n+1}} D_j 
f\MSA| \leqq  2^{m-h} \big ( {\TS \prod^{\ell}_{j=1}} |v_j| \cdot 
{\TS \prod^{\delta}_{j=1}} |v_{n+j}| \big )^{h/L}\MSA \big ( {\TS 
\prod^{\prime\prime}_{\sigma}}\MSB{\mathscr M}_{\sigma}\MSA f \big)^
{1/ \binom{L}{h}}\vspace*{.46ex}
\end{equation}
\begin{equation}\label{4.28}
 |\MSA{\TS \prod^{n}_{j=1}} E_j \Delta_j \MSB {\TS \prod^{m}_{j=n+1}} D_j
 f\MSA| \leqq  2^{m-h} \big ( {\TS \prod^{\ell}_{j=1}} |v_j|^2  \cdot
 {\TS \prod^{\delta}_{j=1}} |v_{n+j}| \big )^{h/L} \MSA \big ( {\TS
\prod^{\prime\prime}_{\sigma}} \MSB {\mathscr N}_{\sigma}\MSA f \big )^
{1 /  \binom{L}{h}}\vspace*{1.65ex}
\end{equation}
\normalsize
wherein $\sigma$ ranges over all subsets 
of $[1,n+r]\MSC\MSC(\MSA = [1,m]\MSH)$ $\MSA$for which
\[
\card(\sigma) = h \MSC,\MSD\MSD \sigma \subseteqq  
\{ 1,...,\ell \} \cup \{ n+1,...,n+\delta \}  
\]
and one defines, e.g.,
\begin{align*}
 {\mathscr M}_{ \{1,...,h \}} f &= \widehat{\max_{\varepsilon^
{\MSA 2}_{\beta}= 0,1}} 
 \MSD \sup_{|u_{\gamma}| \leqq 1} \MSA \big | (S_1 \cdots S_h\MSA f)
 (u_1 v_1,...,u_h v_h\MSA;\MSB \varepsilon_{h+1} v_{h+1},...,\varepsilon_m 
 v_m) \big | \\
 {\mathscr N}_{ \{1,...,h \}} f &= \widehat{\max_{\varepsilon^
{\MSA 2}_{\beta}= 0,1}}
 \MSD \sup_{|u_{\gamma}| \leqq 1} \MSA \big | (T_1 \cdots T_h\MSA f)
 (u_1 v_1,...,u_h v_h
 \MSA ;\MSB \varepsilon_{h+1} v_{h+1},...,\varepsilon_m 
 v_m ) \big |
\end{align*}
when $h \leqq \ell$.  Similarly for a split set   
$\{ 1,...,h_{1} \} \cup \{ n+1,...,n+h_{2} \} $.  \,In each 
instance, the index $\gamma$ ranges over $\sigma$ $\MSAH$and   
``\,\raisebox{-.56ex}{$\widehat{\phantom{n}}$}\,'' 
means$\MSH$ ``$\MSAH$omit $\MSH\varepsilon_\beta 
= 0$ $\MSAH$if$\MSAH$ $\beta > n\MSH . $''

Estimate (4.27) follows immediately from (4.26). $\MSH$The 
proofs of (4.26) and (4.28) are easy mimics of that of 
(4.24).  To keep things simple, one makes an intentional 
overshoot in the coefficient $2^{m-h}$.  

Just as before, the case ``$\MSH h = 0 \MSA$'' holds 
trivially when the product over $\sigma$ is given a 
natural interpretation. 

In regard to any applications with $f = \varphi-\psi$, 
the $\MSA$key$\MSH$ point in (4.24) and (4.26)--(4.28) is 
that $h$ can be taken equal to $1$. $\MSB$This imposes 
only the mild restriction 
that hypothesis $\MSA$(4.9)$\MSA$ be true 
with $\alpha = \MSH 1\MSA$ or $\MSH 2\MSH$ (as opposed to 
some-{\linebreak}thing more like $\alpha \geqq m+N+1$, in our 
earlier notation).

\subsection{}  
Before opting to place any further restrictions on $\alpha\MSH$,   
it seems only prudent to ask 
if there might not exist some  
easy variant of (4.12) [applicable 
for \linebreak \emph{any} $\MSH\alpha\MSH$] in which 
the structural hassles  
that we identified earlier (in \S 4.9) 
are \vspace*{-.05ex}simply absent.

$\MSD$The following result, based on truncation, 
achieves this goal in a \vspace*{-.20ex} relatively cheap   
way.

\phantom{mmm}
\pagebreak[4]

\phantom{mmm}\vspace*{-2.2ex}

\begin{thm48bisA}
$\MSA$Given the situation of Theorem 4.8 and any number 
$\Delta > 1$. Assume further that 
$\min (\Omega_1,\dots, \Omega_k)> 1$. $\MSH$Putting $v^{\bullet} 
\equiv v / \min \{1,\Delta |v| \}\MSB$ \emph{(and, say,}
$0^{\bullet} = 1/ \Delta\MSH)\MSH$, we then have 
\vspace*{-3ex}

\begin{align}\label{4.29}
 |F(t_1,&\dots, t_k) - G(t_1,\dots, t_k)| \\[.93ex]
 &\leqq  c_5 \int^{\Omega_1}_{-\Omega_1} \cdots 
\int^{\Omega_k}_{-\Omega_k} 
\frac{|\varphi (\vec{v})-\psi (\vec{v})|}
{|v^{\bullet}_1| \cdots |v^{\bullet}_k|} 
\MSB d\vec{v} \,+\, c_6 \sum^k_{\ell=1} \frac{m_\ell}{\Omega_\ell}
\nonumber\\
  &\hphantom{ni}+(k+1)\Delta^{-\alpha} \int_{{\mathbb R}^k}
(\max_{1\leqq j \leqq k} |x_j|)^\alpha\, 
(dF \,+\, |g| d\vec{x})\nonumber
\end{align}
for every $(t_1,\dots,t_k) \in {\mathbb R}^k$, wherein $c_5$ and $c_6$
are certain positive constants that depend solely on $k$. 
(Note the absence of $\vec{t}$ on the right hand side.)   
Each factor
$1/ |v^{\bullet}_j|$ in (4.29) can be replaced, 
whenever convenient, by  
the larger \vspace*{-.75ex}value
\[
\frac{\Delta}{|v^{\blacktriangle}_j|}\MSD, 
\quad\MSD\mbox{wherein}
\quad\MSE\MSE v^{\MSZ\blacktriangle} \equiv 
\frac{v}{ \min \{1, |v| \}} 
\MSD\MSD .
\vspace*{-.31ex}
\]
\end{thm48bisA}

\begin{proof}
It is useful to begin with an intermediate result.
Treating $\Delta$ as a natural cut-off, we first write
\[
t^\ast_j = \left \{ \begin{array}{rl}
t_j, & |t_j| \leqq \Delta\\[.62ex]
\Delta, & t_j > \Delta\\[.62ex]
-\Delta, & t_j < -\Delta
\end{array}\right \}
\]
for every index $j$. Since
\begin{align*}
&\hspace*{1.9em} |F(t_1, \eta_2,\ldots, \eta_k) - F(t^\ast_1, 
\eta_2,\ldots, \eta_k)| \,\leqq\, \Delta^{-\alpha}
\int_{{\mathbb R}^k} |x_1|^\alpha dF(\vec{x})\\
& \hspace*{19ex}\vdots\\
&|F(\eta_1,\ldots, \eta_{k-1}, t_k) - 
F(\eta_1,\ldots, \eta_{k-1}, t^\ast_k)| \,\leqq\,
\Delta^{-\alpha} \int_{{\mathbb R}^k} 
|x_k|^\alpha dF(\vec{x})\MSC,
\end{align*}
we immediately see that
\begin{equation}\label{4.30}
|F(t_1,\ldots, t_k) - F(t^\ast_1,\ldots,t^\ast_k)| \,\leqq\,
k\Delta^{-\alpha} \int_{{\mathbb R}^k}
\big ( \max_{1\leqq j\leqq k} |x_j|\big )^\alpha
\,dF(\vec{x})\MSC. 
\end{equation}
Similarly for $G$ and $|g|d\vec{x}\MSA$. \,Observe, however, that one
has
\[
|F(t_j)-G(t_j)| \,\leqq\, |F(t_j)-F(t^\ast_j)|
+ |F(t^\ast_j)-G(t^\ast_j)| + 
|G(t^\ast_j) - G(t_j)|
\]
in an obvious shorthand. Since 
$|\sin u| \leqq \min\{1, |u|\}$, by \vspace*{.05ex}
applying (4.12)
to $\MSA|F(t^\ast_j)-G(t^\ast_j)|\MSA$ and then noting  
that, on $\{|v_j| \leqq \Omega_j\}$,
\[
\msfrac{1}{\Omega_j} + \min \big (|t^\ast_j|, 
\msfrac{1}{|v_j|}\big )= \min \big ( \msfrac{1}{\Omega_j}+ |t^\ast_j|,
\MSB\msfrac{1}{\Omega_j} + \msfrac{1}{|v_j|}\big )\leqq
\min \big ( 2\Delta, \MSB\msfrac{2}{|v_j|}\big )
= \msfrac{2}{|v^{\bullet}_j|}\MSC\MSC , 
\]
one immediately finds \,[as a natural relative of 
(4.12)]\, that\vspace*{.62ex}
\begin{align}\label{4.31}
& |F(t_1,\ldots, t_k) - G(t_1,\ldots, t_k)|\\[1.24ex]
& \MSC\leqq 2^{k} c_1 \sum_{{\mathscr P}}
\int^{\Omega_1}_{-\Omega_1} \cdots \int^{\Omega_k}_{-\Omega_k}
|D_{\mathscr C}(\varphi-\psi)|\prod_{j\in {\mathscr C}}
\frac{1}{|v_j|} \prod_{j\in{\mathscr D}} 
\frac{1}{|v^{\bullet}_j|}
\prod_{j\in {\mathscr B}}\delta (v_j) 
d\vec{v}\nonumber\\[.50ex]
&\hspace*{1em} + c_2 \sum^k_{\ell=1} \frac{m_\ell}{\Omega_\ell}
+ k \Delta^{-\alpha}\int_{{\mathbb R}^k}
(\max_{1\leqq j \leqq k}|x_j|)^\alpha
\MSA(dF + |g| d\vec{x})\vspace*{.31ex}\nonumber
\end{align}
for every $(t_1,\ldots, t_k)\in {\mathbb R}^k$.  
Each ${\mathscr D}$$\MSA$-$\MSA$entry can be improved slightly since,
a couple lines earlier, we could have  
just as easily written 
\begin{align*}
\msfrac{1}{\Omega_j} + \min \big (|t^{\ast}_j|, 
\msfrac{1}{|v_j|} \big ) & =
\min \big ( \msfrac{1}{\Omega_j} + |t_j|, 
\MSC\msfrac{1}{\Omega_j}+ \Delta, 
\MSC\msfrac{1}{\Omega_j} + \msfrac{1}{|v_j|} \big )\\
& \leqq  \min \big ( \msfrac{2}{\Omega_j} + 2 |t_j|,
\MSC 2\Delta, 
\MSC \msfrac{2}{|v_j|} \big) \\
& =  2\min \big ( \msfrac{1}{\Omega_j}+|t_j|, 
\msfrac{1}{|v^{\bullet}_j|} \big )\MSD.
\end{align*}
Letting $\Delta \rightarrow \infty$ in this improved 
version of things  
basically reproduces (4.12).

Going back to (4.31), the issue is now one of 
trying to \emph{restructure} things so that all 
partitions but one (${\mathscr D} = \{1,\dots,k\}$) somehow 
disappear. \,For this purpose, we need to 
return to the proof of Theorem 4.8 and make several 
key modifications aimed at mollifying the 
$1/\pi iw$$\MSC$-$\MSC$singularity in $\widetilde{f}_\Omega$ 
and $\widetilde{g}_\Omega$ (cf. (4.14)).

To simplify matters, it is helpful to first treat the 
case in which $F$ has a continuous density 
function $f(\vec{x}$). We put
\[
L = 2\Delta
\]
and initially work with \emph{any} numbers 
$\{\Omega_j\}^k_{j=1}$ situated within $(1/2\pi, \infty)$. 
We also set\vspace*{.5ex}
\[
F_L (\vec{\xi}) = \int_{{\mathscr A}} f(\vec{x}) d\vec{x}
\quad \mbox{\,and\,}\quad G_L(\vec{\xi}) = \int_{{\mathscr A}} 
g(\vec{x}) d\vec{x}\,,\vspace*{.5ex}
\]
wherein\vspace*{.5ex}
\[
{\mathscr A} = (-L, \xi_1] \times \cdots \times 
(-L, \xi_k]\vspace*{.5ex}
\]
and each $\MSA\xi_j\MSA$ is understood to exceed $-L$. \,Since 
\[
{\TS \prod^k_{j=1}} (-\infty, \xi_j] - {\mathscr A} \MSD\subseteqq\MSD
\{\vec{x} \in {\mathbb R}^k: x_\ell \leqq -L
\mbox{ for some index $\ell$\}}\MSA,
\]
it is self-evident that
\begin{equation}\label{4.32}
\left \{\begin{array}{l}
0\leqq F(\vec{\xi}) - F_L(\vec{\xi}) \leqq
L^{-\alpha} \int_{{\mathbb R}^k }
\big (\max_{1 \leqq j \leqq k} |x_j|\big )^\alpha
f \MSA d\vec{x}\\[1.0ex]
\hphantom{xx}|G(\vec{\xi})-G_{L}(\vec{\xi})| \leqq
L^{-\alpha} \int_{{\mathbb R}^k}
\big (\max_{1 \leqq j \leqq \ell} |x_j|\big )^\alpha
|g| \MSA d\vec{x}
\end{array}\right \}\,.
\end{equation}
\pagebreak

The plan is to 
now bound $\MSA |F_L(\vec{\xi})-G_L(\vec{\xi})| \MSA$ 
by mimicking the  
considerations utilized near equation (4.16) in the  
proof of Theorem 4.8. 
With $\vec{\xi}$ temporarily frozen, we define
\[
\chi_j(x) = \chi_{(-L, \xi_j]}(x) \qquad (1\leqq j \leqq k)
\]
and observe that
\begin{align*}
\chi_{{\mathscr A}} (\vec{x}) & = 
\chi_1(x_1) \cdots \chi_k (x_k)\\[.62ex]
F_L (\vec{\xi})-G_L(\vec{\xi}) & =
\int_{{\mathbb R}^k} (f-g)
\MSC\chi_1(x_1)\, \cdots \,\chi_k(x_k)\,d\vec{x}\,.
\end{align*}
Dropping the subscript $j$, we also note (cf. (2.10)) 
that\vspace*{.5ex}
\begin{equation}\label{4.33}
\frac{b[\Omega(x+L)]-B[\Omega (x-\xi)]}{2}
\leqq \chi_{(-L, \xi]}(x) \leqq
\frac{B[\Omega (x+L)]-b[\Omega(x-\xi)]}{2}\vspace*{.5ex}
\end{equation}
for all $x \in {\mathbb R}$ (the points $-L$ and $\xi$ 
being handled via trivial limits). By \linebreak Theorem 2.2(a), 
the extreme members of (4.33) have a difference equal 
to $K[\Omega (x+L)] + K[\Omega (x-\xi)]$.

Following (4.14), we express the minorant portion of 
(4.33) as a Fourier transform
\begin{align*}
\widetilde{f}_\Omega (x) & = \frac{1}{2} 
\oint^\Omega_{-\Omega} \Big \{
\frac{e^{2\pi i Lw}}{\pi iw}-
\frac{e^{-2\pi i\xi w}}{\pi iw}\Big \}
e^{2\pi ixw}dw\\[2mm]
& + \frac{1}{2} \int^\Omega_{-\Omega}
\Big \{ \frac{e^{2\pi iLw}}{\Omega}\, R_b
\big ( \frac{w}{\Omega}\big )-
\frac{e^{-2\pi i\xi w}}{\Omega} \, R_B
\big ( \frac{w}{\Omega}\big )\Big \}
e^{2\pi ixw} dw
\end{align*}
and proceed similarly with the majorant (calling it 
$\widetilde{g}_\Omega$). Each ``brace'' is clearly Hermitian 
w.r.t. $w$. In marked contrast, however, to what happened 
previously, the \emph{first brace} is now a \emph{continuous} 
function on $[-\Omega, \Omega]\MSA$; $\MSE$its value is simply
\[
2\MSB\frac{\sin [\pi (L+\xi)w]}{\pi w}\MSB e^{\pi i(L-\xi)w}\MSC.
\]
At the same time, it is also natural to write
\begin{align*}
{\mathscr R}^- (w) & = \frac{e^{2\pi iLw}}{\Omega}\,
 R_b \big (\frac{w}{\Omega} \big )-
\frac{e^{-2\pi i \xi w}}{\Omega} \, R_B
\big ( \frac{w}{\Omega} \big );\\[.62ex]
{\mathscr R}^+ (w) & = \frac{e^{2\pi iLw}}{\Omega}\,
R_B \big (\frac{w}{\Omega} \big )-
\frac{e^{-2\pi i \xi w}}{\Omega}\, R_b 
\big (\frac{w}{\Omega}\big )\,.
\end{align*}
One clearly has
\[
|{\mathscr R}^\pm (w)| = \frac{O(1)}{\Omega}
\]
with an implied constant that is absolute (the 
value \small {$ 15 / 4 $} \normalsize 
certainly works here as a bit of 
elementary calculus shows).\vspace*{.31ex}

With the essential components all in place, the earlier 
\vspace*{.19ex} reasoning [beginning at (4.16)] is 
readily repeated with 
${\mathscr P} = \langle \phi, {\mathscr C}, {\mathscr D}\rangle$ 
and leads to
\small
\begin{align*}
&|F_L(\vec{\xi})-G_L(\vec{\xi})|\\
&\hspace{2em} \leqq
c_7 \int^{\Omega_1}_{-\Omega_1} \cdots \int^{\Omega_k}_{-\Omega_k}
|\varphi (2\pi \vec{w})-\psi (2\pi \vec{w})|
\prod^k_{j=1} \Big (\frac{|\sin \pi w_j(L+ \xi_j)|}{\pi|w_j|}
+ \frac{1}{\Omega_j}\Big ) d\vec{w}\\
&\hspace{2.6em} + 2^k c_3 \MSD\sum^k_{\ell=1} 
\frac{m_\ell}{\Omega_\ell}\MSB.
\end{align*}
\normalsize
The coefficient $\MSAH2^{k}\MSA c_3\MSAH$ arises 
from (4.18), \,the fact that
\[
\int^\infty_{-\infty} \{K[\Omega_\ell (x+L)] + 
K[\Omega_\ell (x-\xi)]\} dx = \frac{2}{\Omega_\ell}\MSC ,
\]
and the possibility that each $\delta_j, \varepsilon_j$ [in Lemma 4.3] 
can theoretically become as large as $2$ (cf. just after (4.33)).

In situations where $|\xi_j| \leqq \Delta$, we clearly 
have \vspace*{.46ex}
\small
\begin{align*}
\frac{|\sin \pi w_j(L+\xi_j)\MSA|}{\pi |w_j|}
 \leqq &\min \Big ( \frac{1}{\pi |w_j|}, |L+ \xi_j|\Big )
\leqq \min \Big ( \frac{1}{|w_j|}, \frac{3}{2} L\Big )\\[.93ex]
\frac{1}{\Omega_j} + \frac{|\sin (\ast\ast\ast)|}{\pi|w_j|}
 \leqq \min &\Big ( \frac{1}{\Omega_j}+ \frac{1}{|w_j|},
\frac{1}{\Omega_j} + \frac{3}{2} L\Big ) \leqq
\min \Big ( \frac{2}{|w_j|}, 5L \Big )\vspace*{.46ex}
\end{align*}
\normalsize
$\vphantom{ 2^{X^{X^Q}} }$on $\{|w_j| \leqq \Omega_j \}$, 
\,since $1/\Omega_j < 2\pi < \pi L$. 
\,Here, \vspace*{-.31ex}then,
\begin{align*}
 |F_L(\vec{\xi})-G_L(\vec{\xi})| & \leqq
5^k c_7\MSA\int^{\Omega_1}_{-\Omega_1} \cdots
\int^{\Omega_k}_{-\Omega_k} |\varphi (2\pi\vec{w})-
\psi (2\pi \vec{w})| \prod^k_{j=1}
\min \Big (\frac{1}{|w_j|}, L\Big ) d\vec{w}\\[-.31ex]
& + 2^k c_3\MSC \sum^k_{\ell=1} \frac{m_\ell}{\Omega_\ell}
\,.\vspace*{-.31ex}
\end{align*}

Putting $v_j=2\pi w_j$ and temporarily writing $\MSA \Omega_\ell =  
E_\ell/2\pi \MSA$ now gives\vspace*{-.18ex}
\begin{align*}
 |F_L(\vec{\xi})-G_L(\vec{\xi})| 
& \leqq 5^k c_7\MSA \int^{\Omega_1}_{-\Omega_1} \cdots
\int^{\Omega_k}_{-\Omega_k} |\varphi (\vec{v})-\psi (\vec{v})|
\prod^k_{j=1} \min \big (\frac{1}{|v_j|},  \Delta) d\vec{v}\\[-.31ex]
& + 2^k (2\pi c_3) \sum^k_{\ell=1} \frac{m_\ell}
{\Omega_\ell}\vspace*{-.18ex}
\end{align*}
for the \emph{original} $\MSC\Omega_j$ 
satisfying $\min (\Omega_1,\ldots, \Omega_k)> 1$. 
$\MSZ\MSHZ$(Observe 
that $2\pi c_3$ is simply\linebreak the coefficient $c_2$ that 
occurred earlier in (4.12); $\MSE$cf. (4.18).)
Relation (4.29) follows upon taking $\xi_j = t^\ast_j$, \,noticing 
that
\begin{align}\label{4.34}
|F(t_j)-G&(t_j)| \leqq
|F(t_j)-F(t^\ast_j)| + |F(t^\ast_j)-F_L(t^\ast_j)|\hspace*{6.3em}\\[.24ex]  
&\phantom{m} + |F_L(t^\ast_j)-G_L(t^\ast_j)|+ |G_L(t^\ast_j)-G(t^\ast_j)| +
|G(t^\ast_j)-G(t_j)|\MSB,\nonumber
\end{align}
and finally substituting (4.30) + (4.32).

The case of an arbitrary $F$ is handled by the usual convolution argument. 
In this connection, it may be worthwhile to point out that, while
\begin{align*}
{\mathbb E} [\MSC(\max_j |U_j \pm Y_{j}|)^\alpha ]^{1/\alpha}
& \leqq {\mathbb E}[\MSC(\max_j |U_j|+\max_j |
Y_{j}|)^\alpha]^{1/\alpha}\\
& \leqq {\mathbb E}[\MSC(\max_j |U_j|)^\alpha ]^{1/\alpha}
+ {\mathbb E} [\MSC(\max_j |Y_j|)^\alpha]^{1/\alpha}
\end{align*}
is completely standard for $\alpha \geqq 1$, the elementary inequality
\[
(|x_1| + |x_2|)^\omega \leqq |x_1|^\omega + |x_2|^\omega 
\quad \mbox{for \quad $0<\omega < 1$}
\]
leads to\vspace*{-.15ex}
\begin{align*}
{\mathbb E}[\MSC(\max_j |U_j \pm Y_{j}|)^\alpha]
& \leqq {\mathbb E} [\MSC(\max_j |U_j|
+\max_j |Y_{j}|)^\alpha]\\
& \leqq {\mathbb E}[\MSC(\max_j |U_j|)^\alpha +
(\max_j |Y_{j}|)^\alpha ]\\
& = {\mathbb E}[\MSC(\max_j |U_j|)^\alpha ] +
{\mathbb E}[\MSC(\max_j|Y_{j}|)^\alpha ]\\[-5ex]
\end{align*}
when $0<\alpha <1$. In either situation, taking a 
multivariate 
Gaussian ${\mathscr N}^{\langle k \rangle}_\varepsilon$ 
for \linebreak $Y$ and letting $\varepsilon \rightarrow 0$ 
clearly leads to no change 
in constants in \vspace*{.15ex}(4.29). \bx 
\end{proof}

By making a slight adaptation in the foregoing argument, the following 
result parallel to Theorem 4.8$\MSH$A is obtained virtually 
\vspace*{.10ex}immediately.

\begin{thm48bisB}
Given the situation of Theorem 4.8 and any number $D > 0$.  Assume 
further that $\min (\Omega_1,\ldots,\Omega_k ) > 1$.  Writing
$v^{\bullet} = v/ \min\{1,\Delta |v| \}$ with 
$\Delta = 1+D$,  
and taking \\[-.75ex]
\[
{\mathscr A}_{ab} \MSD=\MSD {\TS \prod^{k}_{j=1}}\MSH (a_j,b_j] \MSD
\subseteqq\MSD {\mathbb R}^{k}\MSC ,\\[.31ex]
\]
we then have, in the sense of measures,\\[-.62ex]
\begin{equation}\label{4.35}
|F\{{\mathscr A}_{ab}\} - G\{{\mathscr A}_{ab}\}|
  \leqq c_8 \int^{\Omega_1}_{-\Omega_1}\cdots 
 \int^{\Omega_k}_{-\Omega_k} \frac{|\varphi(\vec{v}) -
 \psi(\vec{v})|}{|v^{\bullet}_1|\cdots|v^{\bullet}_k|}
 \MSB d\vec{v} \MSB+\MSB c_9 \sum^{k}_{\ell = 1} \frac{m_{\ell}}
 {\Omega_{\ell}}\\[.62ex]
\end{equation}
whenever $\MSB 0 \leqq b_j - a_j \leqq D\MSB$.  $\MSB$The 
coefficients $c_8$ and $c_9$ will depend solely on $k$.\\[-.62ex]
\end{thm48bisB}

\begin{proof} 
The crux of the matter\vspace*{.09ex} begins just prior to 
(4.33); \,one simply replaces $-L$ by $a_j$ and $\xi_j$ by $b_j\MSA$,
and proceeds from there.  Taking $\Omega_j > 1/2 \pi$ 
as \vspace*{.09ex}before, 
the key observation is that, on $\{ |w_j| \leqq \Omega_j \}$,
\begin{align*}
\msfrac{1}{\Omega_j} \MSZ+\MSZ \msfrac{|\sin \pi w_j 
(b_j - a_j)|}{\pi |w_j|}
 &\MSB\leqq\MSB \msfrac{1}{\Omega_j} \MSZ+\MSZ \min 
\Big ( \msfrac{1}{\pi |w_j|}\MSA,
 \MSB |b_j - a_j| \Big ) \\
 &\MSB\leqq\MSB  2 \min \Big ( \msfrac{1}{|w_j|}\MSA,\MSB \pi + D \Big )\MSC .
\end{align*}
This quickly produces 
(4.35) with $c_8 = 2^k c_7$ and $c_9 = 2^k(2\pi c_3) = 2^{k}c_2$. 
$\MSC$\bx
\end{proof} 
\pagebreak

As a $k$$\MSA$-$\MSA$dimensional counterpart to Corollary 3.4, one 
now has

\begin{corollary}
$\MSE\MSC$Given probability distributions $\MSB F_n\MSB$ 
and $\MSB G\MSB$ on $\MSB {\mathbb R}^k\MSB$ having characteristic 
functions $\MSA \varphi_n \MSA$ and $\MSA \psi \MSA$. Assume that $G$
has a continuous density function $\MSA g(\vec{x}) \MSA$ 
which satisfies 
(4.9) and (4.10) of Theorem 4.8.  Assume further that \\[.78ex]
(i)$\phantom{i}$ \quad 
$\int_{{\mathbb R}^k} \MSC (\|\vec{x}\|_\infty)^\alpha \MSB dF_n(\vec{x}) 
= O(1) \MSB$; \\[.78ex]
(ii) \quad $\varphi_n (\vec{v}) \rightarrow \psi(\vec{v})$ pointwise
on ${\mathbb R}^k$ as $n \rightarrow \infty \MSB$. \\[.78ex]
We then have
\[
F_n(\vec{x}) \rightarrow  G(\vec{x})
\]
uniformly on ${\mathbb R}^k$.  An analogous result holds when $n$ is
replaced by a continuous variable $\xi$.
\end{corollary}

\begin{proof}
An easy consequence of either Theorem 4.8$\MSH$A or 4.8$\MSH$B as the
the cut-off parameter $\Delta$ is moved upward incrementally. $\MSB$In 
the case of 4.8$\MSH$B, one exploits 
the $\MSZ$\emph{double}$\MSZ$ cut-off idea implicit in (4.30)+(4.32) 
with, say, $\vsfrac{1}{4} \Delta$ in place of $\Delta$ \linebreak 
and keeps   
$\MSA\xi_j \in [-\vsfrac{L}{2},\MSH\vsfrac{L}{2}]\MSA$ in 
relation (4.32). $\MSB$The 
essential observation is then (4.34). $\MSB$See (4.6)+(4.4) 
concerning the necessary H\"{o}lder estimates. $\MSB$(Recall 
$\MSA$\emph{too} $\MSB$\S3.6\, and our earlier comment about 
continuous 
distributions ${\mathscr P}$.)$\MSE\MSE\MSE$\bx  
\vspace*{.15ex}
\end{proof}

If (i) holds with 
some $\alpha > \ell$ ($\ell \in {\mathbb Z}^+$), and 
${\mathcal D}$ is any partial 
derivative w.r.t.$\MSZ$ $\vec{v}\MSA$ of 
order $\leqq \ell$, the assertion 
$F_n \rightarrow G$ readily implies 
${\mathcal D}\varphi_n \rightrightarrows 
{\mathcal D}\psi$ on compacta.  
Similarly for the associated moments 
and $\MSA$absolute moments$\MSA$ of $dF_n$.

\subsection{}
$\MSZ\MSZ$Our third, and final, variant of Theorem 4.8 will 
arise from taking $\alpha = 1$ in $\MSH$(4.9)$\MSH$, $\MSA$and 
then revisiting the \emph{proof}$\MSZ$ of Theorem 4.8 aided  
by the circle$\MSA$-$\MSA$of$\MSA$-$\MSA$ideas 
looked at in \S4.10. $\MSB$Though its   
statement over $\MSH {\mathscr P} \MSH$ will be 
simple enough \linebreak (see (4.40)$\MSZ$ {\it infra}), the 
argument used in proving it $\MSHZ$will$\MSHZ$ entail a bit of 
bookkeeping. $\MSC\MSH$It is helpful to lay the 
groundwork for $\MSA$this$\MSA$ side of things first.

To that end, we begin by noticing, in connection with (4.25), 
that the functions $D_j f, \MSA E_j f, \MSA \Delta_j f, 
\MSA P_j f$, and $E_j \Delta_j f$ are all Hermitian 
anytime $\MSH f \MSH$ is; cf. \S4.4.  At the same time, 
we recall \vspace*{-.62ex}that
\begin{equation}\label{4.36}
\int^{\gamma_2}_{\gamma_1}\MSA\msfrac{\sin(tv)}{v}\MSA dv = O(1)
\MSE\MSE \mbox{for any}\MSE\MSE \gamma_2 > 
\gamma_1 \MSE,\MSE  t \in {\mathbb R}\vspace*{-.93ex}
\end{equation}
and\vspace*{-.15ex}
\begin{equation}\label{4.37}
\delta(v) = a\delta(av) \quad \mbox{whenever}\quad a > 0.\vspace*{.20ex}
\end{equation}
For ${\mathscr C} \subseteqq \{ 1,2,\ldots,k \}$ 
and $\vec{v} \in {\mathbb R}^k$, 
we now write
\begin{align*}
&\hspace{0em}{\mathscr C}_b(\vec{v}) = \{j \in {\mathscr C}: 
|v_j| \geqq 1 \}\MSA ,\quad  
 {\mathscr C}_s(\vec{v}) = \{j \in {\mathscr C}: |v_j| < 1\}\MSA,\quad
(\,\sim \mbox{big/small}\,)\\[.93ex]
&E_{\mathscr C}(\vec{v}) = \{\vec{\xi}  
: \xi_j = v_j 
\mbox{\,\,if\,\,} j \notin {\mathscr C}, \MSC \xi_j = \pm\MSH v_j \mbox{\,\,if\,\,} 
j \in {\mathscr C}_b(\vec{v}), \MSC |\xi_j| \leqq |v_j| \mbox{\,\,if\,\,} 
j \in {\mathscr C}_s(\vec{v}) \},\\[.93ex]
&{\mathscr E}_{\mathscr C}(\vec{v}) = \{\vec{\xi} 
: \xi_j = v_j
\mbox{\,\,if\,\,} j \notin {\mathscr C}, \MSC \xi_j = \pm\MSH v_j \mbox{\,\,if\,\,} 
j \in {\mathscr C}_b(\vec{v}), \MSC |\xi_j| \leqq 1 \mbox{\,\,if\,\,}  
j \in {\mathscr C}_s(\vec{v}) \}. 
\end{align*}
We also select any numbers $\MSH F_1,\ldots,F_k\MSH$ in $(1,\infty)$ and 
then restrict
$\vec{v}$ to lie in $\prod^{k}_{j=1}[-F_j,F_j]$. By applying the 
partition $[-F_j,F_j] = (-1,1) \cup \{ 1 \leqq |v_j| \leqq F_j \}$,
the box $\prod^{k}_{j=1}[-F_j,F_j]$ clearly splits 
into $2^k$ symmetrically-shaped 
(generally disconnected) subregions.  The set of \emph{interior} points of 
each such subregion will be succinctly referred 
to as a ``slab.''  Taking the union of 
all $2^k$ of these slabs produces $\prod^{k}_{j=1}[-F_j,F_j]$ 
apart from\vspace*{.06ex} 
a set of ($k$-dimensional) measure $0$.  The resulting ``grid''  
furnishes a natural underpinning for the obvious \emph{e}lectrical 
short-circuit interpretation of 
the sets $E_{{\mathscr C}}(\vec{v})$ and
${\mathscr E}_{{\mathscr C}}(\vec{v})$. $\MSAH$Finally,  
let $S_j = \partial / \partial \xi_j$ as in $\MSAH$\S4.10$\MSAH$ and 
$\MSA\vec{v} = 2\pi\vec{w}\MSAH$. $\MSB$We thus have\vspace*{.558ex}
\[ \vec{v} \in {\TS\prod^{k}_{j=1}}\MSH [-F_j,F_j]\hspace{1.5em}\mbox{and}
\hspace{1.52em}\vec{w} \in {\TS\prod^{k}_{j=1}}
\MSH [-\omega_j,\omega_j]\hspace{.20em},  \vspace*{-.176ex}
\]
wherein $\omega_j =  \vsfrac{1}{2\pi}F_j\MSA$.

Associated with each $f \in C^{1}(\prod^{k}_{j=1}[-F_j,F_j])$ 
and ${\mathscr C} \subseteqq \{1,2,\ldots,k\}$ are two 
\emph{piecewise continuous} functions, 
namely, $|f|_{\mathscr C}$ and $\|f\|_{\mathscr C}\MSA$, that 
one introduces by writing
\begin{equation}\label{4.38}
|f|_{\mathscr C}(\vec{v}) = \left \{\begin{array}{ll} 
\max \big[\MSA |f(\vec{\xi})|: \vec{\xi} \in E_{\mathscr C}(\vec{v})
\MSA\big ],&\mbox{\,\,if\,\,} {\mathscr C}_s(\vec{v}) = \phi \\
\\
\sup \big [\MSA|(S_j f)(\vec{\xi})|: \vec{\xi} \in E_{\mathscr C}(\vec{v}), 
j \in {\mathscr C}_s(\vec{v})\MSA\big ], &\mbox{\,\,if\,\,} 
{\mathscr C}_s(\vec{v}) > \phi
\end{array} \right \}
\end{equation}
and
\begin{equation}\label{4.39}
\|f\|_{\mathscr C}(\vec{v}) = \left \{\begin{array}{ll}
\max \big[ \MSA|f(\vec{\xi})|: \vec{\xi} \in {\mathscr E}_{\mathscr C}(\vec{v})
\MSA\big ],&\mbox{\,\,if\,\,} {\mathscr C}_s(\vec{v}) = \phi \\
\\
\sup \big[\MSA|(S_j f)(\vec{\xi})|: \vec{\xi} \in {\mathscr E}_{\mathscr C}(\vec{v}), 
j \in {\mathscr C}_s(\vec{v})\MSA\big], &\mbox{\,\,if\,\,} 
{\mathscr C}_s(\vec{v}) > \phi
\end{array} \right \}
\end{equation}
The functions $|f|_{\mathscr C}$ and $\|f\|_{\mathscr C}$ are 
clearly continuous on each slab. 

Since $\|f\|_{\mathscr C}$ typically manifests at least a 
partial dependence on $\vec{v}$, there is a slight abuse of 
notation when we write (4.39). Needless to \vspace*{-.31ex}say,
\[
|f|_{\mathscr C} \leqq \|f\|_{\mathscr C} \quad\mbox{and}
\quad |f|_\phi = \|f\|_\phi = |f(\vec{v})|\MSC.\vspace*{-.31ex}
\]
In working with $|f|_{\mathscr C}$ and $\|f\|_{\mathscr C}$ 
slab-by-slab, those variables $v_j$  
for which $j \in {\mathscr C}$ should naturally 
be viewed as primary (or$\MSA$ ``active'')$\MSH$. 
$\MSB\MSH$Those having $j \notin {\mathscr C}$ are 
best \vspace*{.46ex} treated    
as auxiliaries$\MSA$.

\begin{thm48bisC}
Consider the situation of Theorem 4.8. $\MSB$Assume that (4.9) holds 
with $\alpha = 1$ and that the quantities $\MSA\Omega_j\MSA$ 
satisfy $\MSB\min(\Omega_1,\ldots,\Omega_k) > 1$.  Let $v^{\blacktriangle}
= v/ \min\{1,|v|\}$ \,\emph{(and} $0^\blacktriangle = 1$\emph{)} as 
in Theorem 4.8$\MSH$A. $\MSE$We then have
\begin{align}\label{4.40}
&\hspace*{-1.3em} |F(t_1,\ldots,t_k) - G(t_1,\ldots,t_k)| \\[2ex]
& \phantom{mm} \leqq \widehat{c}_1\MSB 
\sum_{\mathscr P} \int^{\Omega_1}_{-\Omega_1}
\cdots \int^{\Omega_k}_{-\Omega_k}\MSA\frac{\|\MSB\varphi-\psi\MSB 
\|_{\mathscr C}}{ \prod_{j \in {\mathscr C}} \MSH |v^{\blacktriangle}_j| }\MSA
 \prod_{j \in {\mathscr D}} \msfrac{1}{\Omega_j}\MSA\prod_{j \in {\mathscr B}}
\MSA\delta (v_j)\MSA d\vec{v}\nonumber\\
& \phantom{mniii} + c_2 \sum^{k}_{\ell=1}\MSA\msfrac{m_\ell}
{\Omega_\ell}\vspace*{-.31ex}\nonumber
\end{align}
at every point $(t_1,\ldots,t_k) \in {\mathbb R}^k$, 
wherein $\widehat{c}_1$ and
$c_2$ are certain \emph{[}explicitly computable\emph{]} $\MSA$positive constants 
that depend solely on $k$. \,The constant $c_2$ is identical 
to the one that occurs in (4.12).  
\end{thm48bisC}
\pagebreak

\begin{proof} The basic idea is very simple: $\MSC\MSH$one 
merely repeats the proof of Theorem 4.8 
up to (4.20) with $\omega_j$ in$\MSZ$ place of $\Omega_j$,  
and then$\MSA$, after making a trivial ``switch over'' 
to $\vec{v} = 2\pi\vec{w}$ 
and $\MSA (\omega_j,F_j)\MSH$, applies$\MSH$ (4.25), (4.27), 
and (4.36) in a natural way to bound the value of the 
associated $d\vec{v}$\emph{$\MSA$--$\MSB$integral$\MSA$} 
on a slab-by-slab basis.  

The essential steps in the argument go as follows. $\MSC\MSA$Let 
$T_{{\mathscr B}{\mathscr D}{\mathscr C}_1{\mathscr C}_2}d\vec{v}$
signify the updated integrand in (4.20), including \vspace*{.93ex}
any trivial ${\mathbb Q}$-coefficients.  Clearly
\begin{align}\label{4.41}
 T_{ {\mathscr B} {\mathscr D} {\mathscr C}_1 {\mathscr C}_2 } 
 = \widehat{A} \MSB\prod_{j \in {\mathscr B}} &\MSA\delta(v_j)
 \MSC\prod_{j \in {\mathscr D}} \MSA \msfrac{1}{F_j} 
 R_{\ast}\Big ( \msfrac{v_j}{F_j} \Big ) 
 \prod_{j \in {\mathscr B} \cup {\mathscr D}} e^{-i v_j t_j}
 \prod_{j \in {\mathscr C}_1}\MSH e^{-i v_j t_j}\\
 & \cdot \MSA \msfrac{ D_{{\mathscr C}_1}(\varphi-\psi)(\vec{v}) }
 { { \TS \prod_{j \in {\mathscr C}_1} } (i v_j) }\MSA
 \prod _{j \in {\mathscr C}_2} 
 \Big (\msfrac{\sin(v_j t_j)}{v_j} \Big )\MSB,\nonumber 
\end{align}
where $\MSH\widehat{A}\MSH$ is some real constant depending 
solely on 
$k,{\mathscr B}, {\mathscr D}, 
{\mathscr C}_1 , {\mathscr C}_2$  
and whether\linebreak the 
RHS of (4.41) is attached 
to either the majorant or minorant track.  
Note that $T_{\ast}$ is Hermitian 
w.r.t. $\MSZ\MSZ\vec{v}\MSB$; cf. \S4.4.  After 
fixing any number $\eta$ in $(0,1/k]$, say, 
$1/k$, $\MSB$we set $|v|_{\star} = 
\max\{|v|, |v|^{1-\eta}\}$ and form 
the quantity
\begin{equation}\label{4.42}
Q \MSA\equiv\MSA \widehat{C}\MSC
\sum_{{\mathscr P}'}\MSH \int^{F_1}_{-F_1}
\cdots \int^{F_k}_{-F_k}\MSA
\frac{ |\varphi - \psi|_{{\mathscr C}'} }
{\TS \prod_{j \in {\mathscr C}'}\MSA
|v_j|_{\star}}\MSA\prod_{j \in {\mathscr D}'}\msfrac{1}{F_j}
\MSA\prod_{j \in {\mathscr B}'}\MSA\delta(v_j)\MSA d\vec{v}\MSA,
\end{equation}
wherein ${\mathscr P}' = \langle {\mathscr B}', {\mathscr C}',
{\mathscr D}' \rangle\MSAH$ \`{a}$\MSC\MSC$la \S4.7, 
and $\widehat{C}$ is some sufficiently large 
constant depending solely on $k$.

Let ${\mathscr S}$ be any slab.  The plan is to 
show that the real number $\int_{{\mathscr S}}\MSH 
T_{{\mathscr B}{\mathscr D}{\mathscr C}_1
{\mathscr C}_2}\MSA d\vec{v}$ 
splits into at most $2^k$ summands, each of which has absolute value 
majorized by \emph{some} summand of $\MSA Q\MSA$ (restricted, 
of course, to ${\mathscr S}$).  
No specific relation is implied here between
$\{{\mathscr B}, {\mathscr D}, {\mathscr C}_1, 
{\mathscr C}_2 \}$ and $\{ {\mathscr B}', {\mathscr D}', 
{\mathscr C}' \}$.  \,Since
\[
|\varphi - \psi|_{{\mathscr C}'} \leqq \|\varphi-\psi\|
_{{\mathscr C}'}\quad\MSE\mbox{and}\MSE\quad \int^1_{-1}\MSA
\msfrac{1}{|v|_{\star}}\MSA dv < +\infty\MSB,
\]
once this inclusion property is shown, estimate (4.40) 
follows immediately.

With ${\mathscr S}$ and $\{ {\mathscr B}, {\mathscr D}, 
{\mathscr C}_1, {\mathscr C}_2 \}$ now viewed as fixed, there
is obviously no loss of generality if we assume that things have
been relabelled so that:\\[.31ex]
\hspace{1em}(i) the portion of ${\mathscr C}_2$ ``having'' $|v_j|
< 1$ $\MSA$w.r.t.$\MSA$ ${\mathscr S}\MSA$ 
is $\{ 1,2,\ldots,N \}\MSA$;\\[.10ex]
\hspace{1em}(ii) ${\mathscr C}_1 = \{ N+1,\ldots,m \}\MSA$;\\[.10ex]
\hspace{1em}(iii) ${\mathscr C}_2 - \{1,\ldots,N\} = 
\{m+1,\ldots,M \}\MSAH$.\\[.31ex]
The integers $N, m, M$ are adapted to any empty sets in the
obvious way.  

Letting $\Fh = \varphi - \psi$, we next write $\Fh = I\MSH\Fh$, 
where\vspace*{-.46ex}
\[
I = \prod^N_{j=1}\MSH (P_j + D_j + E_j\Delta_j)\vspace*{-.46ex}
\]
\`{a} la (4.25).  This decomposes the function $\Fh$ into 
$3^N$ summands of the form $W\Fh$, wherein each $W$ is a 
``word'' in (i.e., element of) $\MSB\prod^N_{j=1}
\{ P_j, D_j, E_j\Delta_j \}$. The number $\int_{\mathscr S}
\MSH T_{ {\mathscr B}{\mathscr D}{\mathscr C}_1 
{\mathscr C}_2 }\MSA d\vec{v}$  
$\MSAH$($= T \MSH {\mathscr S}$, for short)$\MSAH$ decomposes similarly.  
We stress here that$\MSA$  
${\mathscr C}_1 \cap  {\mathscr C}_2 = \phi$. 

Suppose for a moment that $W$ contains some $D_\iota$.  The
function $D_{{\mathscr C}_1}(W\Fh)$ is \emph{odd} w.r.t. $v_\iota$. 
\,On the other hand, $\MSB\sin(v_\iota t_\iota) / v_\iota \MSB$ 
is even.  Since ${\mathscr S}$ is symmetric, the 
corresponding ``chunk'' of $T\MSH{\mathscr S}$, viewed as
an iterated integral (cf. (4.41)), is therefore $0$.  \,In other
words: we need only look at $W \in \prod^N_{j=1}\MSH\{ P_j, 
E_j\Delta_j \}$.

At this point, one is finally free to take, w.l.o.g.,\\[.46ex]
\small
\[
 W = \big (\TS \prod^n_{j=1} \MSH E_j\Delta_j\big )\MSH
 \TS \prod^N_{j=n+1} P_j  \quad
 \mbox{and}\quad 
 D_{{\mathscr C}_1} W\Fh = \TS \prod^N_{j=n+1} P_j \MSC
 \big ( D_{{\mathscr C}_1}{\TS \prod^n_{j=1}} \MSH E_j\Delta_j \big )  
 \MSH \Fh \MSB .
\]\\[-.93ex]
\normalsize
Referring back to (4.41), one readily sees that the 
corresponding summand of $T\MSH{\mathscr S}$ 
is $\MSA$real$\MSA$ and 
factors into
\[
 I_N[n+1]\prod^N_{j=n+1} \Big( \int^{1}_{-1} 
 \msfrac{ \sin (t_j v_j) }{v_j} \MSH d\MSA v_j \Big )\MSA ,
\]
wherein (4.36) applies and $I_N[n+1]$ is a certain iterated 
integral in the
variables $\{ v_j : j \in [1,k]\MSH , \MSAH j 
\notin  [n+1,N] \}$ having a ``level 
parameter'' $( v_{n+1},\ldots,v_N )$ 
that has been set equal to $(0,\ldots,0)$.  
The integrand in this more general expression is 
simply (4.41) with its $\{ n+1 \leqq j \leqq N \}$ (sine 
quotient) terms replaced by $1$ and its numerator entry
$D_{{\mathscr C}_1}(\varphi-\psi)$ replaced by
\begin{equation}\label{4.43}
D_{{\mathscr C}_1}\MSH(E_1\Delta_1)\cdots(E_n\Delta_n)\MSH \Fh\MSB.
\end{equation}
The interpretation of $(v_{n+1},\ldots,v_N)$ as a level parameter
makes eminently good sense $\MSZ$near$\MSZ$ $(0,\ldots,0)$ and 
fits with \S4.6. $\MSE$The ${\mathscr C}_2$ 
entries ``still standing'' in the integrand 
have \vspace*{-.08ex}value
\[
\prod^{n}_{j=1} \Big ( \msfrac{\sin (t_j v_j)}{v_j}\Big  )
\prod^{M}_{j=m+1} \Big (\msfrac{\sin (t_j v_j)}{v_j} \Big )\MSB .
\vspace*{-.08ex}
\]
Note that, in the second half of this product, 
one has $|v_j| > 1$;  cf. (iii).

With $W$ fixed, we now form the $\MSH$absolute value$\MSH$ of our 
$T{\mathscr S}\MSB$-$\MSB$summand and apply relations (4.41), (4.36), 
(4.43), and (4.27) with $\ell = n$ and $h = 1$ (or possibly $0$).
See also (4.24).  After reviewing definition (4.38), dividing  
${\mathscr C}_1$ into two chunks (\`{a} la $s / b$), and doing 
a bit of 
slow but elementary bookkeeping, one finds  
an obvious inclusion w.r.t. $\MSHZ\MSZ Q\MSAH$ over $\MSH {\mathscr S}\MSH$, 
wherein\vspace*{.46ex}
\[{\mathscr B}' = {\mathscr B} \cup \{n+1,\ldots,N\},
\MSE\MSB {\mathscr D}'={\mathscr D},
\MSE\MSB {\mathscr C}' = {\mathscr C}_1 \cup \{1,\ldots,n\} \cup
\{m+1,\ldots,M\}\MSH .\vspace*{.46ex}
\]
The key point is that $L \leqq k$ in (4.27)$\MSA$; $\MSB$note 
too that use of   
$\MSAH |\MSH\Fh|_{{\mathscr C}'}\MSAH$ in $\MSH$(4.42)
actually corresponds to making a convenient overshoot in 
the ${\mathscr M}_{\sigma}\MSAH$-$\MSAH$terms.  

This completes the proof of (4.40) with $F_j$ in 
place of $\Omega_j$. $\MSE$\bx \\[-5.1ex]

\end{proof}

By restricting one's attention to those $T_{\ast}$ having
${\mathscr C}_2 = \phi$, it is immediately seen that the 
${\mathscr B}'{\mathscr D}'{\mathscr C}'$$\MSAH$-$\MSAH$process 
``visits'' every portion of $\MSA Q\MSA$ as $\MSA{\mathscr S}\MSA$ 
varies. $\MSD$It also goes without saying here that 
the ``1'' appearing in 
${\mathscr C}_b(\vec{v})\MSA,\MSA {\mathscr C}_s(\vec{v})
\MSA,\MSA {\mathscr E}_{{\mathscr C}}(\vec{v})\MSA,\MSA$ 
etc$\MSA$, \linebreak can be replaced by any other positive 
constant $\tau$, $\MSA$as need be.\vspace*{.1ex}

\subsection{}
Just as in the case of ${\mathbb R}^1$, the 
$\MSH$single$\MSH$ most common application ``of 
all {\linebreak} this'' is 
a result that ensues nearly 
immediately from Corollary 4.12; $\MSH$viz., the 
\pagebreak 
central limit theorem ([Cr, p.112])   
for a sum of $N$ identically 
distributed, independent, vector-valued random 
variables $\MSA\vec{X}_j\MSA$.

One simply looks at
$\varphi_n(\vec{v}) = \varphi( \vec{v} / {\sqrt{n}} )^n $, 
$\MSH$where $\varphi$ is the characteristic function of the common
distribution function $F(\vec{x})$. $\MSB$We assume herein that  
\vspace*{.24ex} 
\begin{equation}\label{4.44}
\int_{{\mathbb R}^k}\,\vec{x}\,dF(\vec{x}) = \vec{0}
\MSD\MSE\quad\mbox{and}\quad\MSE\MSD
\Rank \Big ( \int_{{\mathbb R}^k}\,x_j x_\ell\MSC dF \Big) = k \MSC .
\vspace*{.24ex}
\end{equation}
($\MSH$Cf. $\MSZ\MSZ$(4.10) $\MSA$and$\MSA$ 
[Cr, p.$\MSH$109$\MSE$(145)$\MSH$]$\MSA$.) $\MSD$The key 
technical ingredient in verifying \linebreak hypothesis (ii) in 
Corollary 4.12 is the rudimentary 
fact \vspace*{.75ex}that
\[
\hspace{-1.4em}e^{it} = 1 + it + \tfrac{1}{2} (it)^2 + 
O(|t|^2)\min \{1, |t| \} \hspace{1.3em} 
\mbox{for $\hspace{0.9em} t \in {\mathbb R}\MSE$.}\vspace*{.75ex}
\]
In condition (i), one naturally takes $\alpha$ equal to $2$.

\vspace*{-.31ex}
\subsection{}
$\MSZ$From the standpoint of modern probability theory,  
estimates like (4.12), (4.29), $\MSH$(4.35), and   
$\MSH$(4.40) $\MSH$[with genesis in Lemma 4.3]$\MSH$ have 
a structural style that is 
slightly dated.    
$\MSH$In \S 5.13, we'll 
cite a $\MSHZ$few $\MSHZ$references $\MSHZ$that 
$\MSHZ$help $\MSZ$to $\MSZ$connect things to the more recent 
literature$\MSE$-$\MSD$and  
to smoothing ideas of 
other sorts. \vspace*{.39ex}

\section{More on The Central Limit Theorem}

\renewcommand{\theequation}{5.\arabic{equation}}
\setcounter{equation}{0}

\subsection{}
With its rich history and multitude of formulations, the central
limit theorem ([Cr, Fel1, Fel2, Usp, Po2]) is aptly regarded 
as one of the main results in classical probability theory.

In the first part of this section, we'll use Corollary 4.12 to 
quickly establish two [low-level] variants  
of the C.L.T. having particular 
relevance for our later work with 
logarithms of $L$$\MSA$-$\MSA$functions. 
$\MSB$To keep matters simple, we'll formulate things in a Liapounov 
style and restrict ourselves to contexts where, after 
rescaling, the underlying sequence of \vspace*{-.20ex} independent random 
variables $\,\{X_j\}^\infty_{j=1}\,$ has a common distribution 
function $\MSA F(z)\MSA$ 
on $\MSB{\mathbb C}\MSB$ ($\MSB\cong\MSB {\mathbb R}^2\MSA$).

\subsection{}
It is helpful to begin by recalling several  
very basic inequalities. 
\,For this purpose, let $n\geqq 0$,  
$N \geqq 1,\,\, 
m \geqq 2,\,\, \lambda \geqq 1$, 
and \,$x_j \in [0, \infty)$. \,Put
\[
s_N = (x^2_1 + \dots + x^2_N)^{1 / 2}\, , \quad B_N =
 \max \{x_1,\dots, x_N\}\, ,
\]
and let $\Theta$ signify a complex number (not always 
the same) having modulus no bigger than 1. \,We then have:

\small
\begin{equation}\label{5.1} 
e^{it} = \sum^n_{k=0} \frac{(it)^k}{k!} + \Theta\, 
\frac{|t|^{n+1}}{(n+1)!}
\quad \mbox{for \, $t \in {\mathbb R}$}
\end{equation}
\begin{equation}\label{5.2}
\Log (1+z) = z+\Theta |z|^2 \quad \mbox{for\, $|z| 
\leqq \frac{1}{2}$}
\end{equation}
\begin{equation}\label{5.3}
B_m \leqq \big (\sum^m_1 x^\lambda_j \big )^{\frac{1}{\lambda}}
 \leqq \sum^m_{j=1} x_j \leqq m^{1-\frac{1}{\lambda}}
 \big (\sum^m_1 {x_j}^\lambda\big )^{\frac{1}{\lambda}} \leqq mB_m
\vspace*{-.31ex}
\end{equation}
\begin{equation}\label{5.4}
\Big (\frac{B_N}{s_N}\Big )^3 \leqq \frac{1}{s^3_N}
 \sum^N_{j=1} x_j^3 \leqq \frac{B_N}{s_N} \quad
 \mbox{anytime}\quad s_N \neq 0\MSA .\vspace*{.31ex}
\end{equation}
\normalsize
Relation (5.1) follows from the identity
\small
\[
e^{it} = 1 + \int^t_0 ie^{iu} du
\]
\normalsize
by a trivial induction. At the same \vspace*{.25ex}time, one also sees that
\small
\begin{equation}\label{5.5}
e^{it} = \sum^n_{k=0} \frac{(it)^k}{k!} + 
\Theta \frac{2^{1-\omega} |t|^{n+\omega}}{(1)(1+\omega)
...(n+\omega)}\vspace*{.25ex}
\end{equation}
\normalsize
for any $\omega \in [0,1)$. The remainder term in (5.1) 
thus admits the \vspace*{.25ex}alternate format
\small
\[
\hspace*{11.37em}\Theta \min \Big \{\frac{|t|^{n+1}}{(n+1)!}\, ,
\, \frac{2|t|^n}{n!} \Big \}\, . 
\vspace*{.40ex}\hspace{11.37em}\mbox{(5.1bis)}
\]
\normalsize
Notice too that, in relation (5.3), the first and last terms 
simply correspond to taking $\lambda = \infty$ for 
each fixed $m$. 

\subsection{}
Let $\{X_j\}^\infty_{j=1}$ now be any sequence of 
\emph{complex-valued}$\MSH$, mutually independent random variables 
having a $\MSH$common$\MSH$ distribution function $\,F(z)\,$ 
on $\,{\mathbb C}\,$. \linebreak  Assume that $F$ satisfies
\begin{equation}\label{5.6}
\int_{\mathbb C}\,|z|^3\, dF(z)= \rho^3 < \infty
\end{equation}
in addition to \vspace*{.31ex}
\begin{equation}\label{5.7}
\int_{\mathbb C}\,z\,dF = 0\MSC,\quad \int_{\mathbb C}\,|z|^2\,dF = 
2 \beta^2 \MSC,\quad \int_{\mathbb C}\,z^2\,dF = 0 \hspace{1.2em} 
\mbox{with $\MSE \beta > 0 \MSE\MSE$}.
\end{equation}
(Cf. Remark 5.10 concerning the $z^2$$\MSB$-$\MSB$integral.)  
Put $z = x\,+\,iy$ as usual and then 
define the characteristic function of $F$ by
writing 
\begin{equation}\label{5.8}
\varphi(\xi) = \int_{\mathbb C}\,\exp [i\Ree(\overline{\xi}z)]\,
dF(z)\hspace{1.5em}\mbox{for $\MSB \xi \in {\mathbb C}\MSE\MSE$}.
\end{equation}
Taking $\xi = v_1\,+\,i v_2$ gives 
the standard ${\mathbb R}^2\MSB$--$\MSB$version. Similarly for any 
other distribution function $\MSB F_{arb}\MSB$ on $\,{\mathbb C}\,$.  
The RHS of (5.8) is, of course, nothing other than 
$\MSC{\mathbb E}[\exp(i\Ree(\overline{\xi}\MSA Y)]\MSC$ with $Y=X_j$\,.
\begin{theorem}
With $F$ and $X_j$ as above, let 
$\{b_j\}^\infty_{j=1}$ be any sequence of complex numbers such 
that $b_j \not\equiv 0$ \vspace*{-.62ex}and 
 \begin{equation}\label{5.9}
\lim_{N\rightarrow \infty}\MSD \sum^N_{j=1} \msfrac{|b_j|^3}
{s^3_N}= 0 \MSA ,
\end{equation}
wherein $s_N  = (|b_1|^2+ \cdots + |b_N|^2)^{1/2}$. 
$\MSD$For $t \in {\mathbb R}$, let ${\mathscr A}_t = (-\infty,t]$. 
In the limit of large $N$, the random \vspace*{-.50ex}variable
\[
T_N \equiv \msfrac{1}{s_N}\, \sum^N_{j=1}\,b_{j}X_{j}
\]
will then become distributed like a complex \vspace*{.75ex}Gaussian; 
\,i.e., 
\begin{equation}\label{5.10}
\Pr \{ T_N \in {\mathscr A}_x \times {\mathscr A}_y \} \,\rightarrow\, 
\big (\msfrac{1}{\beta \sqrt{2\pi}} \big)^2\,\int
^x_{-\infty}\,\int^y_{-\infty}\,\exp 
\big (-\msfrac{u^2+v^2}{2\beta^2} \big )\,dv\,du \vspace*{.75ex}
\end{equation}
for every $(x,y) \in {\mathbb R}^2$. \,The convergence  
will be uniform w.r.t. $x$ and $y$.
\end{theorem}

\begin{proof}
By H\"{o}lder's inequality, ${\mathbb E} (|T_N|) \leqq (2 \beta^2)^{1 / 2}
\leqq \rho$.  In view of 
Corollary 4.12, it suffices to check that one has
\[
\varphi_N(\xi) \, \rightrightarrows \,\, 
\exp(-\sfrac{1}{2} \MSA \beta^2 |\xi|^2)
\]
on every finite disk  $\{ |\xi| \leqq A \}$.  Here $\varphi_N$ is the 
characteristic function of $T_N$.  Since the
$X_j$ are independent, one immediately gets
\[
\varphi_N(\xi) = \prod^N_{j=1} \varphi ({\mathcal U}_j),
\quad\MSE {\mathcal U}_j \equiv \msfrac{\overline{b}_j}{s_N}\MSA\xi \, .
\]
Using (5.1) and conditions (5.6)+(5.7), we quickly see that
\begin{align}
\varphi_N(\xi) & =
\prod^N_{j=1} \Big ( 1- \sfrac{1}{2} \beta^2 \Ree^2 ({\mathcal U}_j)-
\sfrac{1}{2} \beta^2  \Imm^2 ({\mathcal U}_j) + \sfrac{1}{6}\MSA\Theta
\MSA\rho^3\MSA |{\mathcal U}_j|^3 \Big )\nonumber\\
& = \prod^N_{j=1} \Big ( 1-\sfrac{1}{2} \beta^2 |{\mathcal U}_j|^2 +
\sfrac{1}{6} \MSA \Theta\MSA \rho^3 \MSA |{\mathcal U}_j|^3 \Big )\,.
\end{align}

Letting $B_N = \max \{|b_1|, \dots , |b_N|\}$, $\MSB$one also knows by  
(5.4) that
\[
\hspace{14.11em}\msfrac{B_N}{s_N} \rightarrow 0\hspace{14.11em}(5.9')
\]
is {\it equivalent} to (5.9).  With this point noted,  
we now keep $N$ large enough to 
have, say, 
\[
4 \beta^2 \MSB \Big ( \frac{A B_N}{s_N}\Big )^2 + 
2 \MSB \sum^N_{j=1} 
\,\Big ( \rho \frac{A |b_j|}{s_N} \Big )^3 \, < 1\, .
\]

Bearing in mind that $|{\mathcal U}_j| \leqq A B_N/s_N$, 
\,for \emph{some} branch of the logarithm we then get:
{\allowdisplaybreaks
\begin{align*}
\log \varphi_N (\xi)\, & =\,
\sum^N_{j=1} \, \Big (\!- \sfrac{1}{2}\,\beta^2 |{\mathcal U}_j|^2 + 
\sfrac{1}{6}\MSA\Theta\MSA\rho^3\MSA|{\mathcal U}_j|^3 \Big )\\
& \quad + \sum^N_{j=1} \, \Theta \Big (\sfrac{1}{2}\,\beta^2
|{\mathcal U}_j|^2 + \sfrac{1}{6}\MSA\rho^3\MSA |{\mathcal U}_j|^3 
\Big )^2\\
& = - \sfrac{1}{2}\beta^2 \, \sum^N_{j=1} \, |{\mathcal U}_j|^2 +
\sfrac{1}{6}\MSA\Theta\MSA \sum^N_{j=1}\,\rho^3
|{\mathcal U}_j|^3\\
& \quad + \sfrac{1}{2}\MSA\Theta \, \sum^N_{j=1} \, 
\big ( \beta |{\mathcal U}_j| \big )^4 +
\sfrac{1}{18}\MSA\Theta\MSA \sum^N_{j=1} \, 
\big ( \rho^3\MSA|{\mathcal U}_j|^3 \big )^2 \\
& = - \sfrac{1}{2}\beta^2 \,\sum^N_{j=1} \,
|{\mathcal U}_j|^2 + \sfrac{1}{6}\MSA\Theta\MSA
\sum^N_{j=1} \,\rho^3 |{\mathcal U}_j|^3\\
& \quad + \sfrac{1}{4}\MSA\Theta\, 
\sum^N_{j=1}\, \Big (  \beta^3 \MSA |{\mathcal U}_j|^3
 +  \rho^3 |{\mathcal U}_j|^3  \Big )\qquad  \{\beta \leqq \rho \} \\
& = - \sfrac{1}{2}\beta^2 \,\sum^N_{j=1} \,|{\mathcal U}_j|^2 +
\sfrac{2}{3}\Theta\,\sum^N_{j=1}\,\rho^3 |{\mathcal U}_j|^3 \\
& = - \sfrac{1}{2}\beta^2 \MSA |\xi|^2 +  \MSA\sfrac{2}{3}\Theta 
\rho^3 A^3\,\sum^N_{j=1}\,\msfrac{|b_j|^3}{s_N^3} \,\, . 
\end{align*}  }
Upon exponentiating and letting $N \rightarrow \infty$, the desired
convergence of $\varphi_N (\xi)$ follows at once. $\MSC$\bx
\end{proof}

Prior to continuing, it is convenient to pause for several comments
concerning the ${\mathbb R}$$\MSB$--$\MSB$analog of Theorem 5.4.  
To this end, let $E(x)$ be any probability distribution 
on ${\mathbb R}$ satisfying an analog of (5.6)+(5.7) with 
respective values $\{ \rho_{E}^3\MSAH; \,0, \beta^2, \beta^2 \}$.  
Let $\{ ({\mathcal X}_j, {\mathscr Y}_j) \}^\infty_{j=1}$ be 
any {\it i.i.d.} $\MSZ\MSZ$sequence of random variables on ${\mathbb R}^2$ 
associated with the product measure $dE \times dE$.  The 
corresponding $F(z)$ clearly satisfies (5.6)+(5.7).
The following result is now immediate either 
by imitating the proof of Theorem 5.4 $\MSB$or$\MSB$ simply 
putting 
$y = \infty$ in (5.10).\vspace*{.15ex}
\begin{corollary}\,(The ${\mathbb R}$$\MSB$--$\MSB$analog of Theorem 5.4.)
Given $E$ and independent random 
variables $\{ {\mathcal X}_j \}^\infty_{j=1}$ as above.  $\MSAH$Let 
$\{b_j\}^\infty_{j=1}$ be any sequence of real  
numbers ($\MSA\not\equiv 0$) satisfying condition (5.9).  
Write\vspace*{-.31ex}
\[
T_N \equiv \msfrac{1}{s_N}\,\sum^N_{j=1}\,b_j {\mathcal X}_j \vspace*{-.31ex}\MSC .
\]
In the limit of large $N$, we then have
\[
\Pr \{ T_N \leqq x \} \rightarrow  
\msfrac{1}{\beta \sqrt{2\pi}}\MSA \int^x_{-\infty}\MSA
\exp \big (-\msfrac{u^2}{2\beta^2} \big )\,du
\]
for every $x \in {\mathbb R}$.  The convergence will be \vspace*{.15ex} 
uniform w.r.t. $x$.
\end{corollary}

Corollary 5.5 dates back to Liapounov (1901); cf. [Lia].  
Either of the indicated proofs immediately adapts to 
encompass the nominally broader framework 
of $(2+\delta)$$\MSB\MSH$-$\MSB$moments 
and \emph{non}-identically 
distributed, $\beta^2$$\MSA$-$\MSA$normalized  
${\mathcal X}_j$ that one traditionally uses in 
stating this result. Cf. estimate (5.5) and [Usp, p.284], 
\,[Cr, \S VI.4], \,[Bil, p.371$\MSH$(27.16)].\footnote{The 
numerators in (5.9) are now modified to include 
appropriate $\rho^{2+\delta}_j$ factors arising 
from the measures $dE_j$.  (Cf. 
the many $\rho^3 |{\mathcal U}_j|^3$ terms 
in the proof of Theorem 5.4.)}     
$\MSB$For effective versions of Corollary 5.5,\linebreak 
see, e.g., [Ess, p.43] or [Fel2, p.544].  A quick 
look at [Fel3, theorem 1] is also illuminating.\vspace*{.50ex}

\begin{remark}
When the variables ${\mathcal X}_j$ have their support restricted to 
a finite interval $[-\Delta, \Delta]$, it is also feasible 
to obtain Corollary 5.5 by a consideration of $T_N$ \!-\! moments. 
Cf. [Bil, pp.407 (example 30.1) \!-- \!410]. Note that, by (5.4), 
Billingsley's limit hypothesis (30.5) basically coincides with (5.9). 

Markov showed in 1913 that, with a bit of ingenuity in forming 
truncated variables, the moment method could be successfully 
recast  
(i.e., ``partitioned'') to yield 
Corollary 5.5 in full generality. $\MSC$Cf. [Usp, pp.388--395] 
and [Bil, p.600 (30.1)].  
$\MSC$An analogous splitting will 
come up later in connection with logarithms of Euler products.
\end{remark}

\subsection{}
In our second C.L.T. variant, we employ a kind 
of vector ``twist'' to ``lift'' the Gaussian of 
Theorem 5.4 $\MSA$up$\MSA$ to a related one 
on ${\mathbb C}^J$.  The twist is induced by replacing 
the earlier coefficients $\MSH b_n \MSH$   
(note the change of 
index $j \MSH \hookrightarrow \MSH n$) 
with \linebreak  a string of appropriately-behaved   
column vectors $\MSH\vec{b}_n \in \MSH{\mathbb C}^J \MSH$.

To keep the new Gaussian as unencumbered \vspace*{-.07ex}as 
possible, we normalize things by insisting that  
\begin{equation}\label{5.12}
\lim_{N\rightarrow \infty}\MSD \msfrac{1}{\sigma_N^2}\, \sum^N_{n=1}\,
\big( b_{n,j} \big) \big(\MSA \overline{b}_{n,\ell} \big)^t 
= \big [\beta_j\, \delta_{j\ell} \big ]\vspace*{.93ex} 
\end{equation}
hold in the sense of matrix addition for certain 
$\beta_j>0$ and certain monotonically increasing positive  
scaling factors $\sigma_N$.   
$\MSB$[$\MSA$Observe that each $J\times J$ 
summand in \linebreak (5.12) 
is nonnegative-definite$\MSH$; $\MSB$the thought of transforming (5.12) 
under a {\it uni-{\linebreak}tary} change of basis is 
thus present here \vspace*{.045ex} virtually 
ab initio$\MSH$. $\MSB$Notice $\MSA$too$\MSA$ 
that each summand is $\MSAH$preserved$\MSAH$ under 
the ``local'' action     
$\MSB \vec{b}_n \MSA \mapsto\MSA \exp(i\omega_n)
\vec{b}_n$.$\MSB$]\vspace*{.62ex}

\begin{theorem}
Let $\{X_n\}^\infty_{n=1}$ be any sequence of complex-valued 
independent random variables having a common distribution 
function $F(z)$, where $F$ satis-{\linebreak}fies 
both (5.6) and (5.7).  
Let $\MSA\{b_{n,j},\, \sigma_N,\,\beta_j \}\MSA$ be as 
in (5.12). Assume that
\begin{equation}\label{5.13}
\lim_{N\rightarrow\infty}\MSD \sum^N_{n=1} \, 
\sum^J_{j=1} \, \frac{|b_{n,j}|^3}
{\sigma_N^3} = 0\,.
\end{equation}
In the limit of large $N$, the ${\mathbb C}^J$-\,valued random
variable
\[
T_N \equiv \,\msfrac{1}{\sigma_N} \sum^N_{n=1}\,X_n ( b_{n,j} )
\] 
then becomes distributed like a multivariate Gaussian
\\[.15ex]
\begin{equation}\label{5.14}
\int_{\mathcal E} \,{\TS \prod^J_{j=1}}\, \msfrac{1}{2\pi\beta_j\beta^{2}} 
 \exp \big(-\msfrac{|w_j|^2}{2\beta_j\beta^{2}}\big )\;
{\TS \prod^J_{j=1}}\, du_j dv_j
\end{equation}
\\[.15ex]
wherein ${\mathcal E} \subseteqq {\mathbb C}^J$ 
and $w_j \equiv u_j+ iv_j$ corresponds 
to the $j^{\rm th}$ component of $T_N$.
\end{theorem}

\begin{proof}
One simply modifies the approach used over ${\mathbb C}^{1}$.    
We need to check [in an obvious notation] that\vspace*{.20ex}
\[
\varphi_N (\zeta_1, \kappa_1, \zeta_2, 
\kappa_2,\dots , \zeta_J, \kappa_J)\;
\rightrightarrows \; {\TS\prod^J_{j=1}} \, \exp 
\big (\!- \sfrac{1}{2} \, \beta_j\beta^{2}\MSH (\zeta^2_j + \kappa^2_j)\big )
\vspace*{.31ex}
\]
holds on every compact subset of ${\mathbb R}^{2J}$. To 
verify this, it is natural to write 
$\xi_j = \zeta_j+ i\kappa_j$ and then look at\vspace*{.15ex}
\begin{align*}
\varphi_N(\xi_1,\dots, \xi_J)\, &=\,
{\mathbb E}\big [\exp (i \,{\TS \sum^J_{j=1}} 
\Ree (\overline{\xi}_j\MSA T_{N,j}))\big ]\\[.31ex]
&=\,{\mathbb E}\big [ \exp(i \MSH\msfrac{1}{\sigma_N}\MSH
{\TS \sum^J_{j=1}}\MSD{\TS \sum^N_{n=1}}\Ree (\overline{\xi}_j b_{n,j} X_n)
\big ]\vspace*{.15ex}
\end{align*}
on the multidisk $\{|\xi_1| \leqq A,\dots, |\xi_J| \leqq A\}$. Putting
\[
\varphi (\xi) = \int_{{\mathbb C}} \exp 
[i \Ree (\overline{\xi}\MSA z)]dF(z) 
\]
as in (5.8) and using the independence of $X_n$, we first see that
\[
\varphi_N(\xi_1,\dots, \xi_J) = \prod^N_{n=1} \varphi ({\mathcal U}_n),
\quad\MSE  {\mathcal U}_n \equiv \msfrac{1}{\sigma_N}\, \sum^J_{j=1}\,
 \overline{b}_{n,j}\MSA\xi_j\, .
\]
From (5.11), we then get
\[
\varphi_N(\xi_1,\dots, \xi_J)  = 
 \prod^N_{n=1} \Big ( 1-\sfrac{1}{2}\beta^2 |{\mathcal U}_n|^2 +
\sfrac{1}{6} \MSA \Theta\MSA \rho^3 \MSA |{\mathcal U}_n|^3 \Big )\,. 
\]

By applying (5.3) with $x_j = |w_j|$, one 
immediately \vspace*{.31ex}checks that
\begin{equation}\label{5.15}
\frac{1}{J} \MSC\leqq\MSC \frac{\|\vec{w}\|_{\infty}}
{\|\vec{w}\|_{1 \MSB}} \MSC\leqq\MSC \frac{\|\vec{w}\|_{\lambda}}
{\|\vec{w}\|_1} \MSC\leqq\MSC 1 \vspace*{.31ex}
\end{equation}
for any $\vec{w} \in {\mathbb C}^J$ and $1 \leqq \lambda < \infty$. 
$\MSC$Temporarily fixing any $\mu \in [1,\infty)$, we now write   
$\MSB C_n = \|\vec{b}_n\|_{\mu}\MSB$ and avail ourselves of
(5.4)$\MSAH$+$\MSAH$(5.15) to see that\vspace{.31ex}
\[
\hspace{13.82em}\msfrac{D_N}{\sigma_N} 
\rightarrow 0 \vspace*{.48ex}\hspace{13.82em}(5.13')
\]
with $D_N = \max \{C_1,..., C_N\}$ is {\it equivalent} to 
(5.13).  $\MSC$Cf. (5.9$'$).

At this point, we set $\mu = 1$, \,so 
that $|{\mathcal U}_n| \leqq A C_n / \sigma_{N}$, \,and then 
keep $N$ large enough to have
\[
4 \beta^2 \Big ( \frac{A D_N}{\sigma_N} \Big )^2 + 
2 \MSA \sum^N_{n=1} \Big ( \rho \frac{A\MSH C_n}
{\sigma_N}\Big )^3 < 1\MSC .
\]
For some branch of the logarithm, we finally get 
\pagebreak
\begin{align*}
\log \varphi_N (\xi_1,\dots, \xi_J)\, & =\,
\sum^N_{n=1} \, \Big ( \!-\sfrac{1}{2}\beta^2 
|{\mathcal U}_n|^2 + \sfrac{1}{6}\MSA\Theta\MSA\rho^3\MSA
|{\mathcal U}_n|^3 \Big )\\
& \quad + \sum^N_{n=1} \, \Theta \Big ( \sfrac{1}{2}\beta^2
|{\mathcal U}_n|^2 + \sfrac{1}{6}\MSA\rho^3\MSA |{\mathcal U}_n|^3 
\Big )^2\\
& = - \sfrac{1}{2}\beta^2\MSA \sum^N_{n=1} \MSB |{\mathcal U}_n|^2 +
\sfrac{2}{3}\MSA\Theta\MSA\sum^N_{n=1}\,
\rho^3 |{\mathcal U}_n|^3\\
& = - \sfrac{1}{2}\beta^2 \MSA\sum^N_{n=1} \MSA
|{\mathcal U}_n|^2 + \sfrac{2}{3}\MSA\Theta\MSA\rho^3\MSA A^3
\MSA\sum^N_{n=1} \MSA \frac{C_n^3}{\sigma_N^3}  \vspace*{.62ex}
\end{align*}
exactly as in the proof of Theorem 5.4. $\MSD$Since $C_n 
\leqq J^{2/ 3}\MSA \|\vec{b}_n\|_{3}$ by (5.3), $\MSH$and
\[
\sfrac{1}{2}\beta^2 \MSA \sum^N_{n=1} \MSB |{\mathcal U}_n|^2 =
\sfrac{1}{2}\beta^2 \MSA \sum^J_{j=1}\, \sum^J_{\ell=1}\,
\overline{\xi}_j \MSH \xi_\ell \big (\msfrac{1}{\sigma_N^2}\,
\sum^N_{n=1} \, b_{n,j}\, \overline{b}_{n,\ell} \big )\MSE ,
\]
the desired convergence of $\varphi_N (\xi_1,\dots, \xi_J)$
follows at once. $\MSC$ \bx \vspace*{.31ex}
\end{proof}

In situations where the coefficients $\MSH b_{n,j}\MSH$ are 
{\it real} $\MSA$and$\MSH$ the probability distri-{\linebreak}bution 
$\MSH E(x)\MSH$ is as in Corollary 5.5, $\MSAH$a temporary passage 
to $dF \MSZ =\MSHZ dE \times dE$ in Theorem $\MSH$5.8$\MSH$ 
immediately shows \vspace*{-.46ex} that
\[
S_N \equiv \msfrac{1}{\sigma_N}\MSC\sum^N_{n=1}\MSC
{\mathcal X}_n\MSH (b_{n,j})\vspace*{-.46ex}
\]
becomes Gaussian \vspace*{.31ex}distributed like
\begin{equation}\label{5.16}
\int_{\mathcal E}\MSC{\TS \prod^J_{j=1}}\MSC \msfrac{1}{\beta
\sqrt{2 \pi \beta_j}}\MSB \exp \big ( -\msfrac{u_j^2}
{2 \beta_j \beta^2} \big )\MSC {\TS \prod^J_{j=1}}\MSC du_j
\vspace*{.31ex}
\end{equation}
in the limit of large $N$ over ${\mathbb R}^J$. The case 
$J = 1$ reduces to a minor extension of \vspace*{.05ex}
Corollary 5.5.

\begin{remark}
Theorem 5.8 is clearly a bit more ``exotic'' than Theorem 5.4.  It
is important to note, however, that both results 
are easy corollaries of the (now standard) 
Lindeberg formulation of the 
$\MSA$C.L.T.$\MSA$ over ${\mathbb R}^{2J}\MSHZ$; $\MSA$condition (5.7) 
enables the key link from ${\mathbb C}^J$ to ${\mathbb R}^{2J}$. 
$\MSAH$Cf. [Cr, \S\S X.3, VI.3--4] and, $\MSZ$for instance, 
[Fel2, pp.260--261, 263$\MSD$(lines 5--6,$\MSC$11--16)].  
A quick look at the change-of-basis idea
utilized in [Ber, pp.44--48$\MSC$(top)] is also helpful.  Seen from
this perspective, Theorem 5.8 (with $J > 1$) is simply a form of
the C.L.T. in which the successive random variable summands have
their supports restricted to ``thin'' sets.  

One very slight advantage to   
our ``pedestrian level'' proofs of 
Theorems 5.4 and $\MSH$5.8$\MSH$ is that, when viewed in 
conjunction with 
Theorems 4.8$\MSA$A--$\MSA$4.8$\MSA$C, mat-{\linebreak}ters 
are $\MSZ$\emph{explicit} enough therein to furnish what would seem to be 
a convenient springboard for the 
development of some parallel results having 
well-controlled error terms, $\MSA$at least for box-like 
${\mathcal E}\MSH$.$\MSA$\footnote{In 
our later work with $L$-functions, $\MSA dF\MSA$ 
will just be ordinary Haar measure on $\{ |z| = 1 \}$.}
$\MSD\MSD$Cf. [Fel2, pp.544--545].
\end{remark}

\begin{remark}
$\MSC$Coming back to a point just touched on [concerning (5.7)], 
suppose for a few moments that, beyond having mean $0$, 
distribution function $F(z)$   
merely satisfies ${\mathbb E}(|z|^q) < \infty$ with some $q \geqq 3$. (Cf.
(5.6).) Let $\| \cdot \|$ denote the Euclidean norm on ${\mathbb C}^J$. 
In formulating Theorem 5.8, we have been guided by a desire
to remain in a structural setting where: 
\begin{tabbing}
\phantom{mi}(i) \=  matters are robust 
enough to produce a good limit distribution for\\  
   \> $T_N$ anytime (5.12) holds and 
$\max \{\| \vec{b}_1 \|,...,\| \vec{b}_N \|
\} /\sigma_N = o(1)\,$;\:\;and, \\
\phantom{m}(ii) \> \emph{in the process of that}$\MSH$, \,every  
${\mathbb R}^{2J}\MSAH$-$\MSAH$ style second moment of $T_N$ is seen\\  
       \> to converge to \emph{some} limit as 
$N \rightarrow \infty$\,.
\end{tabbing}
The essential point in (i) is that convergence should still occur
even if the \linebreak vectors $\MSA \vec{b}_n\MSA$ are 
multiplied by  
arbitrary phase factors $\MSAH \exp(i\omega_n)\MSAH$.

In connection with (ii), thanks to the algebraic \vspace*{.25ex}identity
\[
\left ( \begin{array}{ll}
x_1 x_2 & x_1 y_2\\
y_1 x_2 & y_1 y_2
\end{array}\right ) = \sfrac{1}{2} 
\left ( \begin{array}{rr}
1 & 1\\
-i & i
\end{array}\right )\MSA
\left ( \begin{array}{ll}
z_1 \overline{z}_2 & z_1 z_2\\
\overline{z}_1 \overline{z}_2 & \overline{z}_1 z_2
\end{array}\right )\MSA
\sfrac{1}{2}\MSA
\left ( \begin{array}{rr}
1 & i\\
1 & -i
\end{array}\right ), \vspace*{.25ex}
\]
we certainly want, e.g. $\MSZ$for $J=1$, the 
limiting \vspace*{.25ex}covariance 
\[
{\mathcal M} \equiv \lim_{N\rightarrow \infty}
\left ( \begin{array}{ll}
{\mathbb E}(T_N \overline{T}_N) & {\mathbb E} (T_N^2)\\[0.7ex]
{\mathbb E} (\overline{T}^2_N) & {\mathbb E} (\overline{T}_N T_N)
\end{array}\right ) \vspace*{.25ex}
\]
to \emph{exist}. $\MSC$Taking $|b_{n,1}|\equiv 1$ 
and $\sigma_N = \sqrt{N}$ produces
\[
{\mathcal M} = \lim_{N\rightarrow \infty}
\left ( \begin{array}{cc}
A & \frac{B}{N}\MSA C_N\\[0.7ex]
\frac{\overline{B}}{N}\MSA\overline{C}_N  & A
\end{array}\right )\, ,
\]
with
\\[-0.93ex]
\[
A = \int_{{\mathbb C}} |z|^2 dF(z), \quad 
B= \int_{{\mathbb C}} z^2 dF(z), \quad
\mbox{and}\quad
C_N=\sum^N_{n=1}\MSA b^2_{n,1}\, .
\]
\\[-2.75ex]

Putting $A= 2 \beta^2 > 0$ w.l.o.g., and considering the effect 
of $e^{i\omega_n}\MSB$-$\MSB$twists 
in ${\mathcal M}\MSA$, it is nearly self-evident that 
the coefficient $B$ can only be $0$. \,Similarly 
for $J>1$. \emph{For this reason: the normalization (5.7) is}  
\emph{basically forced on us in the present setting.}

Further corroboration of this point 
ensues from the simple observation 
that, when $q \geqq 4$ and assumption (5.12) 
holds$\MSH$, $\MSAH$convergence of $T_N$ in 
distribution  
automatically carries with it fulfillment of 
property (ii).  One checks this by bounding the 
expectations of $\Ree (T_{N,j})^4$ and
$\Imm (T_{N,j})^4$ for each $j$, $\MSH$and then exploiting the 
comment made immediately after the proof of Corollary 4.12.  
Since $\int z \MSA dF = 0\MSA$, the only 
partitions of $\MSH 4\MSH$ that contribute 
nontrivially here are $\MSH 4\MSH$ and $2+2\MSAH$; the 
aforementioned expectations \vspace*{.05ex}are thus majorized by
\[
\frac{\mbox{constant}}{\sigma_N^4}\bigg (\sum^N_{n=1}\MSA\|\vec{b}_n\|^2 
\bigg )\MSA \bigg ( \sum^N_{\ell=1}\MSA\|\vec{b}_{\ell}\|^2 \bigg )\MSB .
\vspace*{.46ex}
\]
Compare: $\MSH$[Bil, pp.409$\MSA$(middle)$\MSC$--$\MSC$410$\MSA$(top)].  
$\MSB$By making use of the more sophisticated 
Marcinkiewicz-Zygmund inequality (see [MZ1, \S8] or [MZ2, \S3]) 
and the convexity of $\MSA\Theta(u) = u^\beta\MSA$ for 
$\MSH\beta > 1\MSH$, the foregoing observation is readily seen 
to hold for $\MSH 2 < q  < 4 \MSH$ as well$\MSH$. 
\end{remark}

\subsection{}
To wrap things up, it is now helpful to ``make \vspace*{-.105ex} 
a slight change of gears.''  Once one has 
available a reasonable ${\mathbb R}^k\MSA$-$\MSB$analog of
Esseen's lemma, it becomes natural to contemplate finding  
$\MSA k\MSA$-$\MSA$variable    
counterparts for not only the clas-{\linebreak}sical  
Berry-Esseen-type error estimates in the 
C.L.T., $\MSA$but also the 
related {\it re}-{\linebreak}{\it finements}$\MSAH$   
of {\it uniform asymptotic type}$\MSA$ \vspace*{.05ex}traditionally 
associated with the names of Edgeworth and Cram\'{e}r.
$\MSD$See [Cr, \S X.3 (para 1)] for some early remarks on both issues,  
$\MSA$albeit from the standpoint of a slightly different type 
of smoothing, viz., 
[Cr, \S\S $\MSA$VII.3, $\MSB$VII.5--6].$\MSH$\footnote{In this 
connection, see also [Po2, pp.172$\MSD$(lines 4--18), 
177$\MSD$(para 2)].}  

Though this report is 
clearly $\MSHZ\MSZ$\emph{not} the place for any sort of
comprehensive discussion of these  
matters (the technical details being rather heavy even in the case of
Theorems 5.4 and 5.8), it does seem reasonable 
to offer  
at least a few remarks on what can
be achieved using 
Theorems 4.8$\MSA$A$\MSA$--$\MSA$4.8$\MSA$C in \vspace*{.05ex} 
the \emph{classi-{\linebreak}cal} $\MSAH$C.L.T. setting of \S$\MSA$4.14.  
$\MSAH$[The ready availability of a pre-existing literature here  
facilitates preparation of an $\MSH$intelligible$\MSH$ commentary 
within the span of just a few pages.]

\subsection{}
In the $\MSAH$i.i.d. $\MSHZ\MSZ$situation of \S4.14, one has
\[
\vec{T}_n = \sfrac{1}{\sqrt n}\MSA ( \vec{X}_1 + \cdots + \vec{X}_n )
\hspace{1.6em} \mbox{and} \hspace{1.6em} \psi(\vec{v}) = \exp\MSH ( 
-\vsfrac{1}{2} \TS \sum \MSA v_j \MSH a_{j\ell}\MSH v_\ell )\MSB,
\]
wherein $a_{j \ell} = {\mathbb E}(x_j x_\ell)\MSAH$;  cf. (4.44).  The 
characteristic function $\psi(\vec{v})$  
corresponds to that of 
a $k\MSC$-$\MSB\MSH$variate 
Gaussian ($\MSH$[Cr, $\MSHZ\MSZ$p.109$\MSB$(144)]$\MSH$) 
having covariance 
matrix $\MSA [\MSH a_{j\ell}\MSH]\MSA$. $\MSH$Anytime 
the initial probability distribution  
$F(\vec{x})\MSAH$ (of $\MSA$\S$\MSH$4.14)$\MSA$ satisfies
\begin{equation}\label{5.17}
\int_{{\mathbb R}^k} \MSA \big ( \| \vec{x} \|_{\infty} \big )^q \MSA dF
\MSA < \MSA \infty 
\end{equation}
with an appropriately big $q \geqq 3$, there is a concomitant expectation 
that, after a modicum of calculation and parameter optimization,  
Theorems 4.8$\MSH$A$\MSH$--$\MSA$4.8$\MSH$C will each be 
found capable of producing quantitative refinements in 
the C.L.T. analogous  
$\MSH$to $\MSH$those $\MSH$known $\MSH$to $\MSH$hold $\MSH$in $\MSH$the 
$\MSH$case $\MSH$of $\MSH$one$\MSA$-variable$\MSA$.  

In particular, $\MSB$insofar as the characteristic function $\varphi$ 
of $d\MSH F$ satisfies
\begin{equation}\label{5.18}
\limsup_{\| \vec{v} \| \rightarrow \infty} \MSAH |\varphi(\vec{v})| 
\MSA < \MSA 1
\end{equation}
[\emph{i.e.}, Cram\'{e}r's condition (C)], it should emerge rather 
quickly that there exist \linebreak asymptotic developments 
of $\MSAH$Edgeworth-type$\MSA$ over ${\mathbb R}^k$ comparable to 
those articulated in, say, [Cr, \S VII.5], [Ess, pp.48--52], 
and [Fel2, pp.539, 541$\MSAH$(bot-{\linebreak}tom)] 
for ${\mathbb R}^1$.  To help one appreciate the 
more formal side of things, $\MSA$a quick review 
of [Cr, pp.26$\MSB$(lemma 1), 71$\MSB$(lemma 2), 
74$\MSB$(lemma 3), 75$\MSB$(102), 81$\MSB$(C), \,85, 86$\MSAH$(101a)], 
[Ess, p.44$\MSAH$(lemma 2)], and [vBa, pp.72--74$\MSAH$(top)] is very 
useful.  The paper by von Bahr deals specifically 
with ${\mathbb R}^k$ and has the added advantage of addressing 
derivatives of $\MSA\varphi (\vec{v}/ \sqrt n )^{\MSH n}\MSA$ 
as well$\MSA$; cf. $\MSHZ$his Lemmas 2 
and 3.$\MSH$\footnote{$\MSH$Note that the $\alpha$ 
in Lemma 3 depends on the specifics of 
how (5.18) is fulfilled.} 

Just as in the case of ${\mathbb R}^1$, $\MSA$there are no real 
difficulties implementing these technical formalities $\MSZ$in 
$\MSZ$the $\MSZ$overall $\MSZ$calculation  
$\MSZ$\emph{until} $\MSHZ$it $\MSZ$comes time 
to bound the\linebreak  
final error term.  The 
essential question then shifts to how small  
a size can be achieved in the latter 
through an appropriate choice of parameters.  The 
\linebreak outcome in this will naturally depend on the 
particular type of smoothing \linebreak inequality 
that is available.

Theorem 2 in [vBa] shows that, under hypothesis (5.18), 
matters behave {\it precisely as one would expect}$\MSAH$   
given the analogous results over ${\mathbb R}^1$. $\MSC$von 
Bahr's \linebreak theorem serves to refine an earlier 
estimate (of Berry-Esseen type) obtained around twenty 
years earlier by H.$\MSD$Bergstr\"{o}m; $\MSB$see 
[Berg1, $\MSH$pp.$\MSH$109,$\MSD$121$\MSA$(top)] 
and [Berg2, $\MSA$p.40$\MSB$(D)].$\MSH$\footnote{$\MSH$It is 
illuminating at this juncture to also have a look at  
$\MSA$[Rao, $\MSH$p.$\MSH$360, $\MSH$theorem 2]$\MSAH$ and\linebreak  
the less polished, partial refinement obtained 
by [Dun, $\MSH$Sad] 
around 1966--68 utilizing Sadi-{\linebreak}kova's 
slightly $\MSZ$\emph{non}linear,  
but highly suggestive,  
two-dimensional variant of (3.1).} 

In perusing this work, what becomes evident $\MSZ$nearly 
$\MSZ$immediately $\MSZ$is that our Theorem 4.8 $\MSH$(with 
its $D_{\mathscr C}\MSB$-$\MSB$terms and related 
sum over ${\mathscr P}$)$\MSH$ manifests a certain   
\emph{structural similarity} with    
[vBa, pp.77(8), $\MSZ$75(lemma 4)]. $\MSB\MSH$Likewise 
for Theorem 4.8$\MSH$C.$\MSAH$\footnote{Readers 
for whom reference [vBa] is new may find it helpful to view 
this latter assertion in the more specific setting of  
[vBa, pp.$\MSH$74$\MSA$(line 5)($h_T$)(5), 75$\MSA$(top), 
76$\MSA$(bottom), 77$\MSA$(lines 1-10, 13-15, 18-19)].  
$\MSC$Compare: [Berr, p.127$\MSA$(28)], $\MSAH$[Fel2,  
pp.$\MSA$537$\MSA$(3.3)$\MSA$(3.6), 538$\MSA$(3.12)].}     
$\MSA$In proving Theorem 4.8$\MSH$C, we have 
clearly been influenced by several aspects of this 
parallelism. (Cf., in particular, 
[vBa, pp.76$\MSA$(7), $\MSZ$77$\MSA$(line 
14)]. $\MSA$The \linebreak 
decomposition of $\MSA I\MSA$ towards the middle  
of our proof can be seen as a natural counterpart 
of$\MSA$ p.76$\MSA$(7)$\MSH$.)

With these remarks as a backdrop,   
$\MSH$it should come as no big  
surprise that 
applying Theorem 4.8$\MSH$C in conjunction 
with [vBa, Lemma 3] is readily checked $\MSH$to $\MSH$lead  
$\MSH$to $\MSH$a $\MSH$result essentially identical   
$\MSH$to $\MSH$von Bahr's Theorem 2.  One simply 
takes\vspace*{.31ex}
\[
  \Omega_j = \msfrac{1}{\sqrt{k}}\MSH (\sqrt{n}\MSH)^{q-1}
\vspace*{.31ex}
\]
as in [vBa, p.77$\MSA$(lines 17--18)] with $q \geqq 3\MSA$, $\MSB$observes 
that\vspace*{.10ex}
\begin{equation}\label{5.19}
0 \MSB\leqq\MSB |t|^n \exp (-\vsfrac{1}{2}\beta 
\MSH |t|^\frac{2}{q-1} ) \MSB\leqq\MSB 
c_1(n,\beta,q)\MSB  \vspace*{.10ex}
\end{equation}
\begin{equation}\label{5.20}
0 \MSB < \MSB c_2(k,\psi) \MSB \leqq\MSB 
\frac{ \big ( { \TS \sum^k_{j=1} }\MSH |t_j|^n \big)^\psi}
{ {\TS \sum^k_{j=1} }|t_j|^{n \MSH \psi} }
\MSB \leqq\MSB  c_3(k,\psi) \MSB < \MSB \infty \vspace*{.10ex}
\end{equation}
for every $n \geqq 0\MSA, \MSAH \beta > 0\MSA, \MSAH \psi > 0\MSA$, 
$\MSB$and then focuses on estimating the $d\vec{v}\MSAH$-$\MSAH$integrals 
in (4.40) one partition at a time, $\MSA$aided by von Bahr's 
Lemma 3.  Writing
\[
{\mathscr C} = \{1,2,\MSZ...\MSA,\delta \}\hspace{.15em} ,\hspace{1.2em} 
{\mathscr D} = \{\delta +1,\MSZ...\MSA,r \} \hspace{.15em} ,\hspace{1.2em} 
{\mathscr B} = \{ r+1,\MSZ...\MSA,k \}\hspace{.15em},
\]
w.l.o.g., $\MSA$it is helpful to collapse the  
region of integration down to
\[
{\TS \prod^r_{j=1}}\MSH [-\Omega_j,\MSH \Omega_j] \times 
{\TS \prod^k_{j\MSA = \MSA r+1}}\MSA \{ v_j = 0 \} \vspace*{.46ex}
\]
and then introduce slabs for $j \in [1, \delta]\MSH$. $\MSA$In view of 
(4.39), there is no need to \linebreak look at anything worse than a 
$\MSA$first$\MSA$ derivative in Lemma 3. $\MSAH$Since 
$|v_j^\blacktriangle| \geqq 1\MSA$ and
\[
\int^{\Omega_j}_{-\Omega_j}\MSH \msfrac{1}{\Omega_j}
\MSA dv_j = 2 \hspace{.20em}, \hspace{1.50em}
\int^{\infty}_{-\infty}\MSH \exp ( -\vsfrac{1}{2}\beta 
|t|^{\frac{2}{q-1}} )\MSA dt 
\MSB < \MSB \infty \hspace{.20em},
\]
one immediately recovers Theorem 2 of [vBa]. 

Utilization of Theorem 4.8$\MSH$B in place of 4.8$\MSH$C clearly 
produces a \emph{counterpart} of Theorem 2 having 
an extra factor of 
size $\MSA$at most$\MSA$ $\MSH c_{10} (D+1)^k \MSH$ in the final 
error term.  (With a little more effort, $\MSH (D+1)^k \MSH$ can 
be replaced by $\MSA \log^{\MSA k-1\MSH}\MSH(D+e)\MSB$; $\MSA$cf. 
(5.21)+(5.22) below.)

In the case of Theorem 4.8$\MSH$A, matters are slightly more 
involved due to the need to select appropriate values for 
both $\MSAH \alpha\MSAH$ and $\MSAH \Delta\MSAH$ in (4.29)$\MSAH$.  
As will soon become apparent, $\MSZ$however$\MSH$, the choice 
of $\MSH\alpha \in (0,q]\MSH$ is 
largely immaterial \linebreak -- in that any modifications 
therein  
lead to $\MSA$effects$\MSA$ felt only at the level of the 
implied constant in the final error term. $\MSC$(That 
${\mathbb E}(\|\vec{T}_n\|^{\alpha}) = O(1)$ for every 
$\alpha$ in $\MSH [2,q]\MSH$ is an easy consequence of the  
Marcinkiewicz-Zygmund inequality$\MSH$; $\MSAH$cf. the last four lines 
of Remark 5.10. $\MSB\MSH$Setting $\alpha = 2$ is, of 
course, $\MSH$simplest.)

To \emph{bound}$\MSH$ (4.29), one 
starts by taking $\MSH \Omega_j \MSH$ as above 
and letting $\gamma = \vsfrac{2}{q-1} \MSH\MSB$. \linebreak We then 
note (using (5.20)) that\vspace*{.025ex}
\begin{equation}\label{5.21}
\int^{\Omega_1}_{-\Omega_1} \cdots \int^{\Omega_k}_{-\Omega_k} \MSA
\msfrac{\|\vec{v}\|^q \MSH \exp(-\beta\MSH \|\vec{v}\|^{\gamma})}
{|v^{\bullet}_1|\cdots|v^{\bullet}_k|}\MSA d\vec{v} \MSC \leqq\MSC 
c_4(k,q) I(q)I(0)^{k-1}\MSB , \vspace*{.25ex}
\end{equation}
wherein $\|\vec{v}\MSH\|$ is the standard 
Euclidean norm, $\MSAH \Delta > 1$, 
\[
I(m) = \int^{\infty}_{0}\MSH v^m\MSH
\frac{\exp(-\widehat{\beta} \MSH v^\gamma)}{v^{\bullet}}\MSA dv 
\hspace{1.25em}\mbox{for}\hspace{1.00em} m \geqq 0\MSE,
\]
and $\widehat{\beta}$ is some elementary fraction of $\beta$. $\MSC$Since 
$ 1/v^{\bullet} = \min\MSH (\Delta,\MSH 1/v)$ for $v > 0\MSH$,\linebreak 
it is natural to set $\MSA v = w/ \Delta\MSA$ 
in $I(m)\MSA$. $\MSB$By elementary 
calculus, this promptly gives \vspace*{.93ex}
\begin{equation}\label{5.22}
I(m) \MSB = \MSB O(1)\MSH [1 + \delta_{m0}\MSA (\log \Delta)
\MSH] \MSC , \vspace*{.93ex}
\end{equation}
with an implied constant depending solely 
on $\MSA(m,\MSB \beta, \MSB \gamma, \MSB k)\MSA$.  Taking account of 
Lemma 3 and page $\MSHZ$77$\MSD$(lines 17--22) in [vBa], things 
are now \emph{optimized} (modulo unimportant 
constants) $\MSA$by declaring 
$\MSC \Delta = 10\MSAH n^{(q-1)/ 2\alpha}\MSH$. $\MSA$The resulting 
error term matches that of von Bahr's Theorem 2 
apart from an extra factor of $(\log n)^{k-1}\MSA$. $\MSC$That 
is to say, $\MSA$with Theorem 4.8$\MSH$A, $\MSB$one ultimately 
obtains 
\[
|\MSH\mbox{Remainder}| \MSB \leqq \MSB C d(n) 
\Big ( \msfrac{1}{\sqrt{n}} \Big )^{q-2} \MSH ( \log n )
^{k-1}\MSC\MSH,
\]
with some $d(n) = o(1)$.  Exploitation of the final sentence 
in Theorem 4.8$\MSH$A produces the slightly weaker 
bound $\MSC C \MSH d(n)\MSH (1/ \sqrt{n})^{q-2-\nu}\MSA$ with
\[
\nu \MSB=\MSB k \msfrac{q-2}{q+k} \MSA\in\MSA (0,k] \MSB .
\]
\newpage 
\noindent Both error estimates seem eminently reasonable  
given the relative crudity of Theorem 4.8$\MSH$A. $\MSA$(The second   
bound is of interest mainly for large $q$.  In both cases, 
$\MSH$note that $\MSAH C\MSAH$ and $\MSAH d(n)\MSAH$ typically 
depend on $\{ k,\MSB q, \MSA F \}\MSH$; $\MSB$cf. (5.22).)
 
As a quick review of the foregoing considerations 
makes clear, there are \linebreak undoubtedly very  
similar expansions that hold in the setting of Theorems 5.4 
and 5.8$\MSH$ -- $\MSH$at least under hypothesis (5.18) and some 
suitable magnitude restrictions on both $\MSH\|\vec{b}_n\|\MSH$ 
and the remainder term in (5.12). $\MSH$Our subsequent 
work with $L\MSAH$-$\MSAH$functions 
will implicitly entail 
working out the details of at 
least one such case$\MSA$; $\MSB$see footnote 21$\MSH$.
$\MSA$(Notice incidentally that the 
measure in footnote 21 has $\varphi(\xi) =  J_{0}(|\xi|)$ 
for $\xi \in {\mathbb C}\MSB$; $\MSC$cf., e.g., 
[Fel2, p.523$\MSD$(7.8)]. $\MSC$In (5.18),  
$| \varphi(\xi) |$ will thus decay like $1/ \sqrt{|\xi|}\MSA$.)

\subsection{}
Modern treatments of multivariate Edgeworth-type 
expansions have \linebreak largely come to focus on regions 
and set-ups substantially more general than 
rectangles.  $\MSD$See, for instance, [BhR, 
pp.$\MSAH$208, 214$\MSA$(20.44)$\MSA$--$\MSA$215$\MSA$; 
$\MSC$52--54, 210--212]$\MSA$; also [vBa, 
pp.$\MSAH$85$\MSA$(line 5)$\MSA$--$\MSA$87$\MSA$(middle)].  

In connection with such results, it is helpful 
to keep two key facts in mind.  $\MSD$First: 
$\MSA$that in contexts where the relevant 
shapes ${\mathcal E}$ are smoothly bounded 
and overshoots in $q$ are a 
non-issue,\footnote{$\MSH$e.g., when $dF$ is 
compactly supported} the $\MSZ$\emph{uniformity} aspect 
of \S$\MSH$5.11--12 will invariably permit one to 
obtain \emph{some} results of Edgeworth-type simply by 
making inner/outer approximations based on rectangular 
grids and then optimizing the choice 
of grid size(s) only at the very end. $\MSD$Second: 
$\MSA$though Theorems 4.8--4.8$\MSH$C \linebreak have
intentionally been kept 
quite close to (3.1) format-wise$\MSH$,  
$\MSH$it has been found increasingly $\MSZ$expedient  
over the years for work involving 
multivariate normal \linebreak approximation to$\MSH$ make$\MSH$ 
use$\MSH$ of$\MSAH$ smoothing$\MSA$ inequalities$\MSH$ having$\MSH$ 
less$\MSH$ rigid, {\it convo-{\linebreak}lution$\MSB$-$\MSH$based}   
$\MSAH$formats.   
$\MSAH$See [Ess, pp.$\MSH$101, 104$\MSA$(53)(56), 
110$\MSA$(lines 5--9)]\footnote{(lines 5--9 are nicely 
supplemented by [vBa2, pp.$\MSA$63$\MSA$(8), 64(lines 12--19), 67(C)--68])},  
$\MSAH$[Berg2,
p.47$\MSA$(lemma 8)],
$\MSAH$[vBa, pp.79$\MSA$(9), 81$\MSA$(11)], 
$\MSAH$[Saz, pp.184$\MSA$(lemma 2), 195$\MSA$(lemma 6)], 
and [BhR, pp.$\MSB$94$\MSA$(11.13), 
98$\MSA$(11.26), 102$\MSA$(top), 210$\MSA$(top), 
84--88, 274--278] 
for a few examples, in chronological order.       
$\MSA$Pages 274--278 in [BhR] relate to 
the much-discussed $\MSHZ$\emph{Stein method}$\MSHZ$ of establishing 
Berry-Esseen type bounds.\footnote{$\MSH$(Compare 
266 lines 12--19$\MSH$, and note the missing $\MSA \ast K\MSH$  
on 276 line 7.)}  
A concomitant look at [CGS, 
pp.$\MSB$336$\MSA$(12.66)--337$\MSA$(top), 4--6, 
16$\MSA$(2.13)--18$\MSA$(top), 20$\MSA$(2.30), 46--48] 
and [Bar, pp.$\MSB$293--294] serves to provide 
some valuable additional perspective on these more recent        
matters.\vspace*{-0.93ex}

\section*{\hspace{-0.4em}{\fontsize{10.333}{10.333}\selectfont 
\S6. Acknowledgments}}
\vspace*{-.31ex}
{\fontsize{9.9pt}{10.45pt}\selectfont
\hspace*{.60em} I thank Yumi Karlsson for her 
assistance in LaTEXing several earlier drafts   
of this\linebreak report. 
The first of these dates back to 2010, while I was 
still active at Uppsala University in Sweden.  More 
recently, 
I am grateful to the Institute for Advanced \vspace*{.025ex} Study 
in Princeton for 
providing me with excellent working conditions and a supplemental 
research grant (from the Bell Companies) during the Spring Term 
of 2012, when work on the current, improved, version 
was being done.  

}

\vspace*{4.0ex}
\small

\noindent
\textsc{School of Mathematics}\\[-.31ex] 
\textsc{University of Minnesota}\\[-.25ex] 
\textsc{Minneapolis, Mn. 55455 USA}\\[+1.0ex] 
E-mail address: $\MSC\MSC$\textbf{hejhal@math.umn.edu}

\normalsize

\end{mainmatter}

\end{document}